\documentclass[english]{article}
\usepackage[T1]{fontenc}
\usepackage[latin9]{inputenc}
\usepackage[active]{srcltx}
\usepackage{mathtools}
\usepackage{amsmath}
\usepackage{amsthm}
\usepackage{amssymb}
\usepackage{graphicx}
\usepackage{geometry}
\makeatletter

\providecommand{\tabularnewline}{\\}

\numberwithin{equation}{section}
\numberwithin{figure}{section}
\theoremstyle{plain}
\newtheorem{thm}{\protect\theoremname}
\theoremstyle{plain}
\newtheorem{lem}[thm]{\protect\lemmaname}
\theoremstyle{definition}
\newtheorem{defn}[thm]{\protect\definitionname}
\theoremstyle{plain}
\newtheorem{prop}[thm]{\protect\propositionname}
\theoremstyle{remark}
\newtheorem{rem}[thm]{\protect\remarkname}
\theoremstyle{plain}
\newtheorem{cor}[thm]{\protect\corollaryname}

\@ifundefined{date}{}{\date{}}
\usepackage{color}
\usepackage{hyperref}
\usepackage{comment}

\newtheorem{assump}{Assumption}

\makeatother

\usepackage{babel}
\providecommand{\corollaryname}{Corollary}
\providecommand{\definitionname}{Definition}
\providecommand{\lemmaname}{Lemma}
\providecommand{\propositionname}{Proposition}
\providecommand{\remarkname}{Remark}
\providecommand{\theoremname}{Theorem}

\begin{document}
\global\long\def\ve{\varepsilon}%
\global\long\def\R{\mathbb{R}}%
\global\long\def\Rn{\mathbb{R}^{n}}%
\global\long\def\Rd{\mathbb{R}^{d}}%
\global\long\def\E{\mathbb{E}}%
\global\long\def\P{\mathbb{P}}%
\global\long\def\bx{\mathbf{x}}%
\global\long\def\vp{\varphi}%
\global\long\def\ra{\rightarrow}%
\global\long\def\smooth{C^{\infty}}%
\global\long\def\Tr{\mathrm{Tr}}%
\global\long\def\bra#1{\left\langle #1\right|}%
\global\long\def\ket#1{\left|#1\right\rangle }%
\global\long\def\Re{\mathrm{Re}}%
\global\long\def\Im{\mathrm{Im}}%
\global\long\def\bsig{\boldsymbol{\sigma}}%
\global\long\def\btau{\boldsymbol{\tau}}%
\global\long\def\bmu{\boldsymbol{\mu}}%
\global\long\def\bx{\boldsymbol{x}}%
\global\long\def\bups{\boldsymbol{\upsilon}}%
\global\long\def\bSig{\boldsymbol{\Sigma}}%
\global\long\def\bt{\boldsymbol{t}}%
\global\long\def\bs{\boldsymbol{s}}%
\global\long\def\by{\boldsymbol{y}}%
\global\long\def\brho{\boldsymbol{\rho}}%
\global\long\def\ba{\boldsymbol{a}}%
\global\long\def\bb{\boldsymbol{b}}%
\global\long\def\bz{\boldsymbol{z}}%
\global\long\def\bc{\boldsymbol{c}}%
\global\long\def\balpha{\boldsymbol{\alpha}}%
\global\long\def\bbeta{\boldsymbol{\beta}}%
\global\long\def\T{\mathrm{T}}%
\global\long\def\trip{\vert\!\vert\!\vert}%
\global\long\def\lrtrip#1{\left|\!\left|\!\left|#1\right|\!\right|\!\right|}%
\global\long\def\eps{\epsilon}%
\global\long\def\blam{\boldsymbol{\lambda}}%

\title{Toward optimal-scaling DFT: \\
Stochastic Hartree theory in the thermodynamic and complete basis
set limits at arbitrary temperature}
\author{Yuhang Cai\thanks{University of California, Berkeley} \and Michael
Lindsey\footnotemark[1] \thanks{Lawrence Berkeley National Laboratory}}
\maketitle
\begin{abstract}
We present the first mathematical analysis of stochastic density functional
theory (DFT) in the context of the Hartree approximation. We motivate
our analysis via the notion of nearly-optimal or $\tilde{O}(n)$ scaling
with respect to the number $n$ of computational degrees of freedom,
independent of the number of electrons, in both the thermodynamic
and complete basis set limits. Indeed, the promise of such scaling
is the primary motivation for stochastic DFT relative to conventional
orbital-based approaches, as well as deterministic orbital-free alternatives.
We highlight three key targets for mathematical attention, which are
synthesized in our algorithm and analysis. First, we identify a particular
stochastic estimator for the Hartree potential whose sample complexity
is essentially independent of the discretization size. Second, we
reformulate the self-consistent field iteration as a stochastic mirror
descent method where the Fermi-Dirac entropy plays the role of the
Bregman potential, and we prove a nearly discretization-independent
bound on the number of iterations needed to reach fixed accuracy.
Third, motivated by the estimator, we introduce a novel pole expansion
scheme for the square-root Fermi-Dirac operator, preserving $\tilde{O}(n)$
cost per mirror descent iteration even in the complete basis set limit.
Combining these ingredients, we establish nearly-optimal scaling in
both limits of interest under reasonable assumptions on the basis
sets chosen for discretization. Extensive numerical experiments on
problems with as many as $10^{6}$ degrees of freedom validate our
algorithm and support the theory of nearly-optimal scaling.
\end{abstract}

\section{Introduction \label{sec:Introduction}}

Density functional theory (DFT) \cite{KohnSham} is the most widely
used computational tool in electronic structure theory, with far-reaching
applications across quantum chemistry and materials science. The central
idea of DFT is to solve the many-electron problem by way of an effective
single-electron problem, whose Hamiltonian is determined self-consistently
in terms of the electron density.

This work is a mathematical study of a framework for solving DFT known
as stochastic DFT \cite{PhysRevLett.111.106402,cytterStochasticDensityFunctional2018,fabianStochasticDensityFunctional2019},
within which we consider certain modifications and extensions. First
we review the motivation for this approach, through the lens of a
suitable notion of \emph{optimal scaling} on which we shall elaborate.
This notion will motivate and guide our algorithms and analysis.

Conventional approaches to density functional theory at zero temperature
rely on finding self-consistent solutions to the Kohn-Sham equations
\[
\begin{aligned} & H_{\mathrm{eff}}[\rho]\,\psi_{j}=\ve_{j}\psi_{j},\quad j=1,\ldots,N,\\
 & \rho(x)=\sum_{j=1}^{N}\vert\psi_{j}(x)\vert^{2}.
\end{aligned}
\]
 Here the Kohn-Sham orbitals $\psi_{j}$, $j=1,\ldots,N$, are the
lowest $N$ orthonormal eigenfunctions of the effective single-particle
Hamiltonian 
\[
H_{\mathrm{eff}}[\rho]:=-\frac{1}{2}\Delta+v_{\mathrm{ext}}+v_{\mathrm{hxc}}[\rho],
\]
 which consists of terms corresponding to the kinetic energy, the
fixed external potential, and the Hartree-exchange-correlation potential,
respectively. The last of these must be determined self-consistently
via the electron density $\rho$. Moreover, note that $N$ indicates
the number of electrons, which is fixed \emph{a priori} in such a
zero-temperature setting.

The Kohn-Sham equations must be solved within some kind of framework
of discretization. In spite of the diversity of approaches (based,
for example, on finite differences, finite elements, or specialized
and highly successful quantum chemistry basis sets \cite{szabo1996modern}),
roughly speaking one must always incur a cost of at least $O(N^{2}n)$
where $n$ is the number of grid points or basis functions, simply
due to the cost of orthogonalizing the Kohn-Sham orbitals.

As such there is an interest in orbital-free methods that avoid any
explicit dependence on $N$ by working with the density matrix 
\[
P(x,y)=\sum_{j=1}^{N}\psi_{j}(x)\psi_{j}(y),
\]
 from which the electron density can be recovered as the diagonal.
These approaches are also more naturally formulated at finite temperature,
which enhances their interest from the point of view, e.g., of warm
dense matter (WDM) \cite{whiteFastUniversalKohnSham2020}. The zero-temperature
limit can be recovered suitably.

The key idea of such methods is that we can recast the Kohn-Sham equations
in the form 
\begin{align*}
P & =f_{\beta}\left(H_{\mathrm{eff}}[\rho]-\mu\mathbf{I}\right)\\
\rho & =\mathrm{diag}\left[P\right],
\end{align*}
 where 
\[
f_{\beta}(x)=\frac{1}{1+e^{\beta x}}
\]
 is the Fermi-Dirac occupation function at inverse temperature $\beta\in(0,\infty)$,
applied in the sense of the continuous operator calculus, and $\mu\in\R$
is a chemical potential.

Recently, PEXSI (pole expansion and selected inversion) \cite{lin2009pole,CMS2009,JCPM2013}
has gained popularity as one such orbital-free approach, which, unlike
stochastic DFT, is completely deterministic. (We will review the longer
history of stochastic orbital-free methods shortly.) The main idea
of PEXSI is that the Fermi-Dirac function can be approximated accurately
as a sum of only a modest number of poles, and in turn the electron
density can be recovered as the sum of diagonals of inverses of sparse
matrices. The selected inversion component of PEXSI achieves this
latter step with a direct algorithm while avoiding the formation of
the entire inverse matrix. However, the cost is controlled by a suitable
fill-in pattern and scales as $O(n^{2})$ for grid-based discretizations
of general 3D systems. Still, the cost scaling in practice is quite
compelling in many practical scenarios, especially for quasi-1D and
quasi-2D molecules.

The hope of stochastic DFT is to achieve a method with truly $\tilde{O}(n)$
cost, independent of the electron number, where the tilde indicates
the omission of logarithmic factors. This defines a target notion
of \emph{optimal scaling}, as the scaling of simply storing the
electron density on a grid of $n$ points is $O(n)$. With a view
toward achieving this scaling, stochastic DFT constructs an estimator
for the electron density on a grid using a stochastic trace estimator.
It remains to elucidate additional factors encoding dependence on
the target error $\ve$ and the inverse temperature $\beta$, which
we will explain below.

The roots of stochastic DFT go back at least to 1990s \cite{PhysRevLett.70.3631},
and more recent efforts \cite{PhysRevLett.111.106402,cytterStochasticDensityFunctional2018,fabianStochasticDensityFunctional2019,whiteFastUniversalKohnSham2020,liuPlaneWaveBasedStochasticDeterministicDensity2022}
have advanced stochastic DFT as a practical algorithm for \emph{ab
initio} electronic structure problems. However, to our knowledge,
stochastic DFT has not received any attention from the point of view
of mathematical analysis.

In our view, there are several key components deserving of mathematical
attention. 
\begin{enumerate}
\item The \textbf{stochastic trace estimator} used for estimating the effective
potential. Although several different variations have been tried in
the literature (contrast, e.g., \cite{cytterStochasticDensityFunctional2018}
and \cite{fabianStochasticDensityFunctional2019}), we demonstrate
that one choice in particular admits a nearly-dimension-independent
sample complexity, both for grid-based and arbitrary Galerkin discretizations.
We comment that several more advanced approaches have been introduced
and applied to reduce the variance of the stochastic estimator. These
include the embedded-fragments theory \cite{fabianStochasticDensityFunctional2019}
which involves a decomposition of the computational domain into local
fragments, as well as approaches that mix deterministic and stochastic
Kohn-Sham orbitals \cite{whiteFastUniversalKohnSham2020,liuPlaneWaveBasedStochasticDeterministicDensity2022}.
While these approaches fall outside the scope dictated by our pursuit
of optimal scaling, they can offer significant practical speedups
and offer interesting avenues for further theoretical investigation.
\item The \textbf{optimization} algorithm for solving the self-consistency
criterion. Typically, stochastic DFT is solved via an approach resembling
the traditional self-consistent field (SCF) iteration, where a stochastic
estimator for the effective potential is dropped in as a replacement
for the exact effective potential, and special attention is paid to
the chemical potential adjustment \cite{cytterStochasticDensityFunctional2018}.
We adopt a novel point of view on SCF through the lens of \emph{mirror
descent}, which we shall review below. In short, this perspective
allows us to prove a nearly-dimension-independent bound on the number
of iterations required to achieve fixed accuracy. We remark that this
perspective may also shed light on deterministic SCF, where to our
knowledge no analogous dimension-independent results on the convergence
rate are known.
\item The \textbf{matrix function approximation} used within the trace estimator.
We shall see that the aforementioned `good' stochastic trace estimator
requires a fast algorithm for matrix-vector products by $f_{\beta}^{1/2}(H_{\mathrm{eff}})$,
where $f_{\beta}^{1/2}$ denotes the square root of the Fermi-Dirac
function. Stochastic DFT typically approaches this task via Chebyshev
expansion of $f_{\beta}^{1/2}$ over the spectral range of $H_{\mathrm{eff}}$
\cite{cytterStochasticDensityFunctional2018}. However, as the basis
set is refined (per unit volume), i.e., in the \emph{complete basis
set limit}, the number of terms in the expansion grows with $n$
due to the fact that the Laplacian term in the effective Hamiltonian
is unbounded. As such, we introduce an alternative approach based
on \emph{pole expansion}, inspired by PEXSI. We see that the same
contour integration technique as used in PEXSI can be applied to the
\emph{square-root} Fermi-Dirac function, after making a suitable
choice of branch cut.
\end{enumerate}
Notably, our theory only directly concerns the \emph{Hartree approximation},
in which $v_{\mathrm{hxc}}=v_{\mathrm{h}}$ and the exchange and correlation
components of $v_{\mathrm{hxc}}$ are neglected. From a chemical point
of view, this is not considered an effective approximation in electronic
structure theory. However, from an algorithmic point of view, we believe
that our analysis is significant, because the Hartree contribution
is typically dominant in absolute terms, and the algorithm itself
generalizes easily. We offer a more detailed discussion of implications
for more general DFT in Appendix \ref{app:dft}.

Concretely, the Hartree energy enables rigorous analysis because it
is:
\begin{enumerate}
\item \textbf{Convex}. More general exchange-correlation functionals are
not convex and can even admit many local optima, or self-consistent
solutions of the Kohn-Sham equations, complicating the notion of `the'
physically correct solution.
\item \textbf{Quadratic}. The fact that the Hartree energy is quadratic
will imply that our estimator for the Hartree potential is unbiased.
Possibly, the impact of bias due to exchange-correlation contributions
may be less important due to their smaller magnitude, but we leave
further analysis of this possibility for future work.
\end{enumerate}
In our analysis of the Hartree theory, we justify the somewhat surprising
stylized conclusion that \emph{solving the self-consistent Hartree
theory from scratch is about as easy as estimating the Hartree energy
for fixed $\rho$}. Under the hood, the idea supporting this conclusion
is that the convergence rate of stochastic mirror descent (in which
a single-shot estimator is used in each optimization step) balances
perfectly with the slow Monte Carlo error rate for the stochastic
estimator. To establish these claims, we need an analysis of stochastic
mirror descent that explains the self-averaging in the effective potential
that takes place over many iterations. Later in the introduction,
we will review the broader context of mirror descent.

Before offering a more refined summary of our scaling results in terms
of $\beta$ and $\ve$, we must distinguish two qualitatively different
limits in which the basis set size $n$ tends to infinity: 
\begin{enumerate}
\item The\textbf{ thermodynamic limit}. This limit is inspired by a scenario
in which a basis set of uniform resolution is used to discretize an
expanding volume, which in particular covers the case where a quantum
chemistry basis set \cite{szabo1996modern} of fixed accuracy is applied
to an enlarging molecule. In particular, we view the discretization
of the non-interacting Hamiltonian (including the Laplacian term)
as bounded in this limit. Moreover, it is reasonable to assume (via
Lemma \ref{lem:Gbound}) that a sort of gradient bound holds in this
limit, as we argue via physical considerations in Section \ref{sec:Structure-of-the-energy}.
Finally, the strong convexity parameter for the fermionic entropy
(Lemma \ref{lem:stronglyconvex}) can be viewed as proportional to
the volume in this limit, which justifies the notion of relative energy
error appearing in the main convergence theorem (Theorem \ref{thm:convergence}). 
\item The \textbf{complete basis set limit}. In this limit, we consider
a basis set of increasingly fine resolution over a fixed volume. In
this limit, as we have mentioned above, the pole expansion approach
to discretizing the square-root Fermi-Dirac function becomes necessary
to maintain optimal scaling. But, in addition, the analysis bounding
the number of optimization iterations must also be adapted to this
case. To avoid confining ourselves to an unnecessarily concrete choice
of basis, we make certain assumptions on the eigenvalue growth of
the non-interacting Hamiltonian, which are intuitive due to the Laplacian
term, and prove that this growth essentially confines the optimizer
to a sector of constant dimension in which suitable strong convexity
holds. Another subtle point is that an alternative initialization
must be chosen for the mirror descent algorithm to avoid dependence
on the spectral norm of the Laplacian. The main convergence results
are given in Theorem \ref{thm:convergence-1} and Corollary \ref{cor:convergence}.
(Note that a suitable notion of relative energy error must be defined.)
\end{enumerate}
Throughout, we attempt to keep the analysis as general as possible
in terms of the basis set, preferring to identify physically reasonable
assumptions, rather than fixate on establishing these assumptions
for any particular choice of basis set.

Now we turn to a more refined summary of the scaling in terms of the
error tolerance $\ve$ and the inverse temperature $\beta$. The additional
factor due to error dependence is always $\tilde{O}(\ve^{-2})$, but
the factor due to temperature dependence is different in the two limits.
Note also that our notion of energy error is a relative notion that
is explained in further detail in the two main theorems (Theorems
\ref{thm:convergence} and \ref{thm:convergence-1}). Then in the
thermodynamic limit, there is no additional dependence of the number
of iterations on $\beta$. However, in the complete basis set limit,
there is an additional factor of $\tilde{O}(\sqrt{\beta})$. Thus
in the two limits, the overall scalings of the number of iterations
are $\tilde{O}(\ve^{-2})$ and $\tilde{O}(\sqrt{\beta}\,\ve^{-2})$,
respectively.

Next we turn to a discussion of the cost scaling of each iteration.
For a general basis set of size $n$, we must also introduce an auxiliary
grid of size $m$ to spatially resolve the electron density $\rho$.
We relate the grid to the basis set via a suitable factorization of
the two-electron integrals known as tensor hypercontraction \cite{10.1063/1.4732310,10.1063/1.4768233,10.1063/1.4768241}
and its interpretation as an interpolative separable density fitting
\cite{LU2015329}, cf. Section \ref{subsec:Electron-repulsion-integrals}.
For the purpose of this summary, we have in mind the special case
of the periodic sinc basis set used in our numerical experiments,
for which $m=n$, and for which there is a direct correspondence between
basis functions and grid points.

Now the dominant cost of each stochastic mirror descent iteration
is the cost of a single matrix-vector multiplication by $f_{\beta}^{1/2}(H_{\mathrm{eff}})$.
For the purposes of our convergence results, we neglect the error
$\delta$ of this approximation for simplicity. However, note that
in the thermodynamic limit in which the effective Hamiltonian remains
bounded, standard approximation theory results \cite{trefethen2019approximation}
guarantee that Chebyshev approximation suffices to achieve $\tilde{O}(\sqrt{\beta}\,\log(1/\delta)\,n)$
complexity for this task. In the complete basis set limit, heuristically
one expects that a similar scaling is achieved via the pole expansion
technique (as we confirm numerically in our experiments), but a rigorous
analysis is stymied by the difficulty of analyzing preconditioned
linear solvers for indefinite systems. Nevertheless, we comment that
by reducing to the positive definite case (at the cost of squaring
the condition number), we could use the analysis of the preconditioned
conjugate gradient algorithm to achieve $\tilde{O}(\beta\,\log(1/\delta)\,n)$
complexity in the thermodynamic limit, which is still enough to achieve
the optimal scaling in terms of $n$.

For most of our results, we assume a fixed chemical potential for
simplicity, though in Section \ref{sec:Chemical-potential-optimization}
we explain how the algorithm with fixed chemical potential can be
wrapped within another algorithmic layer to determine the chemical
potential, at the cost of an additional factor of $O(\ve^{-1})$.
It may be interesting to consider alternative primal-dual approaches
for optimizing the Hartree potential and the chemical potential at
the same time, though we expect that such an approach may introduce
additional dependence on $\beta$ in the analysis, and we leave such
considerations to future work.

Finally, we validate our theoretical results through detailed numerical
experiments in Section \ref{sec:Experiments}, which demonstrate the
scalability of our pole-expansion-based algorithm in both limits to
discretizations using over a million points.

Before offering an outline of the paper, we close the introduction
with a review of mirror descent and how our work fits into this context.
Mirror descent (MD) is a versatile first-order optimization method
originally proposed by \cite{Nemirovski1983problem} for convex problems.
It has since become a cornerstone in both online learning and stochastic
optimization due to its ability to adapt to high-dimensional settings
and different norm geometries. The algorithm can be viewed as being
induced by the choice of a convex `Bregman potential,' which induces
a mirror map to a dual domain as its gradient. Many practical applications
of mirror descent take place on the probability simplex and, to a
lesser extent, its quantum generalization (the set of density operators).
To our knowledge, our work is the first exploration of a practical
scenario in which mirror descent is applied to the domain of \emph{fermionic
density matrices}. The suitable choice of Bregman potential in our
context is the Fermi-Dirac operator entropy, and interestingly we
establish a connection between mirror descent in this framework and
the SCF that is commonly applied to solve DFT.

More specifically, we must consider a suitable notion of stochastic
mirror descent, due to the use of a stochastic estimator for the gradient,
and exploit self-averaging over the optimization trajectory, as mentioned
above. This task requires us to establish a suitable sub-exponential
concentration bound for our gradient estimator, which can be lifted
to a suitable concentration bound for a martingale difference sequence
appearing in the error analysis of mirror descent. Although we offer
our own complete proof to meet the specifics of our needs, we note
that the lifting step finds several analogies in the literature on
stochastic mirror descent. Indeed, \cite{nemirovskiRobustStochasticApproximation2009,lanOptimalMethodStochastic2012}
establishes high-probability bounds for stochastic mirror descent
(SMD) and accelerated stochastic mirror descent with i.i.d unbiased
gradient estimators and sub-Gaussian noise; \cite{liuHighProbabilityConvergence2023}
considers a more general framework of SMD with unbounded domains;
\cite{vuralMirrorDescentStrikes2022} establishes a convergence rate
for SMD with infinite noise variance; \cite{li2022high} considers
stochastic gradient descent with heavy-tailed noise using the notion
of sub-Weibull distributions; and \cite{eldowaGeneralTailBounds2024}
further uses this class of distributions to establish convergence
rates for SMD with heavy-tailed martingale noise.

\subsection{Outline}

In Section \ref{sec:Preliminaries}, we present the formulation of
the Hartree theory as a convex optimization problem over fermionic
density matrices, as well as the physical interpretation of the terms
appearing in the objective. We also explain the assumptions on these
terms needed to yield a suitable notion of gradient bound, to be justified
later in Section \ref{sec:Structure-of-the-energy}. In Section \ref{sec:Mirror-descent-framework},
we present the mirror descent framework for solving this optimization
problem, including some analysis of the strong convexity of the Bregman
potential. In Section \ref{sec:Gradient-estimator}, we complete the
description of the algorithm by presenting the gradient estimator
for an arbitrary Galerkin basis and proving a suitable concentration
bound. We also explain the matrix function approximation details required
to implement the gradient estimator practically. In Section \ref{sec:Convergence},
we present a convergence theorem for our algorithm which establishes
optimal scaling in the thermodynamic limit. In Section \ref{sec:Structure-of-the-energy},
we return to discuss the structure of the energy term and the physical
considerations justifying our abstract assumptions. This completes
the `main narrative' for the optimal scaling of our approach in the
thermodynamic limit.

In the remaining sections, we consider several ornamentations. To
wit, in Section \ref{sec:Chemical-potential-optimization} we explain
how the algorithm can be used as a subroutine within an algorithm
for determining the chemical potential if the electron number (not
the chemical potential) is fixed \emph{a priori}. In Section \ref{sec:cbl},
we consider the complete basis set limit and explain how the dependence
on the spectral norm of the non-interacting Hamiltonian can be removed
by making assumptions on its eigenvalue growth and taking a suitable
initialization.

In Section \ref{sec:Experiments}, we present numerical experiments
validating our theoretical results and, moreover, demonstrating the
almost-linear scaling of the algorithm (including the novel pole expansion
approach) in both the thermodynamic and complete basis set limits
for model problems. For concreteness, we comment that each of our
largest experiments, which considered grids / basis sets of size $n\approx10^{6}$,
consumed only about 8 hours on a single GPU.

Finally, we outline the content of several supporting appendices which
are referenced as needed in the main text. In Appendix \ref{app:elem},
we prove several facts about the Fermi-Dirac entropy. In Appendix
\ref{app:subexp}, we prove the sub-exponential concentration bound
on the trace estimator. In Appendix \ref{app:contour} we give background,
details, and theory supporting the construction of the pole expansion
for the square-root Fermi-Dirac function via contour integration.
In Appendix \ref{app:dft}, we review considerations for more general
DFT beyond the Hartree theory and comment on the relevance of our
results in this broader setting. In Appendix \ref{app:chemical} we
give proofs supporting the main results of Section \ref{sec:Chemical-potential-optimization}
on chemical potential optimization. Finally, in Appendix \ref{app:cbl}
we give proofs supporting the main results of Section \ref{sec:cbl}
on the convergence theory in the complete basis set limit.

\subsection{Acknowledgments}

This material is based on work supported by the Applied Mathematics
Program of the US Department of Energy (DOE) Office of Advanced Scientific
Computing Research under contract number DE-AC02-05CH11231 and by
the U.S. Department of Energy, Office of Science, Accelerated Research
in Quantum Computing Centers, Quantum Utility through Advanced Computational
Quantum Algorithms, grant no. DE-SC0025572. M.L. was partially supported
by a Sloan Research Fellowship. The authors gratefully acknowledge
conversations with Xiao Liu and Ben Shpiro on stochastic DFT.

\section{Preliminaries \label{sec:Preliminaries}}

We are interested in an optimization problem of the form 
\begin{align}
\underset{X\in\R^{n\times n}}{\text{minimize}}\quad\quad & F_{\beta}(X)-\mu\,\Tr[X]\label{eq:opt}\\
\text{subject to}\quad\quad & 0\preceq X\preceq\mathbf{I}_{n}.\nonumber 
\end{align}
 We can view this problem as unconstrained because $F_{\beta}$ will
act as a barrier for the domain 
\[
\mathcal{X}:=\{X\,:\,0\preceq X\preceq\mathbf{I}_{n}\}.
\]
 Here $\beta\in(0,+\infty]$ defines the \textbf{\emph{inverse temperature}},
which can be taken to be $+\infty$ by suitably interpreting expressions
below, unless indicated otherwise. Meanwhile, $\mu$ is called the
\textbf{\emph{chemical potential}} and can be viewed as a Lagrange
multiplier for the trace constraint in the following alternative problem:
\begin{align}
\underset{X\in\R^{n\times n}}{\text{minimize}}\quad\quad & F_{\beta}(X)\label{eq:opt_constraint}\\
\text{subject to}\quad\quad & \Tr[X]=N,\nonumber \\
 & 0\preceq X\preceq\mathbf{I}_{n}.\nonumber 
\end{align}
 We will discuss later how to solve this constrained problem, given
an oracle for solving the unconstrained problem.

For context, we remark that in both problems, $X=(X_{ij})_{i,j=1}^{n}$
denotes the matrix of coefficients of the \textbf{\emph{density matrix}}
\[
P_{X}(x,y):=\sum_{i,j=1}^{n}X_{ij}\,\psi_{i}(x)\psi_{j}(y)
\]
 in a fixed orthonormal \textbf{\emph{quantum chemistry basis}} $\{\psi_{i}\}_{i=1}^{n}$
of functions on $\R^{d}$. (For concreteness we can take $d=3$, though
the choice does not affect most of the ensuing considerations directly.)

In terms of $X$ we may also define the \textbf{\emph{electron density
}}
\[
\rho_{X}(x)=P_{X}(x,x).
\]
 Observe that by orthonormality 
\[
\int_{\R^{d}}\rho_{X}(x)\,dx=\Tr[X].
\]
 Since the integral of the electron density can be interpreted as
the total electron number of the ensemble, the value $N\in[0,n]$
denotes the \textbf{\emph{electron number}}, which need not be an
integer.

We will view $\{\psi_{i}\}$ as a quantum chemistry basis for an extended
quantum system. Intuitively, we can imagine that the number of basis
functions per atom as being bounded by a constant, and we are considering
a system of increasing volume / number of atoms in the limit $n\ra\infty$.
For now, we will maintain an abstract perspective that avoids discussion
of the physical modeling considerations underlying our optimization
problems, but we will return to these considerations below in Section
\ref{sec:Structure-of-the-energy}. We will further comment on general
DFT, beyond the Hartree approximation, in Appendix \ref{app:dft}.

The objective $F_{\beta}$ is interpreted as a \textbf{\emph{free
energy }}defined by 

\[
F_{\beta}(X):=E(X)+\frac{1}{\beta}S_{\mathrm{FD}}(X),
\]
 which, informally speaking, is bounded within a constant factor of
the free energy density per unit volume.

Within the definition for the free energy, $E(X)$ denotes the \textbf{\emph{energy}}
function and $S_{\mathrm{FD}}(X)$ denotes the \textbf{\emph{Fermi-Dirac
entropy}}: 
\[
S_{\mathrm{FD}}(X)=\Tr[X\log X]+\Tr[(\mathbf{I}_{n}-X)\log(\mathbf{I}_{n}-X)].
\]
 We can view the domain of the convex function $S_{\mathrm{FD}}$
as 
\[
\mathcal{X}:=\{X\in\R^{n\times n}\,:\,0\preceq X\preceq\mathbf{I}_{n}\},
\]
 and the Fermi-Dirac entropy acts as a soft barrier for this domain.

Finally, the energy function $E(X)$ will take the form 
\[
E(X)=\Tr[CX]+\tilde{E}(X),
\]
 where $C\in\R^{n\times n}$ is symmetric (and the notation is chosen
by analogy to the usual notation for the linear cost term in semidefinite
programming) and where 
\[
\tilde{E}(X)=\frac{1}{2}\sum_{ijkl}X_{ij}\,v_{ij,kl}\,X_{kl}
\]
 is the \textbf{\emph{Hartree energy}}. The tensor $v_{ij,kl}$ will
consist of the electron repulsion integrals (ERI).

We assume that 
\begin{equation}
v_{ij,kl}=\sum_{p,q=1}^{m}\Psi_{pi}\Psi_{pj}\,V_{pq}\,\Psi_{qk}\Psi_{ql},\label{eq:ERI0}
\end{equation}
 where $\Psi=(\Psi_{pi})\in\R^{m\times n}$ has orthonormal columns
(i.e., satisfies $\Psi^{\top}\Psi=\mathbf{I}_{n}$) and $V=(V_{pq})\in\R^{m\times m}$
is \emph{positive semidefinite}. In particular it follows that $E(X)$
is \emph{convex}.

In Section \ref{sec:Structure-of-the-energy}, we will explain the
physical background and modeling conditions that support this structure
for the energy function.

The key quantities appearing in our error analysis will be the spectral
norm $\Vert C\Vert$ and the induced operator norm $\Vert V\Vert_{\infty}=\max_{i}\sum_{j}\vert V_{ij}\vert$,
as well as the quantity 
\begin{equation}
c_{\Psi}:=\max_{p=1,\ldots,m}\sum_{i=1}^{n}\vert\Psi_{pi}\vert^{2}.\label{eq:c}
\end{equation}
 We will also explain how these quantities are controlled in terms
of the underlying physical problem in Section \ref{subsec:Electron-repulsion-integrals}.
In fact $c_{\Psi}$ and $\Vert V\Vert_{\infty}$ appear together via
their product, which we give its own notation: 
\begin{equation}
c_{\mathrm{h}}:=c_{\Psi}\Vert V\Vert_{\infty},\label{eq:ch}
\end{equation}
 in which `h' is for Hartree. To tackle the complete basis set limit,
we will have to alter our perspective somewhat, and the relevant discussion
is deferred to Section \ref{sec:cbl}.

For now, it is useful to observe that the gradient of the quadratic
term can be written 
\[
[\nabla\tilde{E}(X)]_{ij}=\sum_{kl}v_{ij,kl}\,X_{kl}.
\]
 Moreover, it is useful to define $\rho(X)=\left[\rho_{q}(X)\right]_{q=1}^{m}\in\R^{m}$
by 
\begin{equation}
\rho_{q}(X)=\sum_{k,l=1}^{n}\Psi_{qk}X_{kl}\Psi_{ql},\label{eq:rhoq}
\end{equation}
 or alternatively as 
\begin{equation}
\rho(X)=\mathrm{diag}\left[\Psi X\Psi^{\top}\right].\label{eq:rhodiag}
\end{equation}
 As we shall elucidate in Section \ref{subsec:Electron-repulsion-integrals},
the vector $\rho_{q}(X)$ can be interpreted as the electron density
$\rho_{X}(x_{q})$ evaluated at a collection of interpolation points
$\{x_{q}\}_{q=1}^{m}$.

In terms of $\rho(X)$, we can alternatively write $\tilde{E}(X)$
as a quadratic form: 
\[
\tilde{E}(X)=\frac{1}{2}\rho(X)^{\top}V\rho(X),
\]
 and we can also express $\nabla\tilde{E}(X)$ in terms of $\rho(X)$
as 
\begin{equation}
\nabla\tilde{E}(X)=\Psi^{\top}\mathrm{diag}^{*}\left[V\rho(X)\right]\Psi.\label{eq:Etilderho}
\end{equation}
 Here `$\mathrm{diag}^{*}$' is the operator that returns a diagonal
matrix with specified diagonal, which is the formal adjoint of the
operator that returns the diagonal of a square matrix. We can identify
$\nabla\tilde{E}(X)$ physically as the \textbf{\emph{Hartree potential}}.

The importance of $c_{\mathrm{h}}$ owes to the following elementary
lemma, which bounds the spectral norm of the gradient $\nabla\tilde{E}(X)$
(hence the notation `G' for gradient). Together with a bound on $\Vert C\Vert$,
this offers a bound on the spectral norm of $\nabla E(X)=C+\nabla\tilde{E}(X)$.
\begin{lem}
\label{lem:Gbound} For any $Y\in\R^{n\times n}$, $\Vert\nabla\tilde{E}[Y]\Vert\leq c_{\mathrm{h}}\Vert Y\Vert$,
and $\Vert\rho(Y)\Vert_{\infty}\leq c_{\Psi}\Vert Y\Vert$. 
\end{lem}

\begin{proof}
Note that by (\ref{eq:Etilderho}), since $\Psi$ is an isometry it
suffices to show that $\Vert V\rho(Y)\Vert_{\infty}\leq c_{\Psi}\Vert V\Vert_{\infty}\Vert Y\Vert$.
In turn it suffices to show that $\Vert\rho(Y)\Vert_{\infty}\leq c_{\Psi}\Vert Y\Vert$.

To see this, note that by (\ref{eq:rhoq}), we have 
\[
\rho_{q}(Y)=\left\langle \Psi_{q,\,:\,},Y\Psi_{q,\,:\,}\right\rangle \leq\Vert\Psi_{q,\,:\,}\Vert\,\Vert Y\Psi_{q,\,:\,}\Vert\leq c_{\Psi}\Vert Y\Vert,
\]
 where we have used the definition (\ref{eq:c}) of $c_{\Psi}$. This
completes the proof.
\end{proof}

\section{Mirror descent framework \label{sec:Mirror-descent-framework}}

Since (\ref{eq:opt}) is a continuous optimization problem on a compact
domain, we know that there exists a minimizer $X_{\star}$. (Moreover
the minimizer is unique when $\beta<+\infty$ by strict convexity.)

We will solve the unconstrained problem (\ref{eq:opt}) using mirror
descent, induced by the choice of $S_{\mathrm{FD}}$ as our \textbf{\emph{Bregman
potential}} on $\mathcal{X}$. See \cite{beck2003mirror} for a complete
background on mirror descent. Here we highlight the structure of a
few of the key objects that arises from this choice of Bregman potential.

First, the \textbf{\emph{mirror map}} $\nabla S_{\mathrm{FD}}$ defined
by our choice is given by 
\[
\nabla S_{\mathrm{FD}}(X)=\log\left(X(\mathbf{I}_{n}-X)^{-1}\right),
\]
 and the convex conjugate $S_{\mathrm{FD}}^{*}$ of the Bregman potential
is defined by 
\[
S_{\mathrm{FD}}^{*}(X^{*})=\Tr\left[\log(\mathbf{I}_{n}+e^{X^{*}})\right].
\]
 The domain $\mathcal{X}^{*}$ of $S_{\mathrm{FD}}^{*}$ is the set
of all symmetric $n\times n$ matrices. The gradient of the convex
conjugate $S_{\mathrm{FD}}^{*}$ defines the \textbf{\emph{inverse
mirror map}}:
\[
\nabla S_{\mathrm{FD}}^{*}(X^{*})=(\mathbf{I}_{n}+e^{-X^{*}})^{-1}.
\]
The Bregman potential also fixes a \textbf{\emph{Bregman divergence}}: 

\[
D(Y\Vert X):=S_{\mathrm{FD}}(Y)-S_{\mathrm{FD}}(X)-\left\langle \nabla S_{\mathrm{FD}}(X),Y-X\right\rangle 
\]
 for $X,Y\in\mathcal{X}$. (Here and throughout the manuscript, $\left\langle \,\cdot\,,\,\cdot\,\right\rangle $
refers to the Frobenius inner product.)

\subsection{Algorithm overview \label{subsec:Algorithm-overview}}

The mirror descent iteration ((2.7) in \cite{beck2003mirror}) (with
step size $\gamma_{t}>0$ at iteration $t$) for (\ref{eq:opt}) is
defined by 
\[
\nabla S_{\mathrm{FD}}(X_{t+1})=(1-\beta^{-1}\gamma_{t})\nabla S_{\mathrm{FD}}(X_{t})-\gamma_{t}\left(\nabla E(X_{t})-\mu\mathbf{I}_{n}\right).
\]
 We will offer a more physical interpretation in Section \ref{subsec:Physical-interpretation}
below. For reasons to be clarified in Section \ref{subsec:proximal},
due to an alternative interpretation of this algorithm, we will always
take $\gamma_{t}\leq\beta$.

As we shall see later, computational efficiency demands that we estimate
$\nabla E$ stochastically. For this reason, we are interested more
generally in the approximate update 
\begin{equation}
\nabla S_{\mathrm{FD}}(X_{t+1})=(1-\beta^{-1}\gamma_{t})\nabla S_{\mathrm{FD}}(X_{t})-\gamma_{t}\left(G_{t}-\mu\mathbf{I}_{n}\right),\label{eq:update}
\end{equation}
 where $G_{t}$ is an estimator for the gradient.

As our initial condition we take 
\[
X_{0}=\frac{1}{2}\mathbf{I}_{n},
\]
 and we let $T$ denote the time horizon of the algorithm, which furnishes
iterates $X_{0},\ldots,X_{T}$. (We will consider an alternative initial
condition in Section \ref{sec:cbl}.) 

In Section \ref{sec:Gradient-estimator} we will explain how the gradient
estimator is constructed and how several key properties are guaranteed.

\subsection{Physical interpretation ($\beta<+\infty$) \label{subsec:Physical-interpretation}}

In this section we assume $\beta<+\infty$.

It is useful to define the \textbf{\emph{Fermi-Dirac function }}
\[
f_{\beta}(x)=\frac{1}{1+e^{\beta x}}.
\]
 Then evidently 
\[
\nabla S_{\mathrm{FD}}^{*}(X^{*})=f_{\beta}(-\beta^{-1}X^{*})=f_{\beta}(H),
\]
 where we view $H:=-\beta^{-1}X^{*}$ as an \textbf{\emph{effective
Hamiltonian }}corresponding to our state $X\in\mathcal{X}$, corresponding
to $X^{*}\in\mathcal{X}^{*}$ via the mirror map.

If we define
\[
H_{t}=-\beta^{-1}X_{t}^{*}=-\beta^{-1}\nabla S(X_{t}),\quad t=1,\ldots,T,
\]
 then we can alternatively characterize the mirror descent update
(\ref{eq:update}) as an update rule for the effective Hamiltonian:
\begin{equation}
H_{t+1}=(1-\beta^{-1}\gamma_{t})H_{t}+\beta^{-1}\gamma_{t}\left(G_{t}-\mu\mathbf{I}_{n}\right),\label{eq:hamiltonianupdate}
\end{equation}
 where $H_{0}=0$. In each update we set $H_{t+1}$ to be a convex
combination of the $H_{t}$ with $G_{t}-\mu\mathbf{I}_{n}$. 

Notice that any fixed point of the iteration map in which exact gradients
$G_{t}=\nabla E(X_{t})$ are used (which necessariliy coincides with
the unique optimizer $X_{\star}$) must satisfy 

\begin{equation}
H_{\star}=\nabla E(X_{\star})-\mu\mathbf{I}_{n},\quad X_{\star}=f_{\beta}\left(\nabla E(X_{\star})-\mu\mathbf{I}_{n}\right).\label{eq:fixedpoint}
\end{equation}

\subsection{Formulation as a proximal algorithm \label{subsec:proximal}}

We introduced the update (\ref{eq:update}) as the instantiation of
mirror descent on the objective 
\[
E(X)-\mu\,\Tr[X]+\beta^{-1}S_{\mathrm{FD}}(X)
\]
with Bregman potential $S_{\mathrm{FD}}$ and step size $\gamma_{t}>0$.
However, we can equivalently view the same update as a proximal algorithm:
\begin{equation}
X_{t+1}=\underset{X\in\mathcal{X}}{\text{argmin}}\left\{ E(X_{t})+\left\langle G_{t}-\mu\mathbf{I}_{n},X-X_{t}\right\rangle +\frac{1}{\beta}S_{\mathrm{FD}}(X)+\frac{1}{\eta_{t}}D(X\Vert X_{t})\right\} ,\label{eq:update2}
\end{equation}
 for a certain choice of $\eta_{t}>0$.

Note that mirror descent can \emph{always} be interpreted as a proximal
algorithm (Proposition 3.2 in \cite{beck2003mirror}), but a key difference
here is that the regularization term $\beta^{-1}S_{\mathrm{FD}}(X)$
is not linearized inside the `argmin.' We are treating it exactly
within the argmin in the sense of composite function minimization
that is famously used for nonsmooth terms that enjoy exploitable structure
\cite{beck2009fast}. Therefore this proximal formulation is different
from the one that is automatically enjoyed by the update as an instance
of mirror descent.

Now to establish the connection between (\ref{eq:update}) and (\ref{eq:update2}),
observe that solving the first-order optimality condition for (\ref{eq:update2})
yields 
\[
\nabla S_{\mathrm{FD}}(X_{t+1})=\left(1-\frac{\eta_{t}}{\eta_{t}+\beta}\right)\nabla S(X_{t})-\frac{\eta_{t}\beta}{\eta_{t}+\beta}\left(G_{t}-\mu\mathbf{I}_{n}\right),
\]
 which coincides with (\ref{eq:update}) under the identification
\[
\gamma_{t}:=\frac{\eta_{t}\beta}{\eta_{t}+\beta}<\beta.
\]

\subsection{Elementary facts \label{subsec:Elementary-facts}}

For now we state a key property of the Fermi-Dirac entropy, which
follows from the famous strong convexity of the von Neumann entropy
\cite{10.1063/1.4871575}.
\begin{lem}
\label{lem:stronglyconvex}$S_{\mathrm{FD}}$ is $(2/n)$-strongly
convex on $\mathcal{X}$ with respect to the nuclear norm $\Vert\,\cdot\,\Vert_{*}$.
\end{lem}

The proof is given in Appendix \ref{app:elem}.

It is also useful to show that the Bregman divergence from the initial
condition $X_{0}=\frac{1}{2}\mathbf{I}_{n}$ is bounded.
\begin{lem}
\label{lem:divbound}For all $X\in\mathcal{X}$, we have 
\[
D(X\Vert X_{0})\leq n\log2.
\]
\end{lem}

The proof is also given in Appendix \ref{app:elem}.

\section{Gradient estimator \label{sec:Gradient-estimator}}

An estimator for $\nabla\tilde{E}(X_{t})$ can be defined by using
a stochastic trace estimator \cite{hutchinson1989stochastic} to estimate
the density $\rho(X_{t})$, viewed as a matrix diagonal following
(\ref{eq:rhodiag}). Note that a similar estimator was used recently
in \cite{lindsey2023fastrandomizedentropicallyregularized}, which
also considered randomized algorithms for semidefinite programming.
Remarkably, we will require only a single shot per optimization step
in our estimator.

Indeed, let $z_{0},\ldots,z_{T-1}\in\R^{n}$ denote independent and
identically distributed standard Gaussian random vectors. Then define
\[
\hat{X}_{t}:=X_{t}^{1/2}z_{t}z_{t}^{\top}X_{t}^{1/2}=\left[X_{t}^{1/2}z_{t}\right]\left[X_{t}^{1/2}z_{t}\right]^{\top}.
\]
 Note that the matrix square root $X_{t}^{1/2}$ can be viewed as
a matrix function of $H_{t}$: 
\[
X_{t}^{1/2}=f_{\beta}^{1/2}(H_{t}),
\]
 where $\beta<+\infty$ for simplicity and $f_{\beta}^{1/2}$ is the
pointwise square root of the Fermi-Dirac function. We will discuss
below in Section \ref{subsec:Matrix-function-implementation} how
this allows us to construct the matrix-vector multiplication $X_{t}^{1/2}z_{t}$
efficiently, without forming the matrix $X_{t}^{1/2}$ directly. For
now we simply assume that this operation can be performed exactly.

Then define 
\begin{equation}
\hat{\rho}_{t}:=\rho(\hat{X}_{t})=\mathrm{diag}\left[\Psi\hat{X}_{t}\Psi^{\top}\right]=(\Psi X_{t}^{1/2}z_{t})^{\odot2},\label{eq:rhohat}
\end{equation}
 where the superscript denotes an entrywise power, and finally: 
\begin{equation}
\tilde{G}_{t}:=\nabla\tilde{E}(\hat{X}_{t})=\Psi^{\top}\mathrm{diag}^{*}\left[V\hat{\rho}_{t}\right]\Psi,\label{eq:tildeGt}
\end{equation}
 following the expression (\ref{eq:Etilderho}) for $\nabla\tilde{E}(X)$
in terms of $\rho(X)$.

It is useful to define the filtration $\{\mathcal{F}_{t}\}_{t=0}^{T-1}$
generated by the random vectors $\{z_{t}\}_{t=0}^{T-1}$. Then evidently
$X_{t+1}$ is measurable with respect to $\mathcal{F}_{t}$ for $t=0,\ldots,T-1$,
and moreover 

\[
\begin{aligned} & \E[\hat{X}_{t}\,\vert\,\mathcal{F}_{t-1}]\,=\,X_{t},\\
 & \E[\hat{\rho}_{t}\,\vert\,\mathcal{F}_{t-1}]\,=\,\rho(X_{t}),\\
 & \E[\tilde{G}_{t}\,\vert\,\mathcal{F}_{t-1}]\,=\,\nabla\tilde{E}(X_{t}).
\end{aligned}
\]
 In particular, $\tilde{G}_{t}$ defines an unbiased estimator for
the gradient at step $t$ (conditioned on the previous randomness).

We finally define our full gradient estimate as 
\begin{equation}
G_{t}=C+\tilde{G}_{t},\label{eq:Gt}
\end{equation}
 which satisfies 
\[
\E[G_{t}\,\vert\,\,\mathcal{F}_{t-1}]=\nabla E(X_{t}).
\]
It is also useful to define the error of the estimator: 
\begin{equation}
\Delta_{t}:=G_{t}-\nabla E(X_{t})=\tilde{G}_{t}-\nabla\tilde{E}(X_{t}),\label{eq:Delta}
\end{equation}
 which satisfies 
\[
\E[\Delta_{t}\,\vert\,\mathcal{F}_{t-1}]=0.
\]

\subsection{Sub-exponential concentration \label{subsec:Sub-exponential-concentration}}

We need concentration bounds for quantities related to the gradient
estimator. To formulate these, we state the definition of a sub-exponential
random variable, following the text of \cite{wainwright2019high}.
(We will use lower-case letters to denote scalar random variables
and vectors in this manuscript, in order to avoid confusion with the
capital-letter notation for matrices.) We slightly extend definitions
to deal with the fact that our sequence of iterates $X_{t}$ is random,
while at each iteration the random vector $z_{t}$ is chosen independently
of all preceding randomness.
\begin{defn}
[Definition 2.7 of \cite{wainwright2019high}] A random variable $x$
with mean $\mu=\E[x]$ is sub-exponential if there are non-negative
parameters $(\nu,b)$ such that 
\[
\E\left[e^{\lambda(x-\mu)}\right]\leq e^{\frac{\nu^{2}\lambda^{2}}{2}}\quad\text{for all }\vert\lambda\vert\leq\frac{1}{b}.
\]
 Given a $\sigma$-algebra $\mathcal{F}$, we will say that $x$ satisfying
$\mu=\E[x\,\vert\,\mathcal{F}]$ is sub-exponential with parameters
$(\nu,b)$, \emph{conditioned on $\mathcal{F}$}, if 
\[
\E\left[e^{\lambda(x-\mu)}\,\vert\,\mathcal{F}\right]\leq e^{\frac{\nu^{2}\lambda^{2}}{2}}\quad\text{for all }\vert\lambda\vert\leq\frac{1}{b}
\]
 holds almost surely. Here $\nu,b\geq0$ can be random variables which
are measurable with respect to $\mathcal{F}$.
\end{defn}

Importantly, sub-exponential random variables satisfy the following
tail bound, which we adapt from \cite{wainwright2019high}: 
\begin{prop}
[Proposition 2.9 of \cite{wainwright2019high}] \label{prop:subexp}
Suppose that $x$ is sub-exponential with parameters $(\nu,b)$ and
$\E[x]=\mu$. Then for $t\geq0$, 
\[
\P\left[x\geq\mu+t\right]\leq\begin{cases}
e^{-\frac{t^{2}}{2\nu^{2}}}, & \text{if }\,0\leq t\leq\frac{\nu^{2}}{b},\\
e^{-\frac{t}{2b}}, & \text{if }\,t>\frac{\nu^{2}}{b}.
\end{cases}
\]
 Likewise, if $x$ is sub-exponential with parameters $(\nu,b)$,
conditioned on some $\sigma$-algebra $\mathcal{F}$, and $\E[x\,\vert\,\mu]=\mathcal{F}$,
then for $t\geq0$, it holds almost surely that 
\[
\P\left[x\geq\mu+t\,\vert\,\mathcal{F}\right]\leq\begin{cases}
e^{-\frac{t^{2}}{2\nu^{2}}}, & \text{if }\,0\leq t\leq\frac{\nu^{2}}{b},\\
e^{-\frac{t}{2b}}, & \text{if }\,t>\frac{\nu^{2}}{b}.
\end{cases}
\]
\end{prop}

\begin{rem}
The second statement can be proved just by following the ordinary
proof of the first statement, which uses the Chernoff technique. But
we can also view the second statement as following directly from the
first statement via the regular conditional probability.
\end{rem}

In order to state some results more transparently later on, we state
and prove the following simple corollary.
\begin{cor}
\label{cor:subexp} Suppose $x$ is sub-exponential with deterministic
parameters $(\nu,2\nu)$, conditioned on some $\sigma$-algebra $\mathcal{F}$,
and $\E[x\,\vert\,\mathcal{F}]=\mu$. Then for any $\delta\in(0,1]$,
\[
x\leq\mu+\nu\left(\frac{1}{2}+4\log(1/\delta)\right)
\]
 with probability at least $1-\delta$.
\end{cor}

The proof is given in Appendix \ref{app:subexp}.

Sub-exponential concentration bounds are relevant because if $z\sim\mathcal{N}(0,\mathbf{I}_{n})$,
the trace estimator $z^{\top}Az$, which satisfies $\E[z^{\top}Az]=\Tr[A]$,
is a sub-exponential random variable. The result is true more generally
the entries of $z$ are i.i.d. sub-Gaussian random variables with
zero mean and unit variance \cite{Meyer2021-pn}. However, for simplicity,
we just consider the case of Gaussian $z$ and extend the usual statement
to account for the fact that the matrix $A$ may be random, while
$z$ is chosen independently of $A$.
\begin{lem}
\label{lem:hutchsubexp} Let $\mathcal{F}$ be a $\sigma$-algebra,
and let $A\in\R^{n\times n}$ be a random matrix that is measurable
with respect to $\mathcal{F}$. Let $z\sim\mathcal{N}(0,\mathbf{I}_{n})$
be independent of $\mathcal{F}$. Then $z^{\top}Az$ is sub-exponential
with parameters $(2\Vert A\Vert_{\mathrm{F}},4\Vert A\Vert)$, conditioned
on $\mathcal{F}$.
\end{lem}

The proof is also given in Appendix \ref{app:subexp}.

We will need one more definition and result from \cite{wainwright2019high}:
\begin{defn}
[(2.26) of \cite{wainwright2019high}] \label{def:mds} We say that
a sequence $\{y_{t}\}_{t=0}^{T-1}$, adapted to the filtration $\{\mathcal{F}_{t}\}_{t=0}^{T-1}$,
is a \emph{martingale difference sequence }if $\E\left[\vert y_{t}\vert\right]<+\infty$
for all $t=0,\ldots,T-1$ and
\[
\E[y_{t}\,\vert\,\mathcal{F}_{t-1}]=0
\]
 for all $t=1,\ldots,T-1$.
\end{defn}

Martingale difference sequences satisfy the following Bernstein-type
concentration bound: 
\begin{thm}
[Theorem 2.19 of \cite{wainwright2019high}] \label{thm:mds} Let
$\{y_{t}\}_{t=0}^{T-1}$ be a martingale difference sequence adapted
to the filtration $\{\mathcal{F}_{t}\}_{t=0}^{T-1}$, and suppose
that $y_{t}$ is sub-exponential with \emph{deterministic }parameters
$(\nu_{t},b_{t})$, conditioned on $\mathcal{F}_{t-1}$. Then $\sum_{t=0}^{T-1}y_{t}$
is sub-exponential with parameters 
\[
\left(\sqrt{\sum_{t=0}^{T-1}\nu_{t}^{2}},\ \max_{t=0,\ldots,T-1}\{b_{t}\}\right).
\]
\end{thm}

\subsection{Gradient concentration bounds \label{subsec:Gradient-concentration-bounds}}

We can apply Lemma \ref{lem:hutchsubexp} in several ways to obtain
the following concentration bounds which will be used to control the
convergence rate of (\ref{eq:update}) with high probability.

The first is a bound on the spectral norm of the gradient.
\begin{lem}
\label{lem:Gestbound} For any $\delta\in(0,1]$, the inequality 
\[
\max_{t=0,\ldots,T-1}\Vert\hat{G}_{t}\Vert\leq2\left(1+4\log(Tm/\delta)\right)c_{\mathrm{h}}
\]
 holds with probability at least $1-\delta$.
\end{lem}

\begin{proof}
Note that each entry of $\hat{\rho}_{t}=\mathrm{diag}\left[\Psi\hat{X}_{t}\Psi^{\top}\right]$
(cf. (\ref{eq:rhohat})) can be viewed as a trace estimate via 
\[
[\hat{\rho}_{t}]_{q}=z_{t}^{\top}\left[X_{t}^{1/2}\Psi^{\top}e_{q}e_{q}^{\top}\Psi X_{t}^{1/2}\right]z_{t}.
\]
 Hence by Lemma \ref{lem:hutchsubexp}, for $t\geq1$ the random variable
$[\hat{\rho}_{t}]_{q}$ is sub-exponential with parameters $(2\Vert A_{t}\Vert_{\mathrm{F}},4\Vert A_{t}\Vert)$,
conditioned on $\mathcal{F}_{t-1}$, where $A_{t}:=X_{t}^{1/2}\Psi^{\top}e_{q}e_{q}^{\top}\Psi X_{t}^{1/2}$.
(When $t=0$, $A_{t}$ is deterministic, so $\hat{\rho}_{1}$ is simply
sub-exponential with these parameters.)

This matrix $A_{t}$ is rank-one and we can compute $\Vert A_{t}\Vert_{\mathrm{F}}=\Vert A_{t}\Vert=\sqrt{\Tr[A_{t}^{2}]}=[\rho(X_{t})]_{q}$.
From Lemma \ref{lem:Gbound}, we have that $\Vert\rho(X_{t})\Vert_{\infty}\leq c_{\Psi}\Vert X_{t}\Vert\leq c_{\Psi}$.
Therefore each entry $[\hat{\rho}_{t}]_{q}$ is sub-exponential with
the deterministic parameters $(2c_{\Psi},4c_{\Psi})$, conditioned
on $\mathcal{F}_{t-1}$. (Again, we don't need to condition for $t=0$.)

We can then apply Corollary \ref{cor:subexp} with $\nu=2c_{\Psi}$
to deduce that 
\[
[\hat{\rho}_{t}]_{q}\leq[\rho(X_{t})]_{q}+c_{\Psi}\left(1+8\log(Tm/\delta)\right)
\]
 holds with probability at least $1-\frac{\delta}{Tm}$.

Then by the union bound, together with the fact that $[\rho(X_{t})]_{q}\leq\Vert\rho(X_{t})\Vert_{\infty}\leq c_{\Psi}$,
we conclude that 
\[
\Vert\hat{\rho}_{t}\Vert_{\infty}\leq2\left(1+4\log(Tm/\delta)\right)c_{\Psi}
\]
 holds for all $t=0,\ldots,T-1$ with probabilty at least $1-\delta$.

Then it follows that 
\[
\Vert\hat{G}_{t}\Vert=\Vert\Psi^{\top}\mathrm{diag}^{*}\left[V\hat{\rho}_{t}\right]\Psi\Vert\leq\Vert V\Vert_{\infty}\Vert\hat{\rho}_{t}\Vert_{\infty}\leq2\left(1+4\log(Tm/\delta)\right)c_{\mathrm{h}}
\]
 holds for all $t=0,\ldots,T-1$ with probability at least $1-\delta$,
as was to be shown.
\end{proof}
Next is a fact that we will need in combination with Theorem \ref{thm:mds}
above.
\begin{lem}
\label{lem:mdsprep} Suppose $Y\in\R^{n\times n}$ is random but measurable
with respect to $\mathcal{F}_{t-1}$ for some $t\geq1$, and define
$y=\left\langle \Delta_{t},Y\right\rangle $. Assume that $\Vert Y\Vert\leq1$.
Then $y$ is sub-exponential with parameters $(2c_{\mathrm{h}}\Vert X_{t}\Vert_{\mathrm{F}},4c_{\mathrm{h}})$,
conditioned on $\mathcal{F}_{t-1}$.
\end{lem}

\begin{proof}
Compute 
\begin{align*}
\left\langle \tilde{G}_{t},Y\right\rangle  & =\left\langle \nabla\tilde{E}(\hat{X}_{t}),Y\right\rangle \\
 & =\sum_{ijkl}Y_{ij}\,v_{ij,kl}\,[\hat{X}_{t}]_{kl}\\
 & =\left\langle \hat{X}_{t},\nabla\tilde{E}(Y)\right\rangle \\
 & =\left\langle X_{t}^{1/2}z_{t}z_{t}^{\top}X_{t}^{1/2},\nabla\tilde{E}(Y)\right\rangle \\
 & =z_{t}^{\top}\left[X_{t}^{1/2}\nabla\tilde{E}(Y)X_{t}^{1/2}\right]z_{t},
\end{align*}
 so $\left\langle \tilde{G}_{t},Y\right\rangle $ can be viewed as
a trace estimator for $A:=X_{t}^{1/2}\nabla\tilde{E}(Y)X_{t}^{1/2}$
(which is itself measurable with respect to $\mathcal{F}_{t-1}$).
Then by Lemma \ref{lem:hutchsubexp}, we have that $\left\langle \tilde{G}_{t},Y\right\rangle $
is sub-exponential with parameters $(2\Vert A\Vert_{\mathrm{F}},4\Vert A\Vert)$,
conditioned on $\mathcal{F}_{t-1}$.

Since $Y$ is measurable with respect to $\mathcal{F}_{t-1}$, we
have that 
\[
\E\left[\left\langle \tilde{G}_{t},Y\right\rangle \,\vert\,\mathcal{F}_{t-1}\right]=\left\langle \nabla\tilde{E}(X_{t}),Y\right\rangle ,
\]
 and therefore
\[
y=\left\langle \tilde{G}_{t},Y\right\rangle -\E\left[\left\langle \tilde{G}_{t},Y\right\rangle \,\vert\,\mathcal{F}_{t-1}\right],
\]
 i.e., $y$ is the difference of $\left\langle \tilde{G}_{t},Y\right\rangle $
and its conditional mean. It follows that $y$ is sub-exponential
with the same parameters, conditioned on $\mathcal{F}_{t-1}$.

Now we can compute 
\[
\Vert A\Vert=\Vert X_{t}^{1/2}\nabla\tilde{E}(Y)X_{t}^{1/2}\Vert\leq\Vert X_{t}^{1/2}\Vert^{2}\,\Vert\nabla\tilde{E}(Y)\Vert\leq c_{\mathrm{h}},
\]
 where we have used Lemma \ref{lem:Gbound} and the fact that $\Vert X_{t}\Vert\leq1$
in the last inequality.

Moreover, 
\begin{align*}
\Vert A\Vert_{\mathrm{F}}^{2} & =\Tr\left[\nabla\tilde{E}(Y)X_{t}\nabla\tilde{E}(Y)X_{t}\right]\\
 & =\left\langle \left(\nabla\tilde{E}(Y)X_{t}\right)^{\top},\nabla\tilde{E}(Y)X_{t}\right\rangle \\
 & \leq\Vert\nabla\tilde{E}(Y)X_{t}\Vert_{\mathrm{F}}^{2}\\
 & \leq\Vert\nabla\tilde{E}(Y)\Vert^{2}\Vert X_{t}\Vert_{\mathrm{F}}^{2}\\
 & \leq c_{\mathrm{h}}^{2}\Vert X_{t}\Vert_{\mathrm{F}}^{2},
\end{align*}
 and this completes the proof.
\end{proof}

\subsection{Matrix function implementation \label{subsec:Matrix-function-implementation}}

In order to implement the gradient estimator we must perform the matrix-vector
multiplications $X_{t}^{1/2}z_{t}$ within the construction of $\hat{\rho}_{t}$
(\ref{eq:rhohat}). We already commented that when $\beta<+\infty$,
we can think of 
\[
X_{t}^{1/2}=f_{\beta}^{1/2}(H_{t})=f_{1}^{1/2}(\beta H_{t}).
\]

One standard approach to such a matrix-vector multiplication involving
a matrix function is to perform Chebyshev approximation \cite{trefethen2019approximation}
of order $r$ of $f_{\beta}^{1/2}$ on an interval that bounds the
spectrum of $H_{t}$. (Equivalently, we can perform Chebyshev approximation
of $f_{1}^{1/2}$ on an interval bounding the spectrum of $\beta H_{t}$.)
Then the three-term recurrence \cite{trefethen2019approximation}
can be used to compute $f_{\beta}^{1/2}(H_{t})z_{t}$ using $r$ matrix-vector
multiplications by $H_{t}$, which we can construct efficiently.

Another approach involves the approximation of $f_{\beta}^{1/2}$
using a contour integral approach. A similar approach for approximating
the Fermi-Dirac function is applied within PEXSI method \cite{lin2009pole,CMS2009,JCPM2013}
for DFT, but here we must adapt the approach to deal with the branch
cut due to the square root. Ultimately, this approach requires us
to solve several linear systems involving shifted copies of the effective
Hamiltonian. The details are given in Appendix \ref{app:contour}.

We will not consider the error due to matrix function approximation
in our ensuing analysis as it can be made exponentially small by increasing
the polynomial order (in the first approach) or the number of poles
(in the second approach). However, we comment that the second approach
can benefit from a good preconditioner, which is particularly useful
when a larger basis set is used and the norm of the discretized Laplacian
operator appearing within $H_{t}$ contributes to the conditioning.

Either approach motivates us to bound the spectrum of $H_{t}$ over
the optimization trajectory. This can be done in terms of the following
bound for the gradient over the optimization trajectory.
\begin{defn}
\label{def:gmax} Let 
\[
\tilde{g}_{\max}=\max_{t=0,\ldots,T-1}\Vert\tilde{G}_{t}\Vert.
\]
\end{defn}

\begin{rem}
Recall from Lemma \ref{lem:Gestbound} that $\hat{g}_{\max}=O(c_{\mathrm{h}}\,\log(Tm/\delta))$
with probability at least $1-\delta$.
\end{rem}

\begin{lem}
\label{lem:EffHambound} Assume $\beta<+\infty$, fix $T$, and assume
that $\gamma_{t}\in(0,\beta]$ for all $t=0,\ldots,T-1$. Then 
\[
\Vert H_{t}\Vert\leq\Vert C-\mu\mathbf{I}_{n}\Vert+\tilde{g}_{\max}
\]
 for all $t=0,\ldots,T$.
\end{lem}

\begin{proof}
From (\ref{eq:hamiltonianupdate}) we know that the effective Hamiltonian
satisfies the update 
\[
H_{t+1}=(1-\beta^{-1}\gamma_{t})H_{t}+(\beta^{-1}\gamma_{t})\,(G_{t}-\mu\mathbf{I}_{n}).
\]
 Moreover, $G_{t}=C+\tilde{G}_{t},$ so 
\[
\Vert H_{t+1}\Vert\leq(1-\beta^{-1}\gamma_{t})\Vert H_{t}\Vert+(\beta^{-1}\gamma_{t})\,\left(\Vert C-\mu\mathbf{I}_{n}\Vert+\tilde{g}_{\max}\right)
\]
 for all $t=0,\ldots,T-1$.

Then by induction (noting that $H_{0}=0$) it follows that 
\[
\Vert H_{t}\Vert\leq\Vert C-\mu\mathbf{I}_{n}\Vert+\tilde{g}_{\max}
\]
 for all $t=0,\ldots,T$.
\end{proof}

\section{Convergence theorem and proof \label{sec:Convergence}}

Now we prove the convergence of the algorithm (\ref{eq:update}).
For the proof, we are inspired by \cite{beck2003mirror}, from which
we adapt several key facts and arguments. In particular, we reproduce
the statement of this lemma (Lemma 2 from \cite{chen1993convergence}): 
\begin{lem}
\label{lem:pythagoras}Let $\Phi$ be convex on $\mathcal{X}$, $\tilde{X}\in\mathcal{X}$,
and 
\[
X'=\underset{X\in\mathcal{X}}{\textrm{argmin}}\left\{ \Phi(X)+D(X\Vert\tilde{X})\right\} .
\]
 Then for any $Y\in\mathcal{X}$, we have 
\[
\Phi(Y)+D(Y\Vert\tilde{X})\geq\Phi(X')+D(X'\Vert\tilde{X})+D(Y\Vert X').
\]
\end{lem}

Now we state and prove our convergence theorem.
\begin{thm}
\label{thm:convergence} For any $\delta\in(0,1]$, define $c_{T,m,\delta}:=2\left(1+4\log(2Tm/\delta)\right).$
Consider algorithm defined by (\ref{eq:update2}) with step size $\eta_{t}=\eta:=\frac{1}{c_{T,m,\delta}\,c_{\mathrm{h}}\,\sqrt{T}}$,
or equivalently by (\ref{eq:update}) with $\gamma_{t}:=\frac{\eta\beta}{\eta+\beta}$,
and initial condition $X_{0}=\mathbf{I}_{n}/2$. Then 
\[
\frac{1}{T}\sum_{t=0}^{T-1}\frac{F_{\beta}(X_{t})-\mu\Tr[X_{t}]}{n}\leq\frac{F_{\beta}(X_{\star})-\mu\Tr[X_{\star}]}{n}+\frac{c_{T,m,\delta}\,c_{\mathrm{h}}}{\sqrt{T}}\left(\log2+\frac{1}{4}+\frac{1}{\sqrt{n}}\right)+\frac{\Vert C-\mu\mathbf{I}_{n}\Vert}{T}
\]
 holds with probability at least $1-\delta$.
\end{thm}

\begin{rem}
In particular it follows that there exists $t\in\{0,1,\ldots,T-1\}$
such that 
\[
\frac{F_{\beta}(X_{t})-\mu\Tr[X_{t}]}{n}\leq\frac{F_{\beta}(X_{\star})-\mu\Tr[X_{\star}]}{n}+O\left(\frac{c_{\mathrm{h}}\,\log(Tm/\delta)}{\sqrt{T}}\right)+\frac{\Vert C-\mu\mathbf{I}_{n}\Vert}{T}.
\]
 It also follows from the theorem, by the convexity of $F_{\beta}$,
that 
\[
\frac{F_{\beta}(\overline{X}_{T})-\mu\Tr[\overline{X}_{T}]}{n}\leq\frac{F_{\beta}(X_{\star})-\mu\Tr[X_{\star}]}{n}+O\left(\frac{c_{\mathrm{h}}\,\log(Tm/\delta)}{\sqrt{T}}\right)+\frac{\Vert C-\mu\mathbf{I}_{n}\Vert}{T},
\]
 where $\overline{X}_{T}:=\frac{1}{T}\sum_{t=0}^{T-1}X_{t}$ is the
Cesàro sum.
\end{rem}

\begin{rem}
It is sensible to expect that the error bound for the objective should
scale with $n$, since the optimal value itself should scale at least
size-extensively as $\Omega(n)$. In fact, it may even scale faster,
due to the heavy tails of the electron-electron Coulomb interaction,
as we shall discuss below in Section \ref{subsec:Electron-repulsion-integrals}.
We can view $\frac{1}{n}F_{\beta}$ as the \textbf{\emph{free energy
density}} per basis function.
\end{rem}

\begin{rem}
Observe that the dependence of the error on $\Vert C-\mu\mathbf{I}_{n}\Vert$
(which is modulated by $1/T$) is milder than the dependence on $c_{\mathrm{h}}$
(which is modulated by $1/\sqrt{T}$). This is fortunate because in
our intended applications, $C$ involves the Galerkin projection of
a Laplacian operator, which may have a significantly larger norm than
$c_{\mathrm{h}}$ (which arises from the discretization of an integral
operator). In practice, however, quantum chemistry basis sets are
not so large that $\Vert C\Vert$ should be considered unbounded.
\end{rem}

\begin{proof}
For simplicity, without loss of generality we can set $\mu=0$ by
absorbing $\mu\,\Tr[X]$ into the definition of $E(X)$ via the replacement
$C\leftarrow C-\mu\mathbf{I}_{n}$.

Recall $G_{t}=C+\tilde{G}_{t}$ denotes our gradient estimator (\ref{eq:Gt}),
where $\tilde{G}_{t}$ is defined by (\ref{eq:tildeGt}). Also recall
that $\Delta_{t}=G_{t}-\nabla E(X_{t})$ denotes the error of our
gradient estimator (defined earlier in (\ref{eq:Delta})), and that
$\E\left[\Delta_{t}\,\vert\,\mathcal{F}_{t-1}\right]=0$.

Within Lemma \ref{lem:pythagoras}, consider the choice 
\[
\Phi(X)=\eta\left[E(X_{t})+\left\langle G_{t},X-X_{t}\right\rangle +\frac{1}{\beta}S_{\mathrm{FD}}(X)\right],
\]
 for $\Phi$. Also consider $X_{t}$ in the place of $\tilde{X}$,
as well as $X_{\star}$ in the place of $Y$. Our proximal interpretation
of the update (\ref{eq:update2}) indicates that moreover we can take
$X_{t+1}$ in the place of $X'$. Then we deduce 
\begin{align*}
 & \eta\left[E(X_{t})+\left\langle G_{t},X_{\star}-X_{t}\right\rangle +\frac{1}{\beta}S_{\mathrm{FD}}(X_{\star})\right]+D(X_{\star}\Vert X_{t})\\
 & \quad\quad\geq\ \eta\left[E(X_{t})+\left\langle G_{t},X_{t+1}-X_{t}\right\rangle +\frac{1}{\beta}S_{\mathrm{FD}}(X_{t+1})\right]+D(X_{t+1}\Vert X_{t})+D(X_{\star}\Vert X_{t+1}).
\end{align*}
 We can then cancel and rearrange terms to determine that 
\begin{equation}
D(X_{\star}\Vert X_{t+1})\leq D(X_{\star}\Vert X_{t})+\eta\left\langle G_{t},X_{\star}-X_{t+1}\right\rangle +\frac{\eta}{\beta}\left[S_{\mathrm{FD}}(X_{\star})-S_{\mathrm{FD}}(X_{t+1})\right]-D(X_{t+1}\Vert X_{t}).\label{eq:in1}
\end{equation}
Now the strong convexity of $S_{\mathrm{FD}}$ (Lemma \ref{lem:stronglyconvex})
implies that 
\begin{equation}
D(X_{t+1}\Vert X_{t})\geq\frac{1}{n}\Vert X_{t+1}-X_{t}\Vert_{*}^{2},\label{eq:in2}
\end{equation}
 and we can also compute 
\begin{align}
\left\langle G_{t},X_{\star}-X_{t+1}\right\rangle = & \left\langle G_{t},X_{\star}-X_{t}\right\rangle +\left\langle G_{t},X_{t}-X_{t+1}\right\rangle \nonumber \\
= & \left\langle \nabla E(X_{t}),X_{\star}-X_{t}\right\rangle +\left\langle \Delta_{t},X_{\star}-X_{t}\right\rangle +\left\langle G_{t},X_{t}-X_{t+1}\right\rangle \nonumber \\
= & \left\langle \nabla E(X_{t}),X_{\star}-X_{t}\right\rangle +\left\langle \Delta_{t},X_{\star}-X_{t}\right\rangle +\left\langle C,X_{t}-X_{t+1}\right\rangle +\left\langle \tilde{G}_{t},X_{t}-X_{t+1}\right\rangle \nonumber \\
\le & \left[E(X_{\star})-E(X_{t})\right]+\left\langle \Delta_{t},X_{\star}-X_{t}\right\rangle +\left\langle C,X_{t}-X_{t+1}\right\rangle +\Vert\tilde{G}_{t}\Vert\,\Vert X_{t+1}-X_{t}\Vert_{*},\label{eq:in3}
\end{align}
 where in the last line we have used the convexity of $E$ and the
fact that the spectral and nuclear norms are duals.

Then by combining (\ref{eq:in1}), (\ref{eq:in2}), and (\ref{eq:in3}),
we obtain: 
\begin{align}
D(X_{\star}\Vert X_{t+1}) & \leq D(X_{\star}\Vert X_{t})+\eta\left[E(X_{\star})-E(X_{t})\right]+\eta\left\langle \Delta_{t},X_{\star}-X_{t}\right\rangle +\frac{\eta}{\beta}\left[S_{\mathrm{FD}}(X_{\star})-S_{\mathrm{FD}}(X_{t+1})\right]\nonumber \\
 & \quad\quad+\ \eta\left\langle C,X_{t}-X_{t+1}\right\rangle +\eta\Vert\tilde{G}_{t}\Vert\,\Vert X_{t+1}-X_{t}\Vert_{*}-\frac{1}{n}\Vert X_{t+1}-X_{t}\Vert_{*}^{2}.\label{eq:in4}
\end{align}
 Now we can use the general inequality $ab\leq\frac{1}{2}a^{2}+\frac{1}{2}b^{2}$
to bound the first expression in the last line: 
\begin{equation}
\eta\Vert\tilde{G}_{t}\Vert\,\Vert X_{t+1}-X_{t}\Vert_{*}\leq\frac{n}{4}\eta^{2}\Vert\tilde{G}_{t}\Vert^{2}+\frac{1}{n}\Vert X_{t+1}-X_{t}\Vert_{*}^{2}.\label{eq:in5}
\end{equation}
 Then combining (\ref{eq:in4}) and (\ref{eq:in5}) and defining $d_{t}:=D(X_{\star}\Vert X_{t})$
for all $t$, we find that 
\begin{align*}
d_{t+1}-d_{t} & \leq\eta\left[E(X_{\star})-E(X_{t})\right]+\eta\left\langle \Delta_{t},X_{\star}-X_{t}\right\rangle +\frac{\eta}{\beta}\left[S_{\mathrm{FD}}(X_{\star})-S_{\mathrm{FD}}(X_{t+1})\right]\\
 & \quad\quad+\ \eta\left\langle C,X_{t}-X_{t+1}\right\rangle +\frac{n}{4}\eta^{2}\Vert\tilde{G}_{t}\Vert^{2}.
\end{align*}
Now recall that $F_{\beta}=E+\beta^{-1}S_{\mathrm{FD}}$, so we can
group terms to obtain 
\begin{align*}
d_{t+1}-d_{t} & \leq\eta\left[F_{\beta}(X_{\star})-F_{\beta}(X_{t})\right]+\eta\left\langle \Delta_{t},X_{\star}-X_{t}\right\rangle +\frac{\eta}{\beta}\left[S_{\mathrm{FD}}(X_{t})-S_{\mathrm{FD}}(X_{t+1})\right]\\
 & \quad\quad+\ \eta\left\langle C,X_{t}-X_{t+1}\right\rangle +\frac{n}{4}\eta^{2}\Vert\tilde{G}_{t}\Vert^{2}
\end{align*}
 and, after some further rearrangement to isolate the terms involving
$F_{\beta}$: 
\begin{align*}
\left[F_{\beta}(X_{t})-F_{\beta}(X_{\star})\right] & \leq\frac{1}{\eta}\left[d_{t}-d_{t+1}\right]+\frac{n}{4}\eta\Vert\tilde{G}_{t}\Vert^{2}+\left\langle \Delta_{t},X_{\star}-X_{t}\right\rangle \\
 & \quad\quad+\ \left\langle C,X_{t}-X_{t+1}\right\rangle +\frac{1}{\beta}\left[S_{\mathrm{FD}}(X_{t})-S_{\mathrm{FD}}(X_{t+1})\right].
\end{align*}
 Finally we can sum both sides over $t=0,\ldots,T-1$, taking advantage
of telescoping, to deduce: 
\begin{align}
\sum_{t=0}^{T-1}\left[F_{\beta}(X_{t})-F_{\beta}(X_{\star})\right] & \leq\frac{1}{\eta}\left[d_{0}-d_{T}\right]+\frac{n}{4}\eta\sum_{t=0}^{T-1}\Vert\tilde{G}_{t}\Vert^{2}+\sum_{t=0}^{T-1}\left\langle \Delta_{t},X_{\star}-X_{t}\right\rangle \nonumber \\
 & \quad\quad+\ \left\langle C,X_{0}-X_{T}\right\rangle +\frac{1}{\beta}\left[S_{\mathrm{FD}}(X_{0})-S_{\mathrm{FD}}(X_{T})\right].\label{eq:reflater}
\end{align}
 Note that our initialization $X_{0}=\mathbf{I}_{n}/2$ was chosen
to minimize $S_{\mathrm{FD}}$, the last expression on the right-hand
side is $\leq0$. Moreover $d_{T}\geq0$, and $\left\langle C,X_{0}-X_{T}\right\rangle \leq\Vert C\Vert\,\Vert X_{0}-X_{T}\Vert_{*}\leq n\Vert C\Vert$
(since $-\mathbf{I}_{n}\preceq X_{0}-X_{T}\preceq\mathbf{I}_{n}$),
so we obtain a simpler inequality: 
\begin{equation}
\sum_{t=0}^{T-1}\left[F_{\beta}(X_{t})-F_{\beta}(X_{\star})\right]\leq\frac{d_{0}}{\eta}+\frac{n}{4}\eta\sum_{t=0}^{T-1}\Vert\tilde{G}_{t}\Vert^{2}+\sum_{t=0}^{T-1}\left\langle \Delta_{t},X_{\star}-X_{t}\right\rangle +n\Vert C\Vert.\label{eq:in6}
\end{equation}
 Now by Lemma \ref{lem:Gestbound}, 
\begin{equation}
\sum_{t=0}^{T-1}\Vert\tilde{G}_{t}\Vert^{2}\leq c_{T,m,\delta}^{2}\,c_{\mathrm{h}}^{2}\,T\label{eq:in7}
\end{equation}
 with probability at least $1-\delta/2$, so finally we turn toward
the control of the random variable $\sum_{t=0}^{T-1}y_{t}$, where
$y_{t}:=\left\langle \Delta_{t},X_{\star}-X_{t}\right\rangle $.

Evidently$\{y_{t}\}_{t=0}^{T-1}$ is adapted to the filtration $\{\mathcal{F}_{t}\}_{t=0}^{T-1}$,
and moreover, since $X_{t}$ is measurable with respect to $\mathcal{F}_{t-1}$,
we also have $\E\left[y_{t}\,\vert\,\mathcal{F}_{t-1}\right]=0.$
Therefore $\{y_{t}\}$ is a martingale difference sequence in the
sense of Definition \ref{def:mds}. Also, by Lemma \ref{lem:mdsprep},
$y_{t}$ is sub-exponential with parameters $(2\sqrt{n}c_{\mathrm{h}},4c_{\mathrm{h}})$,
conditioned on $\mathcal{F}_{t-1}$. (We are using the immediate fact
that $\Vert X_{t}\Vert_{\mathrm{F}}\leq\sqrt{n}$, since $X_{t}\in\mathcal{X}$.)
Then by Theorem \ref{thm:mds}, we have that $\sum_{t=0}^{T-1}y_{t}$
is sub-exponential with parameters $\left(2\sqrt{Tn}\,c_{\mathrm{h}},4c_{\mathrm{h}}\right)$.
We can always loosen these parameters a bit to say that $\sum_{t=0}^{T-1}y_{t}$
is sub-exponential with parameters $\left(2\sqrt{Tn}\,c_{\mathrm{h}},4\sqrt{Tn}c_{\mathrm{h}}\right)$
and then apply Corollary \ref{cor:subexp} to deduce that 
\begin{equation}
\sum_{t=0}^{T-1}y_{t}\leq2\left(\frac{1}{2}+4\log(2/\delta)\right)\sqrt{Tn}\,c_{\mathrm{h}}\leq c_{T,m,\delta}c_{\mathrm{h}}\sqrt{Tn}\label{eq:in8}
\end{equation}
 with probability at least $1-\delta/2$.

Combining (\ref{eq:in6}), (\ref{eq:in7}), and (\ref{eq:in8}), as
well as Lemma \ref{lem:divbound}, which says that $d_{0}\leq n\log2$,
we see that 
\[
\sum_{t=0}^{T-1}\left[F_{\beta}(X_{t})-F_{\beta}(X_{\star})\right]\leq\frac{n\log2}{\eta}+\frac{n}{4}\eta\,c_{T,m,\delta}^{2}\,c_{\mathrm{h}}^{2}\,T+c_{T,m,\delta}c_{\mathrm{h}}\sqrt{Tn}+n\Vert C\Vert
\]
 with probability at least $1-\delta$. By taking $\eta=\frac{1}{c_{T,m,\delta}\,c_{\mathrm{h}}\,\sqrt{T}}$,
we complete the proof.
\end{proof}

\section{Structure of the energy \label{sec:Structure-of-the-energy}}

Now we elucidate the composition of the energy function $E(X)$ introduced
above in Section \ref{sec:Preliminaries} and explain how to understand
the quantities $\Vert C-\mu\mathbf{I}\Vert$ and $c_{\mathrm{h}}=c_{\Psi}\Vert V\Vert_{\infty}$
appearing in our convergence analysis in terms of the underlying physical
problem. Our analysis does not rely on the presentation outlined in
this section, which is included only for physical motivation and context.
We include a discussion of general stochastic DFT beyond the Hartree
approximation in Appendix \ref{app:dft}, where we also offer some
commentary on the broader implications of our analysis.

\subsection{Single-electron term \label{subsec:Single-electron-term}}

$C$ is the matrix of the single-particle part of the quantum chemistry
Hamiltonian in the $\{\psi_{i}\}$ basis: 
\[
C_{ij}=\int\psi_{i}(x)\left[-\frac{1}{2}\Delta+v_{\mathrm{ext}}\right]\psi_{j}(x)\,dx,
\]
 where $v_{\mathrm{ext}}$ denotes the diagonal external potential.
Typically the chemical potential should be chosen to lie within the
spectrum of $C$. In this case $\Vert C-\mu\mathbf{I}\Vert\leq\Vert C\Vert$.
It is common to assume some bound on $\Vert C\Vert$ for many calculations
within a quantum chemistry basis. We defer discussion of the complete
basis set limit (in which $\Vert C\Vert$ is unbounded) to Section
\ref{sec:cbl}.

\subsection{Electron repulsion integrals \label{subsec:Electron-repulsion-integrals}}

Recall from (\ref{eq:ERI0}) that we assume the following structure
for the ERI tensor 
\begin{equation}
v_{ij,kl}=\sum_{p,q=1}^{m}\Psi_{pi}\Psi_{pj}\,V_{pq}\,\Psi_{qk}\Psi_{ql},\label{eq:thc1}
\end{equation}
 where $V$ is positive semidefinite and $\Psi$ has orthonormal columns.
This tensor should satisfy \cite{szabo1996modern} 
\begin{equation}
v_{ij,kl}\approx\int\psi_{i}(x)\psi_{j}(x)\,v_{\mathrm{ee}}(x-y)\,\psi_{k}(y)\psi_{l}(y)\,dx\,dy,\label{eq:eri}
\end{equation}
 where $v_{\mathrm{ee}}$ is the potential of the electron-electron
interaction which defines a positive semidefinite kernel $v_{\mathrm{ee}}(x-y)$.
In electronic structure we typically have the Coulomb potential $v_{\mathrm{ee}}(z)=\frac{1}{\vert z\vert}$,
but it is also interesting to consider for example the case of the
Yukawa potential $v_{\mathrm{ee}}(z)=\frac{e^{-\alpha\vert z\vert}}{\vert z\vert}$
modeling screened interactions.

To understand how (\ref{eq:thc1}), which can be viewed as a tensor
hypercontraction (THC) format \cite{10.1063/1.4732310,10.1063/1.4768233,10.1063/1.4768241}
for the ERI, is related to (\ref{eq:eri}), let $\{\phi_{p}\}_{p=1}^{m}$
denote a collection of interpolating functions associated to a collection
interpolation points $\{x_{p}\}_{p=1}^{m}$ satisfying 
\[
w_{p}:=\int\phi_{p}(x)\,dx>0
\]
 and 
\begin{equation}
g(x)\approx\sum_{p=1}^{m}g(x_{p})\,\phi_{p}(x)\label{eq:interp}
\end{equation}
 for functions $g$ that are sufficiently smooth on a suitable region
of support. Since quantum chemistry basis functions are typically
smooth and effectively compactly supported (with exponentially decaying
tails), it is reasonable to assume that $m$ can be taken to be proportional
to the volume of the computational domain, with $\{x_{p}\}_{p=1}^{m}$
taken to be equispaced grid points and $\{\phi_{p}\}_{p=1}^{m}$ chosen
to be an interpolating Meyer wavelet basis \cite{Daubechies_1992}
or a more compactly supported alternative such as a gausslet basis
\cite{10.1063/1.5007066}, while maintaining that (\ref{eq:interp})
holds with high accuracy for all choices $g=\psi_{i}\psi_{j}$. (The
accuracy can be reasonably assumed to be spectral with respect to
$m$.)

In turn it follows that the approximation
\[
\psi_{i}(x)\psi_{j}(x)\approx\sum_{p=1}^{m}\psi_{i}(x_{p})\,\psi_{j}(x_{p})\,\phi_{p}(x)
\]
 can be substituted into the right-hand side of (\ref{eq:eri}) to
justify the approximation, where we define $V$ and $\Psi$ in (\ref{eq:ERI0})
by 
\[
V_{pq}=\frac{1}{w_{p}w_{q}}\int\phi_{p}(x)\,v_{\mathrm{ee}}(x-y)\,\phi_{q}(y)\,dx\,dy
\]
 and 
\begin{equation}
\Psi_{pi}\approx\tilde{\Psi}_{pi}:=\psi_{i}(x_{p})\,\sqrt{w_{p}}.\label{eq:Psi}
\end{equation}
 By construction $V$ is positive semidefinite. Moreover 
\[
[\tilde{\Psi}^{\top}\tilde{\Psi}]_{ij}=\sum_{p}\psi_{i}(x_{p})\psi_{j}(x_{p})\,w_{p}=\int\sum_{p=1}^{m}\psi_{i}(x_{p})\psi_{j}(x_{p})\,\phi(x)\,dx\approx\int\psi_{i}(x)\,\psi_{j}(x)\,dx=\delta_{ij},
\]
 i.e., $\tilde{\Psi}$ almost has orthonormal columns. Therefore we
can take $\Psi$ to be approximately equal to $\tilde{\Psi}$ but
satisfying $\Psi^{\top}\Psi=\mathbf{I}_{n}$ exactly, e.g., by symmetric
orthogonalization $\Psi=\tilde{\Psi}(\tilde{\Psi}^{\top}\tilde{\Psi})^{-1/2}$.

In our analysis, we have taken the perspective that it is simpler
to assume that the format (\ref{eq:thc1}) holds exactly for the ERI
and perform analysis given this format, rather than work with the
right-hand side of (\ref{eq:eri}) directly. The preceding arguments
justify that this format can be reasonably assumed, but a full analysis
of the approximation is orthogonal to our efforts in this work, as
the THC format is widely used in quantum chemistry calculations \cite{10.1063/1.4732310,10.1063/1.4768233,10.1063/1.4768241}.
Moreover, the THC format is important from an implementation point
of view in that it allows for the construction of the gradient estimator
via the fast construction of the discrete electron density $\hat{\rho}_{t}$,
following (\ref{eq:rhohat}). Indeed, in more general stochastic DFT,
it is important to lift the quantum chemistry basis representation
to a grid-based representation in order to implement more general
density functionals \cite{fabianStochasticDensityFunctional2019}.

Now we want to explain how the quantity $c_{\mathrm{h}}=c_{\Psi}\Vert V\Vert_{\infty}$
appearing in our analysis can be understood physically. For simplicity,
assume (as is the case for our choices of interpolating basis described
above) that $w_{p}=w$ is uniform, that $\frac{1}{w}\int\vert\phi_{q}(x)\vert\,dx=O(1)$,
and that $\sum_{q=1}^{m}\vert\phi_{q}(x)\vert=O(1)$ in the thermodynamic
limit. Also assume that all the basis functions $\{\psi_{i}\}$ are
supported on some computational domain $\mathcal{D}$.

Then define the constants: 
\[
c_{\psi}:=\sup_{x}\left\{ \sum_{i=1}^{n}\vert\psi_{i}(x)\vert^{2}\right\} ,\quad c_{\mathrm{ee}}:=\sup_{x\in\mathcal{D}}\int\left|v_{\mathrm{ee}}(x-y)\right|\,dy.
\]
 Under the reasonable assumption that each of our quantum chemistry
basis functions effectively overlaps with only $O(1)$ other basis
functions in the thermodynamic limit, it is reasonable to assume that
$c_{\psi}=O(1)$ in this limit. Moreover, if $v_{\mathrm{ee}}$ has
rapid decay as in the case of the Yukawa potential, then $c_{\mathrm{ee}}=O(1)$
as well.

Meanwhile, for $\Psi$ as constructed in (\ref{eq:Psi}), $c_{\Psi}\lessapprox w\,c_{\psi}$.
Moreover, under our stated assumptions, $\sum_{q=1}^{m}\vert V_{pq}\vert=O\left(c_{\mathrm{ee}}/w\right).$
Therefore $\Vert V\Vert_{\infty}=O\left(c_{\mathrm{ee}}/w\right)$,
and in turn 
\[
c_{\mathrm{h}}=O(c_{\psi}c_{\mathrm{ee}}).
\]
 Under the assumptions outlined above, this upper bound is $O(1)$
in the thermodynamic limit.

However, if $v_{\mathrm{ee}}$ is the Coulomb potential, $c_{\mathrm{ee}}$
will grow as the volume of the computational domain grows. This is
a fair conclusion since the total Hartree energy grows proportionally
as well, so the error bound that we achieve in Theorem \ref{thm:convergence}
is still accurate in relative terms.

\section{Chemical potential optimization \label{sec:Chemical-potential-optimization}}

In this section we explain how an oracle for the unconstrained problem
(\ref{eq:opt}) can be used to solve the problem (\ref{eq:opt_constraint})
in which the optimization variable $X$ satisfies the particle number
constraint $\Tr[X]=N$, where $N\in[0,n]$. It is useful to define
the \textbf{\emph{filling factor}}
\[
\nu=\frac{N}{n}\in(0,1),
\]
\textbf{\emph{ }}i.e., the number of electrons per basis function. 

Let 
\begin{equation}
g_{\beta}(\mu):=\sup_{0\preceq X\preceq\mathbf{I}_{n}}\{\mu\,\Tr[X]-F_{\beta}(X)\}\label{eq:gbeta}
\end{equation}
 denote the negative of the optimal value of the unconstrained problem
(\ref{eq:opt}) as a function of the chemical potential $\mu$.

Then the dual problem for the constrained problem (\ref{eq:opt_constraint})
is precisely to maximize the concave objective 
\begin{equation}
g_{N,\beta}(\mu):=N\mu-g_{\beta}(\mu)\label{eq:gNmu}
\end{equation}
 over $\mu\in\R$.

We will use our oracle for approximately solving the unconstrained
problem (\ref{eq:opt}) as an approximate evaluator for this objective
$g_{N,\beta}$. We want to show that it will suffice to apply this
objective at a finite collection of values $\mu$. For this to be
the case, we want (1) to show that $g_{N,\beta}$ is Lipschitz and
(2) to determine \emph{a priori }bounds on the location of the maximizer.
These tasks are accomplished via the following two lemmas, which are
proved in Appendix \ref{app:chemical} 
\begin{lem}
\label{lem:glip}$g_{N,\beta}:\R\ra\R$ is Lipschitz with Lipschitz
constant $n$.
\end{lem}

\begin{lem}
\label{lem:mubound} The maximum of $g_{N,\beta}$ is attained within
the interval 
\[
\left[\lambda_{\min}(C)-c_{\mathrm{h}}-\beta^{-1}\log(\gamma^{-1}),\ \lambda_{\max}(C)+c_{\mathrm{h}}+\beta^{-1}\log([1-\gamma]^{-1})\right].
\]
 This holds even for $\beta=+\infty$, interpreting the terms involving
$\beta$ as zero.
\end{lem}

Then for positive integer $M$, let $\mu_{k}$, $k=0,\ldots,K$, denote
equispaced points with 
\[
\mu_{0}=\lambda_{\min}(C)-c_{\mathrm{h}}-\beta^{-1}\log(\gamma^{-1}),\quad\mu_{K}=\lambda_{\max}(C)+c_{\mathrm{h}}+\beta^{-1}\log([1-\gamma]^{-1}).
\]
 We can solve (\ref{eq:opt}) with $\mu=\mu_{k}$, for each $k=0,\ldots,K$
using algorithm (\ref{eq:update}), with step size chosen as in Theorem
\ref{thm:convergence} for a given failure probability $\delta/(K+1)$,
so that by the union bound we are guaranteed that the estimate of
Theorem \ref{thm:convergence} succeeds for all $k$ with probability
at least $1-\delta$. For each $k$, use the Cesàro mean $\overline{X}_{T}$
over $T$ iterations (cf. Theorem \ref{thm:convergence}) to produce
an estimate $\hat{g}_{N,\beta,k}\approx g_{N,\beta}(\mu_{k})$. Then
let $\hat{g}_{N,\beta}:=\max_{k=0,\ldots,K}\hat{g}_{N,\beta,k}$ denote
the maximum estimated value.
\begin{prop}
\label{prop:chemopt}Let $p^{\star}$ denote the optimal value of
(\ref{eq:opt_constraint}). The estimate $\hat{g}_{N,\beta}$ furnished
by the above procedure satisfies 
\[
\frac{\vert\hat{g}_{N,\beta}-p^{\star}\vert}{n}=O\left(\frac{c_{\mathrm{h}}\,\log(KTm/\delta)}{\sqrt{T}}+\left[\frac{1}{T}+\frac{1}{M}\right]\left[\Vert C\Vert+c_{\mathrm{h}}+\beta^{-1}\log\left(\frac{1}{\gamma(1-\gamma)}\right)\right]\right)
\]
 holds with probability at least $1-\delta$. Again the term involving
$\beta$ can be ignored if $\beta=+\infty$.
\end{prop}

The proof is also given in Appendix \ref{app:chemical}.

\section{Complete basis set limit \label{sec:cbl}}

In the complete basis set limit, where we consider an increasing number
of basis functions per unit volume, there are two problems with the
preceding analysis. First, the norm $\Vert C\Vert$ becomes unbounded
due to the fact that the Galerkin projection of the Laplacian operator
becomes unbounded in the complete basis set limit. Second, the size
$n$ of the basis set no longer remains within a fixed proportion
of the number of electrons, and the free energy density per basis
function $F_{\beta}/n$ appearing in Theorem \ref{thm:convergence}
is no longer physically meaningful.

Therefore in the complete basis set limit, we hope to provide an alternative
to Theorem \ref{thm:convergence} which avoids any explicit dependence
on $\Vert C\Vert$ or $n$. This will require some additional assumptions.
Specifically, we must assume some sufficient growth in the eigenvalues
of $C$, so that the high-energy eigenstates that arise from the finer
discretization contribute only negligibly. We must also assume that
the temperature is not arbitrarily high, so that entropic effects
do not bestow nontrivial total occupation upon these high-energy eigenstates.

We note that to eliminate the dependence on $\Vert C\Vert$ and $n$
from our analysis, we must pay the price of some dependence on $\beta$
(which in particular we will assume to be finite). We will see that
this dependence arises from the fact that we take an alternative initial
condition $X_{0}$, for which $D(X_{\star}\Vert X_{0})$ depends on
$\beta$.

Before formalizing the assumptions at which we have gestured above,
we will phrase a convergence result more abstractly in terms of the
quantities $D(X_{\star}\Vert X_{0})$ and $\Tr[X_{0}]$. We will also
make the following simplifying assumption, which eases some of the
notation in a few places.

\begin{assump}

\label{assump:veenn} In this section, assume that $\nabla\tilde{E}(X)\succeq0$
for all $X\in\mathcal{X}$.

\end{assump}

Note that this assumption holds automatically for $\tilde{E}$ induced
by a choice of electron-electron potential $v_{\mathrm{ee}}\geq0$,
such as the Yukawa or Coulomb potential. However, the assumption is
not really necessary, and it can be guaranteed by replacing $C\leftarrow C-c_{\mathrm{h}}\mathbf{I}_{n}$
while adopting the compensating change $\tilde{E}(X)\leftarrow\tilde{E}(X)+c_{\mathrm{h}}\Tr[X]$.

Moreover, note that that Assumption \ref{assump:veenn} implies that
our gradient estimator satisfies 
\begin{equation}
\tilde{G}_{t}\succeq0,\quad t=0,\ldots,T-1.\label{eq:Gtnn}
\end{equation}

\subsection{Initial condition and SCF interpretation \label{subsec:cblinit}}

We do not change the mirror descent update rule (\ref{eq:update}),
but we consider an alternative choice of initial condition, namely
\[
X_{0}=f_{\beta}(C-\mu\mathbf{I}_{n}),
\]
 which contrasts with our previous choice of $X_{0}=f_{\beta}(0)=\mathbf{I}_{n}/2$.
Equivalently, we start with effective Hamiltonian 
\[
H_{0}=C-\mu\mathbf{I}_{n}.
\]

It is instructive to recall (\ref{eq:hamiltonianupdate}), i.e., that
the effective Hamiltonian update can be viewed as a convex combination:
\[
H_{t+1}=(1-\beta^{-1}\gamma_{t})H_{t}+\beta^{-1}\gamma_{t}\left(G_{t}-\mu\mathbf{I}_{n}\right).
\]
 Therefore if we define the `effective potential' 
\[
V_{t}:=H_{t}-(C-\mu\mathbf{I}_{n}),
\]
 we see that equivalently the effective potential satisfies the update
rule 
\begin{equation}
V_{t+1}=(1-\beta^{-1}\gamma_{t})V_{t}+(\beta^{-1}\gamma_{t})\tilde{G}_{t},\label{eq:effVupdate}
\end{equation}
 where we recall that $\tilde{G}_{t}$ denotes our estimator for the
Hartree potential $\nabla\tilde{E}(X_{t})$. Since $H_{0}=C-\mu\mathbf{I}_{n}$,
the corresponding initial condition for this update is 
\[
V_{0}=0.
\]
 The effective Hamiltonian is recovered in terms of $V_{t}$ via 
\[
H_{t}=C-\mu\mathbf{I}_{n}+V_{t}.
\]

Therefore, the mirror descent update with this choice of initial condition
coincides \emph{precisely }with the ordinary self-consistent field
(SCF) iteration \cite{fabianStochasticDensityFunctional2019} for
the finite-temperature Hartree approximation with simple mixing, where
the mixing parameter is $\beta^{-1}\gamma_{t}\in(0,1)$, modulo the
replacement of the Hartree potential with its stochastic estimator.

Note that the following lemma is immediate from (\ref{eq:effVupdate}),
by the same logic as in the proof of Lemma \ref{lem:EffHambound}: 
\begin{lem}
\label{lem:Veffbound} $0\preceq V_{t}\preceq\tilde{g}_{\max}\,\mathbf{I}_{n}$
for all $t=0,\ldots,T-1$.
\end{lem}

\begin{proof}
This follows via (\ref{eq:effVupdate}) from induction together with
(\ref{eq:Gtnn}).
\end{proof}

\subsection{Key definitions and lemmas \label{subsec:cbl_lemmas}}

Before proceeding with the convergence proof, we need a generalization
of the strong convexity result (Lemma \ref{lem:stronglyconvex}) which
supplies a sharper strong convexity parameter for $S_{\mathrm{FD}}(X)$
according to the trace of $X$ (which shall remain bounded in the
complete basis set limit for fixed volume).

Accordingly, we define 
\[
\mathcal{X}_{\tau}:=\{X:\,0\preceq X\preceq\mathbf{I}_{n},\quad\Tr[X]\leq\tau\},
\]
 so $\mathcal{X}_{\tau}\subset\mathcal{X}_{n}=\mathcal{X}$. Then
we prove: 
\begin{lem}
\label{lem:stronglyconvextau}$S_{\mathrm{FD}}$ is $(1/\tau)$-strongly
convex on $\mathcal{X_{\tau}}$ with respect to the nuclear norm $\Vert\,\cdot\,\Vert_{*}$.
\end{lem}

The proof is given in Appendix \ref{app:cbl}.

In the complete basis set limit, we also cannot rely on $c_{\mathrm{h}}=\Vert V\Vert_{\infty}c_{\Psi}$
(\ref{eq:ch}) being bounded by a constant. The reason is that although
it is possible to maintain that $c_{\Psi}$ is bounded by a constant
in the complete basis set limit, in fact $\Vert V\Vert_{\infty}$
scales with the density of basis functions per unit volume. (A concrete
illustration will be offered in Section \ref{subsec:Hartree-contribution}
below.) In fact, the main importance of $c_{\mathrm{h}}$ in our arguments
above is in bounding terms related to the Hartree potential $\nabla\tilde{E}(X)$.
In our preceding analysis, we use the trivial bound $\rho(X)\leq1$
to bound the potential in terms of $c_{\mathrm{h}}$. However, in
the complete basis set limit (with a fixed number of electrons per
unit volume), we expect that the occupation of any individual site
will scale inversely with the density of basis functions per unit
volume, so that the growth of $\Vert V\Vert_{\infty}$ and the decay
of $\rho(X)$ compensate for each other.

Therefore it is more natural to rely on the quantity 
\begin{equation}
\tilde{c}_{\mathrm{h}}:=\max\left(\Vert\,\vert V\vert\,\rho(X_{\star})\Vert_{\infty},\ \max_{t=0,\ldots,T-1}\Vert\,\vert V\vert\,\rho(X_{t})\Vert_{\infty}\right).\label{eq:chtilde}
\end{equation}
 Here $\vert V\vert$ denotes the entrywise absolute value of $V$.
For a non-negative electron-electron interaction $v_{\mathrm{ee}}\geq0$
and an interpolating basis as we shall consider in Section \ref{subsec:Hartree-contribution},
naturally we will have $\vert V\vert=V$. Therefore following the
interpretation of Section \ref{subsec:Electron-repulsion-integrals},
$\tilde{c}_{h}$ can be viewed as a bound on $L^{\infty}$ norm of
the spatial Hartree potential (evaluated on a set of interpolating
points) over the optimization trajectory.

We will not provide any \emph{a priori }bounds for this quantity,
which would rely on more detailed analysis involving the composition
of the single-particle matrix $C$ as a sum of kinetic and potential
terms (cf. Section \ref{subsec:Hartree-contribution}). If the external
potential $v_{\mathrm{ext}}$ is assumed to be bounded (as could be
guaranteed via the use of a nuclear pseudopotential \cite{Gygi}),
then since the Hartree potential is also uniformly bounded (via Lemma
\ref{lem:Veffbound}) over the optimization trajectory, the effective
Hamiltonian $H_{t}$ would consist of a kinetic (Laplacian term) together
with a bounded effective potential. We conjecture that for reasonable
choices of basis (such as the periodic sinc basis considered in Section
\ref{sec:Experiments} below), such assumptions would suffice for
a suitable \emph{a priori }bound for $\rho(X_{t})$ over the optimization
trajectory. However, such an analysis will take us outside the intended
scope of our results, which we intend to phrase more naturally in
terms of physical quantities and with a minimum of detailed assumptions.
In any case, with or without \emph{a priori }bounds, the quantity
(\ref{eq:chtilde}) is the physically relevant quantity to be featured
in the analysis.

Before proceeding with the convergence proof, we need to rephrase
some of our earlier lemmas more delicately in terms of $\tilde{c}_{h}$.

First, we modify Lemma \ref{lem:Gestbound} as follows: 
\begin{lem}
\label{lem:Gestbound-1} For any $\delta\in(0,1]$, the inequality
\[
\max_{t=0,\ldots,T-1}\Vert\hat{G}_{t}\Vert\leq2\left(1+4\log(Tm/\delta)\right)\tilde{c}_{\mathrm{h}}
\]
 holds with probability at least $1-\delta$.
\end{lem}

Second, we modify Lemma \ref{lem:mdsprep}. We make the statement
more specific, tailored precisely to how we will use the lemma (analogously
to the use of Lemma \ref{lem:mdsprep} in the proof of Theorem \ref{thm:convergence}).
\begin{lem}
\label{lem:mdsprep-1} For all $t=0,\ldots,T-1$, the random variable
$\left\langle \Delta_{t},X_{\star}-X_{t}\right\rangle $ is sub-exponential
with parameters $(4\tilde{c}_{\mathrm{h}}\Vert X_{t}\Vert_{\mathrm{F}},8\tilde{c}_{\mathrm{h}})$,
conditioned on $\mathcal{F}_{t-1}$.
\end{lem}

The proofs of both lemmas are in Appendix \ref{app:cbl}.

\subsection{Convergence theorem and proof \label{subsec:cbl_convergence}}
\begin{thm}
\label{thm:convergence-1} For any $\delta\in(0,1]$, define $c_{T,m,\delta}:=2\left(1+4\log(2Tm/\delta)\right).$
Also let $\tau:=\max\left(\Tr[X_{0}],1\right)$, where $X_{0}=f_{\beta}(C-\mu\mathbf{I}_{n})$
is the initial condition. Consider algorithm defined by (\ref{eq:update2})
with step size $\eta_{t}=\eta:=\sqrt{\frac{2D(X_{\star}\Vert X_{0})}{\tau\,c_{T,m,\delta}^{2}\,\tilde{c}_{\mathrm{h}}^{2}\,T}}$,
or equivalently by (\ref{eq:update}) with $\gamma_{t}:=\frac{\eta\beta}{\eta+\beta}$.
Then 
\[
\frac{1}{T}\sum_{t=0}^{T-1}\left(F_{\beta}(X_{t})-\mu\Tr[X_{t}]\right)\leq F_{\beta}(X_{\star})-\mu\Tr[X_{\star}]+\frac{\left(\sqrt{2\tau\,D(X_{\star}\Vert X_{0})}+2\tau\right)c_{T,m,\delta}\,\tilde{c}_{\mathrm{h}}}{\sqrt{T}}
\]
 holds with probability at least $1-\delta$.
\end{thm}

\begin{proof}
As in the proof of Theorem \ref{thm:convergence}, we can assume without
loss of generality that $\mu=0$.

First we claim that 
\[
\Tr[X_{t}]\leq\Tr[X_{0}]
\]
 for all $t$.

To see this, recall that $X_{t}=f_{\beta}(H_{t})$, and moreover $H_{t}=H_{0}+V_{t}\succeq H_{0}$
by Lemma \ref{lem:Veffbound}. Then Weyl's monotonicity theorem guarantees
that 
\[
\lambda_{k}(H_{t})\geq\lambda_{k}(H_{0})
\]
 for all $k$, where $\lambda_{k}$ denotes the $k$-th eigenvalue,
ordered non-decreasingly. Then since $f_{\beta}$ is a decreasing
function, it follows that 
\[
\Tr[X_{t}]=\sum_{k=1}^{n}f_{\beta}(\lambda_{k}(H_{t}))\leq\sum_{k=1}^{n}f_{\beta}(\lambda_{k}(H_{0}))=\Tr[X_{0}],
\]
 as claimed. Since $\tau=\max\left(\Tr[X_{0}],1\right)$ it follows
that $X_{t}\in\mathcal{X}_{\tau}$ for all $t$.

Again for simplicity, we replace $C\leftarrow C-\mu\mathbf{I}_{n}$
to ease the notation in the proof. Then by the exact same argument
as in the proof of Theorem \ref{thm:convergence}, where instead by
way of Lemma \ref{lem:stronglyconvextau} we use $\tau^{-1}$ in the
place of $2/n$ as the strong convexity parameter, we deduce (cf.
(\ref{eq:reflater})) that 
\begin{align*}
\sum_{t=0}^{T-1}\left[F_{\beta}(X_{t})-F_{\beta}(X_{\star})\right] & \leq\frac{1}{\eta}\left[d_{0}-d_{T}\right]+\frac{\tau}{2}\eta\sum_{t=0}^{T-1}\Vert\tilde{G}_{t}\Vert^{2}+\sum_{t=0}^{T-1}\left\langle \Delta_{t},X_{\star}-X_{t}\right\rangle \\
 & \quad\quad+\ \left\langle C,X_{0}-X_{T}\right\rangle +\frac{1}{\beta}\left[S_{\mathrm{FD}}(X_{0})-S_{\mathrm{FD}}(X_{T})\right].
\end{align*}
 Importantly, we can now group the last two terms in a meaningful
way due to our choice of initial condition. Indeed, since $C=\frac{1}{\beta}\nabla S_{\mathrm{FD}}(X_{0})$,
we can compute: 
\begin{align*}
 & \left\langle C,X_{0}-X_{T}\right\rangle +\frac{1}{\beta}\left[S_{\mathrm{FD}}(X_{0})-S_{\mathrm{FD}}(X_{T})\right]\\
 & \quad\quad=\ -\frac{1}{\beta}D(X_{T}\Vert X_{0}),
\end{align*}
 which is in particular $\leq0$.

Again, $d_{T}\geq0$, so we deduce 
\[
\sum_{t=0}^{T-1}\left[F_{\beta}(X_{t})-F_{\beta}(X_{\star})\right]\leq\frac{d_{0}}{\eta}+\frac{\tau}{2}\eta\sum_{t=0}^{T-1}\Vert\tilde{G}_{t}\Vert^{2}+\sum_{t=0}^{T-1}\left\langle \Delta_{t},X_{\star}-X_{t}\right\rangle .
\]
 The bounding of the last two terms proceeds similarly to the proof
of Theorem \ref{thm:convergence}, except that we do not use the trivial
bound $\Vert X_{t}\Vert_{\mathrm{F}}\leq\sqrt{n}$ to control the
sub-exponentiality parameters of $\sum_{t=0}^{T-1}\left\langle \Delta_{t},X_{\star}-X_{t}\right\rangle $.

Indeed, carrying forward the argument without inserting this bound,
we deduce from Lemma \ref{lem:mdsprep-1} and Theorem \ref{thm:mds}
that $\sum_{t=0}^{T-1}\left\langle \Delta_{t},X_{\star}-X_{t}\right\rangle $
is sub-exponential with parameters $\left(4\sqrt{T}\Vert X_{t}\Vert_{\mathrm{F}}\,\tilde{c}_{\mathrm{h}},8\tilde{c}_{\mathrm{h}}\right)$.
Now $\Vert X_{t}\Vert_{\mathrm{F}}\leq\Tr[X_{t}]$ since $X_{t}\succeq0$,
and in turn $\Tr[X_{t}]\leq\tau$. Therefore, since $\tau\geq1$,
we can loosen the sub-exponentiality parameters to $\left(4\tau\sqrt{T}\,\tilde{c}_{\mathrm{h}},8\tau\sqrt{T}\,\tilde{c}_{\mathrm{h}}\right)$
and apply Corollary \ref{cor:subexp} to deduce that 
\[
\sum_{t=0}^{T-1}\left\langle \Delta_{t},X_{\star}-X_{t}\right\rangle \leq2\tau c_{T,m,\delta}\tilde{c}_{\mathrm{h}}\sqrt{T}
\]
 with probability at least $1-\delta/2$. (Compare with (\ref{eq:in8}).
Relative to that inequality, we have simply replaced $\sqrt{n}$ in
the right-hand side with $\tau$ and $c_{\mathrm{h}}$ with $\tilde{c}_{h}$.)

Similarly to the proof of Theorem \ref{thm:mds} (except that we use
Lemma \ref{lem:Gestbound-1} in the place of Lemma \ref{lem:Gestbound}),
we can bound 
\[
\sum_{t=0}^{T-1}\Vert\tilde{G}_{t}\Vert^{2}\leq c_{T,m,\delta}^{2}\,\tilde{c}_{\mathrm{h}}^{2}\,T
\]
 with probability at least $1-\delta/2$.

Then it follows that 
\[
\sum_{t=0}^{T-1}\left[F_{\beta}(X_{t})-F_{\beta}(X_{\star})\right]\leq\frac{d_{0}}{\eta}+\frac{\tau}{2}\eta\,c_{T,m,\delta}^{2}\,\tilde{c}_{\mathrm{h}}^{2}\,T+2\tau c_{T,m,\delta}\tilde{c}_{\mathrm{h}}\sqrt{T}
\]
 holds with probability at least $1-\delta$.

By choosing 
\[
\eta=\sqrt{\frac{d_{0}}{\frac{\tau}{2}c_{T,m,\delta}^{2}\,\tilde{c}_{\mathrm{h}}^{2}\,T}},
\]
 we complete the proof. 
\end{proof}

\subsection{Control in terms of eigenvalue growth \label{subsec:cbl_eig}}

Note that in Theorem \ref{thm:convergence-1}, the explicit dependence
on $\Vert C\Vert$ and $n$ have been eliminated. However, we want
to argue that $\Tr[X_{0}]$ and $D(X_{\star}\Vert X_{0})$ remain
bounded independent of these quantities in the complete basis set
limit.

In order to control these quantities, we must assume some model of
eigenvalue growth for the single-particle Hamiltonian matrix $C$.
Throughout we let $\lambda_{k}(A)$, $k=1,\ldots,n$, denote the $k$-th
eigenvalue of a symmetric matrix $A$, counted in non-decreasing order.
Then the eigenvalue growth assumption is as follows.

\begin{assump}

\label{assump:growth}Assume that $\lambda_{k}(C)\geq(k/c_{\lambda})^{\alpha}$
for some constant $\alpha,c_{\lambda}>0$.

\end{assump}

Note that this assumption is naturally satisfied (after a suitable
scalar shift) with $\alpha=2/d_{\mathrm{eff}}$ for quasi-$d_{\mathrm{eff}}$-dimensional
systems due to the growth of the eigenvalues of the Laplacian, in
both the thermodynamic and complete basis set limits. Intuitively,
we can think of $c_{\lambda}$ as proportional to the $d_{\mathrm{eff}}$-dimensional
volume or the number of atoms in the problem. We will not pursue a
detailed analysis guaranteeing the assumption since it is orthogonal
to the content of this work.

It is somewhat cumbersome to deal with the case of general $\alpha$,
so for simplicity we simply take $\alpha=1$ in the ensuing analysis.
The numerical constants that we obtain are not so physically relevant
anyway---the asymptotic eigenvalue growth kicks in once we have entered
the scattering part of the spectrum, while in practice we are mostly
interested in a chemical potential that targets the bound states.
Our main goal is simply to demonstrate the dependence of the quantities
$\Tr[X_{0}]$ and $D(X_{\star}\Vert X_{0})$ on $\beta$, $c_{\mathrm{h}}$,
and $c_{\lambda}$, as well as to verify their independence from $\Vert C\Vert$
and $n$. The dependence on $\mu$ that we obtain is less physically
relevant. For these reasons, we are content to adopt the simplifying
assumption $\alpha=1$. The heuristic conclusions do not change in
the general case.

Now we state two lemmas bounding $\Tr[X_{0}]$ and $D(X_{\star}\Vert X_{0})$,
respectively. The proofs are given in Appendix \ref{app:cbl}.
\begin{lem}
\label{lem:TrX0}Suppose that Assumption \ref{assump:growth} holds
with $\alpha=1$. Then 
\[
\Tr[X_{0}]\leq c_{\lambda}(\mu+\beta^{-1})+1.
\]
\label{lem:DXstarX0} Suppose that Assumption \ref{assump:growth}
holds with $\alpha=1$. Then 
\[
D(X_{\star}\Vert X_{0})\leq c_{\lambda}\left[(2\beta c_{\mathrm{h}}+1)\mu+4\beta^{-1}\right]+2.
\]
\end{lem}

These two lemmas, together with Theorem \ref{thm:convergence-1},
immediately imply the following.
\begin{cor}
\label{cor:convergence} Suppose that Assumption \ref{assump:growth}
holds with $\alpha=1$. Also suppose $\beta\geq1$. Under the same
hypotheses as in the statement of Theorem \ref{thm:convergence-1},
the inequality 
\[
\frac{1}{T}\sum_{t=0}^{T-1}\left(F_{\beta}(X_{t})-\mu\Tr[X_{t}]\right)\leq F_{\beta}(X_{\star})-\mu\Tr[X_{\star}]+O\left(\frac{\log(Tm/\delta)\,(1+\sqrt{\beta c_{\mathrm{h}}})\,\tilde{c}_{\mathrm{h}}\,n_{\mathrm{eff}}}{\sqrt{T}}\right)
\]
 holds with probability at least $1-\delta$, where $n_{\mathrm{eff}}:=1+c_{\lambda}[\mu+\beta^{-1}]$.
\end{cor}

\begin{rem}
Following Assumption \ref{assump:growth}, it is reasonable to think
of $n_{\mathrm{eff}}$ as something akin to an effective basis set
size which remains in proportion with the particle number. Importantly,
it remains proportional to the volume (i.e., proportional to $c_{\lambda}$)
in the thermodynamic and complete basis set limits (even taken simultaneously).
\end{rem}

\section{Numerical experiments \label{sec:Experiments}}

Now we describe several numerical experiments supporting our theory. 

We will consider functions on the box domain $\mathcal{D}:=\prod_{j=1}^{d}[0,L_{j}]^{d}$
where $d$ is the spatial dimension and $L_{1},\ldots,L_{d}$ are
the box dimensions, with periodic boundary conditions. We will consider
a Galerkin discretization in the periodic sinc (or planewave dual)
basis \cite{boyd_chebyshev_2001} for this periodic domain. This basis
is equivalent in its span to a suitable planewave basis, but the interpolating
property of the periodic sinc basis will allow us to represent diagonal
potentials more conveniently via a pseudospectral approximation \cite{boyd_chebyshev_2001}.

\subsection{Domain and basis \label{subsec:Domain-and-basis}}

Adopt the notation $[j]:=\{0,\ldots,j-1\}$ for arbitrary non-negative
integer $j$. 

Now let $n_{i}$ denote the number of grid points / basis functions
in dimension $i=1,\ldots,d$. Accordingly let $\mathcal{I}:=\prod_{i=1}^{d}[n_{i}]$
denote the indexing set for our computational grid, and define grid
points 
\[
\mathbf{x}_{\mathbf{j}}:=\mathbf{j}\Delta\mathbf{x}=(j_{1}\Delta x_{1},\ldots,j_{d}\Delta x_{d})
\]
 indexed by $\mathbf{j}\in\mathcal{I}$. Here $\Delta x_{i}=\frac{L_{i}}{n_{i}}$
is the mesh size in each dimension, $\Delta\mathbf{x}=(\Delta x_{1},\ldots,\Delta x_{d})$,
and $n_{i}$ is an \emph{odd }integer number of discretization points
per dimension.

It is also useful to define $\mathcal{V}:=\prod_{i=1}^{d}L_{i}$ (the
domain volume), $n=\prod_{i=1}^{d}n_{i}$ (the total number of grid
points / basis functions), $\mathbf{n}=(n_{1},\ldots,n_{d})$, and
$\Delta\mathcal{V}=\mathcal{V}/n$ (the discrete volume element).

Next, we define the planewaves planewaves adapted to this periodic
box by 
\[
e_{\mathbf{k}}(\mathbf{x}):=\frac{1}{\sqrt{\mathcal{V}}}e^{2\pi i\mathbf{k}\cdot(\mathbf{x}/\mathbf{L})},\quad\mathbf{k}\in\hat{\mathcal{I}},
\]
 where our dual indexing set $\hat{\mathcal{I}}$ is defined by 
\[
\hat{\mathcal{I}}:=\prod_{i=1}^{d}\{-\ell_{i},\ldots,\ell_{i}\},\quad\ell_{i}:=(n_{i}-1)/2
\]
 and vector quotients are interpreted elementwise. The set $\{e_{\mathbf{k}}\}_{\mathbf{k}\in\hat{\mathcal{I}}}$
is orthonormal with respect to the $L^{2}(\mathcal{D})$ inner product,
which we denote by $\left\langle \,\cdot\,,\,\cdot\,\right\rangle $.

Then the periodic sinc basis $\{\psi_{\mathbf{j}}\}_{\mathbf{j}\in\mathcal{I}}$
is defined by periodic translation of a reference basis function:
\[
\psi_{\mathbf{j}}(\mathbf{x})=\psi_{\mathbf{0}}(\mathbf{x}-\mathbf{j}\Delta\mathbf{x}),
\]
 where the reference is in turn defined as 
\[
\psi_{\mathbf{0}}(\mathbf{x}):=\frac{1}{\sqrt{\Delta\mathcal{V}}}\prod_{i=1}^{d}\frac{\sin(\pi m_{i}x_{i}/L_{i})}{m_{i}\sin(\pi x_{i}/L_{i})}
\]
 and vector products are interpreted elementwise.

The set $\{\psi_{\mathbf{j}}\}_{\mathbf{j}\in\mathcal{I}}$ is also
orthonormal with respect to the $L^{2}(\mathcal{D})$ inner product.
Moreover, these functions can be written in terms of planewaves via
the unitary transformation: 
\[
\psi_{\mathbf{j}}=\frac{1}{\sqrt{n}}\sum_{\mathbf{k}\in\hat{\mathcal{I}}}e^{-2\pi i\mathbf{k}\cdot(\mathbf{j}/\mathbf{n})}e_{\mathbf{k}}.
\]

Finally, $\{\psi_{\mathbf{j}}\}$ can be rescalled to obtain a basis
\[
\phi_{\mathbf{j}}:=\psi_{\mathbf{j}}\sqrt{\Delta\mathcal{V}},
\]
 which is an interpolating basis for the grid $\{x_{\mathbf{j}}\}$
in the sense that 
\[
f(x)=\sum_{\mathbf{j}\in\mathcal{I}}f(\mathbf{x}_{\mathbf{j}})\,\phi_{\mathbf{j}}(x)
\]
 for $f$ lying in the span of trigonometric polynomials (suitably
scaled to our periodic domain) of multivariate order up to $\left(\frac{n_{1}-1}{2},\ldots,\frac{n_{d}-1}{2}\right)$.
In particular, our periodic sinc basis includes all such planewaves
in its complex span.

\subsection{Single-particle matrix \label{subsec:Single-particle-matrix}}

The single-particle matrix (cf. Section \ref{subsec:Single-electron-term})
\[
C=K+U
\]
 is constructed from the discretizations of two terms: (1) the kinetic
term and (2) the diagonal external potential $v_{\mathrm{ext}}$.

The kinetic matrix $K$ is produced by direct Galerkin projection.
We have 
\[
K_{\mathbf{j},\mathbf{j}'}=-\frac{1}{2}\int_{\mathcal{D}}\psi_{\mathbf{j}}(x)\Delta\psi_{\mathbf{j}'}(x)\,dx.
\]
 Now the Laplacian is translation-invariant, so $K_{\mathbf{j},\mathbf{j}'}=K_{\mathbf{j}-\mathbf{j}',0}$
is determined by a single row (or column) and can be diagonalized
by the $d$-dimensional unitary discrete Fourier transform $\mathcal{F}$
(with dual indexing set $\hat{\mathcal{I}}$). One can compute: 
\[
K=\frac{1}{2}\mathcal{F}D\mathcal{F}^{*}
\]
 where $D=\mathrm{diag}(d)$ is a diagonal matrix, defined by 
\[
d_{\mathbf{k}}=\sum_{i=1}^{d}(2\pi k_{i}/L_{i})^{2},\quad\mathbf{k}\in\hat{\mathcal{I}}.
\]

Meanwhile, the external potential is discretized within the pseudospectral
approximation \cite{boyd_chebyshev_2001} via $U=\mathrm{diag}(u)$
where 
\[
u_{\mathbf{j}}=v_{\mathrm{ext}}(\mathbf{x}_{\mathbf{j}}),\quad\mathbf{j}\in\mathcal{I}.
\]

\subsection{Hartree contribution \label{subsec:Hartree-contribution}}

It remains to describe how the ERI (\ref{eq:ERI0}) are constructed.
In our case, the `interpolating basis' size $m$ (cf. Section \ref{subsec:Electron-repulsion-integrals})
coincides with the basis set size $n$.

We follow the construction and notation outlined in Section \ref{subsec:Electron-repulsion-integrals},
only changing the indexing notation to coincide with our multi-indexing
convention. Then to specify the ERI (\ref{eq:ERI0}), we need to fix
$\Psi_{\mathbf{p}\mathbf{j}}$ (an $n\times n$ unitary matrix) as
well as $V_{\mathbf{p}\mathbf{q}}$ (an $n\times n$ positive semidefinite
matrix).

It is elementary to verify that the interpolating basis weights $w_{\mathbf{p}}=\int_{\mathcal{D}}\phi_{\mathbf{p}}(x)\,dx$
are all given by $w_{\mathbf{p}}=\Delta\mathcal{V}$. Moreover, 
\[
\Psi_{\mathbf{p}\mathbf{j}}:=\psi_{\mathbf{j}}(x_{\mathbf{p}})\sqrt{w_{\mathbf{p}}}=\delta_{\mathbf{p}\mathbf{j}}
\]
 is the identity (hence in particular unitary without need for correction),
and finally, still following Section \ref{subsec:Electron-repulsion-integrals},
we take

\[
V_{\mathbf{p}\mathbf{q}}=\frac{1}{\Delta\mathcal{V}}\int_{\mathcal{D}}\psi_{\mathbf{p}}(x)\,v_{\mathrm{ee}}(x-y)\,\psi_{\mathbf{q}}(y)\,dx\,dy.
\]

Now $V$, like $K$, is diagonalized by the unitary discrete Fourier
transform, and we can write 
\[
V=\frac{1}{\Delta\mathcal{V}}\,\mathcal{F}\,\hat{V}\,\mathcal{F}^{*},
\]
 where $\hat{V}=\mathrm{diag}(\hat{v})$.

To implement the Yukawa interaction, we define $\hat{v}$ as 

\[
\hat{v}_{\mathbf{k}}=\frac{\alpha^{2}}{\alpha^{2}+\sum_{i=1}^{d}(2\pi k_{i}/L_{i})^{2}}
\]
 The Coulomb interaction can be recovered by setting $\alpha=0$ and
setting $\hat{v}_{\mathbf{0}}=0$. (Note that more generally, the
$\mathbf{k}=0$ mode can be removed, since this will simply contributed
a constant shift to the Hartree potential.)

Importantly, note that for this choice of basis many of the general
constructions outlined above simplify intuitively. First, we have
simply 
\[
\rho(X)=\mathrm{diag}(X),
\]
 and in turn 
\[
\tilde{E}(X)=\frac{1}{2}\rho(X)^{\top}V\rho(X),\quad\nabla\tilde{E}(X)=\mathrm{diag}^{*}\left[V\rho(X)\right].
\]

Meanwhile, the gradient estimator can be constructed by first defining
\[
\hat{\rho}_{t}=(X_{t}^{1/2}z_{t})^{\odot2}
\]
 and then 
\[
\tilde{G}_{t}=\mathrm{diag}^{*}\left[V\hat{\rho}_{t}\right].
\]

The computational difficulties lie in (1) the matrix-vector multiplication
$X_{t}^{1/2}z_{t}$, which we implement via the contour method (cf.
Appendix \ref{app:contour}) that is robust to both the thermodynamic
and complete basis set limits and (2) the matrix-vector multiplication
$V\hat{\rho}_{t}$, which we implement via FFTs. (Note that to implement
the contour approach, we must in particular call fast matrix-vector
multiplications by $H_{t}$, which can themselves be reduced to diagonal
matrix operations via FFTs.)

\subsection{Numerical results\label{subsec:Numerical-results}}

We conduct experiments running the mirror descent (\ref{eq:hamiltonianupdate})
for the periodic discretization outlined above. In particular, we
examine the algorithm\textquoteright s robustness by varying the box
size, grid resolution, inverse temperature, and spatial dimension.
All of our numerical experiments are implemented in JAX, and the code
is publicly available at 
\[
\texttt{https://github.com/willcai7/MirrorDescent-DFT}
\]

\textbf{Hardware. }Each experiment was carried out using a single
Nvidia A100 GPU, equipped with 80GB of memory.

\textbf{Problem specification. }Our objective function consists of
several terms. We set the Yukawa parameter to $\alpha=0.5$ to define
the Hartree contribution. To generate the external potential, we first
create a random background charge distribution by uniformly sampling
$\lfloor\zeta\mathcal{\mathcal{V}}\rfloor$ points from our grid and
assigning them each an equal unit `charge.' The external potential
is then given by $v_{\text{ext}}=-V\rho_{\text{ext}}$, where $V$
is the discretized Yukawa kernel defined above with $\alpha=0.5$.
In our simulations, we fix $\zeta=1$ and set the chemical potential
to $\mu=0$. We always choose a cubic domain, i.e., $\mathbf{L}=L\mathbf{1}_{d}$
for varying $L$.

\textbf{Mirror descent. }Our simulations adopt an exponential decay
schedule for the stpe size, given by 
\[
\gamma_{t}=\gamma\cdot\exp(-t/1000).
\]
For most simulations, we use a step-size of $\gamma=1$; however,
when $\beta=0.5$, we reduce the step-size to $\gamma=0.5$. The reason
for deviating from the theoretical prescription for the step size
is that in the theory it is more convenient to consider a finite time
horizon that is specified \emph{a priori}. But in practice, it is
easier to manage an infinite time horizon with a decay schedule. We
chose the exponential schedule for its robust practical performance.
Similar considerations often apply in analysis of other mirror descent
algorithms. We run the mirror descent for a maximum of 5000 iterations,
and we initialize with $H_{0}=C-\mu I$.

\textbf{Contour method. }To estimate the gradient for the Hartree
contribution in the objective, we employ the contour method outlined
in Appendix \ref{app:contour}. We refer to a single approximate matrix-vector
multiplication by $f_{\beta}^{1/2}(H_{t})$ as a `contour matvec.'
We set the number of poles to 20. Within this method, we use the BiCGSTAB
linear solver \cite{van1992bi} with a maximum of 1000 iterations
and a tolerance of $10^{-5}$ to solve the linear systems. The preconditioner
is chosen as in Appendix \ref{app:precondition}. Unlike our theoretical
analysis, which uses a single Gaussian sample to define the gradient
estimator in each iteration, our experiments perform an empirical
average over $N_{g}$ Gaussian samples. This allows us to exploit
GPU parallelism for the implementation of a batch of $N_{g}$ contour
matvecs. We typically use $N_{g}=20$, except for our largest experiments
where we use $N_{g}=10$ (cf. Table \ref{tab:time}). We denote the
batch of Gaussian samples at each iteration as $z_{t}\in\R^{n\times N_{g}}$,
where each column of $z_{t}$ represents a single Gaussian sample.

\textbf{Evaluation. }We assess convergence by comparing our results
against the final output of the deterministic SCF algorithm (for problems
that are small enough such that it is tractable) which serves as our
ground truth. When the exact ground truth is available, we report
the relative errors of the density function along with the absolute
errors in the free energy density, Hartree energy density, and number
of electrons per volume. Additionally, we define a \textquotedblleft \textbf{gold
standard}\textquotedblright{} method as a benchmark for evaluating
the convergence of the density functions. The gold standard assumes
access to the exact optimizer $X_{\star}$ and then estimates the
density using the same Gaussian samples $z_{t}$, i.e., it completely
avoids the optimization. As such it is not supposed to define a realistic
scalable algorithm. Rather, we wish to demonstrate that our scalable
algorithm performs similarly to this unrealistic gold standard. Specifically,
the gold standard defines the density by 
\[
\hat{\rho}_{\text{gold},t}=\frac{\text{diag\ensuremath{\big[}\ensuremath{\ensuremath{\sum_{i=1}^{t}}}}\big(X_{\star}^{1/2}z_{t})\cdot(X_{\star}^{1/2}z_{t})^{\top}\big]}{N_{g}\cdot t}.
\]

\textbf{Numerical results.} Figures \ref{fig:1D-limits}, \ref{fig:1D-beta},
\ref{fig:2D-limits}, \ref{fig:2D-beta}, \ref{fig:3D-limits}, and
\ref{fig:3D-beta} display the external potentials, final density
functions, relative density errors, Hartree energy densities, free
energy densities, and electrons per volume across various dimensions,
inverse temperatures, grid resolutions, and box sizes. In all plots,
at each iteration we average the results from the latter half of the
mirror descent iterations completed thus far, in order to define asymptotically
consistent estimators. Several key observations emerge:
\begin{itemize}
\item \textbf{Robustness in two limits.} The mirror descent algorithm shows
strong performance in both the thermodynamic and complete basis set
limits. In Figures \ref{fig:1D-limits}, \ref{fig:2D-limits}, and
\ref{fig:3D-limits}, subplots (a) and (b) highlight its robustness
in the thermodynamic limit, while subplots (b) and (c) confirm its
stability in the complete basis limit.
\item \textbf{Robustness with respect to inverse temperatures and dimensions.}
Figures \ref{fig:1D-beta}, \ref{fig:2D-beta}, and \ref{fig:3D-beta}
support the algorithm\textquoteright s resilience against variations
in inverse temperature, and all figures collectively demonstrate consistent
performance across different dimensions.
\item \textbf{Comparable performance to the \textquotedblleft gold standard\textquotedblright{}
method.} Across all settings, mirror descent performs nearly on par
with the \textquotedblleft gold standard\textquotedblright{} method,
implying that its computational complexity is almost equivalent to
that of estimating the density given the converged optimizer.
\end{itemize}
Furthermore, we applied mirror descent to two large 3D systems, as
shown in Figure \ref{fig:large}. For a grid size of $\mathbf{n}=(101,101,101)$,
the density matrix exceeds $10^{6}\times10^{6}$ in size. While deterministic
SCF becomes intractable at this scale, our stochastic mirror descent
approach converges as efficiently as it does for smaller systems.
This result highlights the capability of stochastic mirror descent
to handle much larger systems than SCF.

\textbf{Wall clock. }Table \ref{tab:time} provides a comprehensive
summary of the scaling of the average wall clock time $T_{\text{vec}}$
of a batch of $N_{g}$ contour matvecs. The averages are computed
by taking an empirical average over the optimization trajectory with
a stride length of 20 iterations. We also report $T_{\text{vec}}/n$
for all simulations. (We report the time for a batch of contour matvecs
rather than individual ones, since the batched BiCGSTAB solver in
JAX exhibits nonlinear scaling with respect to the number $N_{g}$
of Gaussian samples, due to parallelism.) Several key observations
can be made from these results:
\begin{itemize}
\item \textbf{Scaling with inverse temperature:} The data for the simulations
in Figures \ref{fig:1D-beta}, \ref{fig:2D-beta}, and \ref{fig:3D-beta}
suggest that the average contour matvec time scales as $O(\sqrt{\beta})$
with respect to the inverse temperature $\beta.$ Although we cannot
prove this aspect of the scaling theoretically for the contour method,
it matches the scaling conjectured in the introduction.
\item \textbf{Complete basis set limit:} The data for the simulations in
the (a) and (b) subplots of Figures \ref{fig:1D-limits}, \ref{fig:2D-limits},
and \ref{fig:3D-limits} reveal that the average contour matvec time
scales close to $O(n)$ with respect to the number $n$ of grid points
as the grid spacing is refined for a fixed volume. This validates
our notion of optimal scaling in the complete basis set limit.
\item \textbf{Scaling with box size:} The data for simulations in the (b)
and (c) subplots of Figures \ref{fig:1D-limits}, \ref{fig:2D-limits},
and \ref{fig:3D-limits}, along with those in (a) and (b) of Figure
\ref{fig:large}, demonstrate that the average contour matvec time
is largely insensitive to increases in the box size.
\end{itemize}
Overall, these numerical results underscore the efficiency, robustness,
and scalability of the mirror descent algorithm. The favorable scaling
with respect to inverse temperature, grid resolution, and box size
not only highlight the method\textquoteright s computational efficiency
but also its potential for application in large-scale simulations
where time complexity is a critical factor.

{\bfseries{}
\begin{table}
\centering{}%
\begin{tabular}{|c|c|c|c|c|c|c|c|}
\hline 
Figure & $\mathbf{n}$ & $\mathbf{L}$ & $\beta$ & $\gamma$ & $N_{g}$ & $T_{\text{vec}}$ (s) & $T_{\text{vec}}/n$ (s)\tabularnewline
\hline 
\hline 
\ref{fig:1D-limits}.(a) & 6 & 10 & 10 & 1.0 & $20$ & 0.0139 & $2.32\times10^{-3}$\tabularnewline
\hline 
\ref{fig:1D-limits}.(b) & 1281 & 10 & 10 & 1.0 & 20 & 0.0229 & $1.79\times10^{-5}$\tabularnewline
\hline 
\ref{fig:1D-limits}.(c) & 12801 & 100 & 10 & 1.0 & 20 & 0.4565 & $3.57\times10^{-5}$\tabularnewline
\hline 
\ref{fig:1D-beta}.(a) & 12801 & 100 & 0.5 & 0.5 & 20 & 0.0961 & $7.51\times10^{-6}$\tabularnewline
\hline 
\ref{fig:1D-beta}.(b) & 12801 & 100 & 2 & 1.0 & 20 & 0.1971 & $1.54\times10^{-5}$\tabularnewline
\hline 
\ref{fig:1D-beta}.(c) & 12801 & 100 & 40 & 1.0 & 20 & 0.7206 & $5.63\times10^{-5}$\tabularnewline
\hline 
\ref{fig:2D-limits}.(a) & (51,51) & (10,10) & 10 & 1.0 & 20 & 0.0255 & $9.80\times10^{-5}$\tabularnewline
\hline 
\ref{fig:2D-limits}.(b) & (101,101) & (10,10) & 10 & 1.0 & 20 & 0.0807 & $7.91\times10^{-5}$\tabularnewline
\hline 
\ref{fig:2D-limits}.(c) & (101,101) & (100,100) & 10 & 1.0 & 20 & 0.1844 & $1.81\times10^{-5}$\tabularnewline
\hline 
\ref{fig:2D-beta}.(a) & (101,101) & (100,100) & 0.5 & 0.5 & 20 & 0.0614 & $6.02\times10^{-6}$\tabularnewline
\hline 
\ref{fig:2D-beta}.(b) & (101,101) & (100,100) & 2 & 1.0 & 20 & 0.0946 & $9.27\times10^{-6}$\tabularnewline
\hline 
\ref{fig:2D-beta}.(c) & (101,101) & (100,100) & 40 & 1.0 & 20 & 0.5438 & $5.33\times10^{-5}$\tabularnewline
\hline 
\ref{fig:3D-limits}.(a) & (11,11,11) & (10,10,10) & 10 & 1.0 & 20 & 0.0189 & $1.42\times10^{-5}$\tabularnewline
\hline 
\ref{fig:3D-limits}.(b) & (21,21,21) & (10,10,10) & 10 & 1.0 & 20 & 0.0471 & $5.08\times10^{-6}$\tabularnewline
\hline 
\ref{fig:3D-limits}.(c) & (21,21,21) & (30,30,30) & 10 & 1.0 & 20 & 0.0598 & $6.46\times10^{-6}$\tabularnewline
\hline 
\ref{fig:3D-beta}.(a) & (21,21,21) & (30,30,30) & 0.5 & 0.5 & 20 & 0.0341 & $3.68\times10^{-6}$\tabularnewline
\hline 
\ref{fig:3D-beta}.(b) & (21,21,21) & (30,30,30) & 2 & 1.0 & 20 & 0.0418 & $4.51\times10^{-6}$\tabularnewline
\hline 
\ref{fig:3D-beta}.(c) & (21,21,21) & (30,30,30) & 40 & 1.0 & 20 & 0.0929 & $1.00\times10^{-5}$\tabularnewline
\hline 
\ref{fig:large}.(a) & (101,101,101) & (10,10,10) & 10 & 1.0 & 10 & 3.4839 & $3.38\times10^{-6}$\tabularnewline
\hline 
\ref{fig:large}.(b) & (101,101,101) & (100,100,100) & 10 & 1.0 & 10 & 5.2709 & $5.12\times10^{-6}$\tabularnewline
\hline 
\end{tabular}\caption{This table shows the average time $T_{\text{vec}}$ of a batch of
$N_{g}$ contour matvecs, as well as $T_{\text{vec}}/n$, for all
simulations. During each simulation, timings were recorded every 20
iterations and then aggregated to compute averages. \label{tab:time}}
\end{table}
}

\begin{figure}
\centering{}\includegraphics[bb=90bp 70bp 505bp 812bp,clip,width=0.65\textwidth]{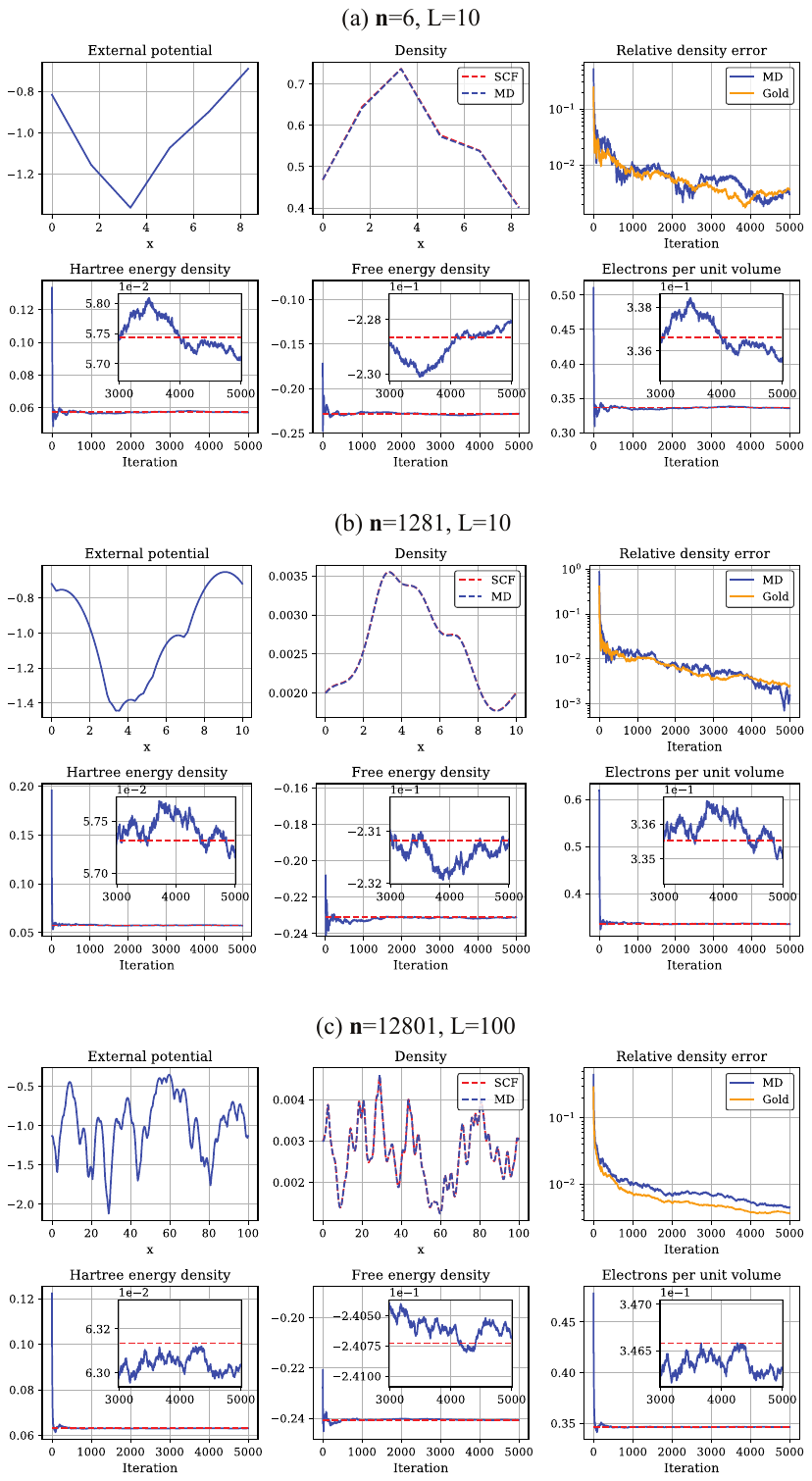}\caption{1D simulation results of the mirror descent algorithm. Here we fix
$\beta=10$ and change the box size $\mathbf{L}$ and the grid size
$\mathbf{n}$. The blue and gold lines denote the results of the mirror
descent and the gold standard, while the red dashed lines denote the
final outputs of the deterministic SCF. \label{fig:1D-limits}}
\end{figure}

\begin{center}
\begin{figure}
\begin{centering}
\includegraphics[bb=90bp 80bp 505bp 812bp,clip,width=0.65\textwidth]{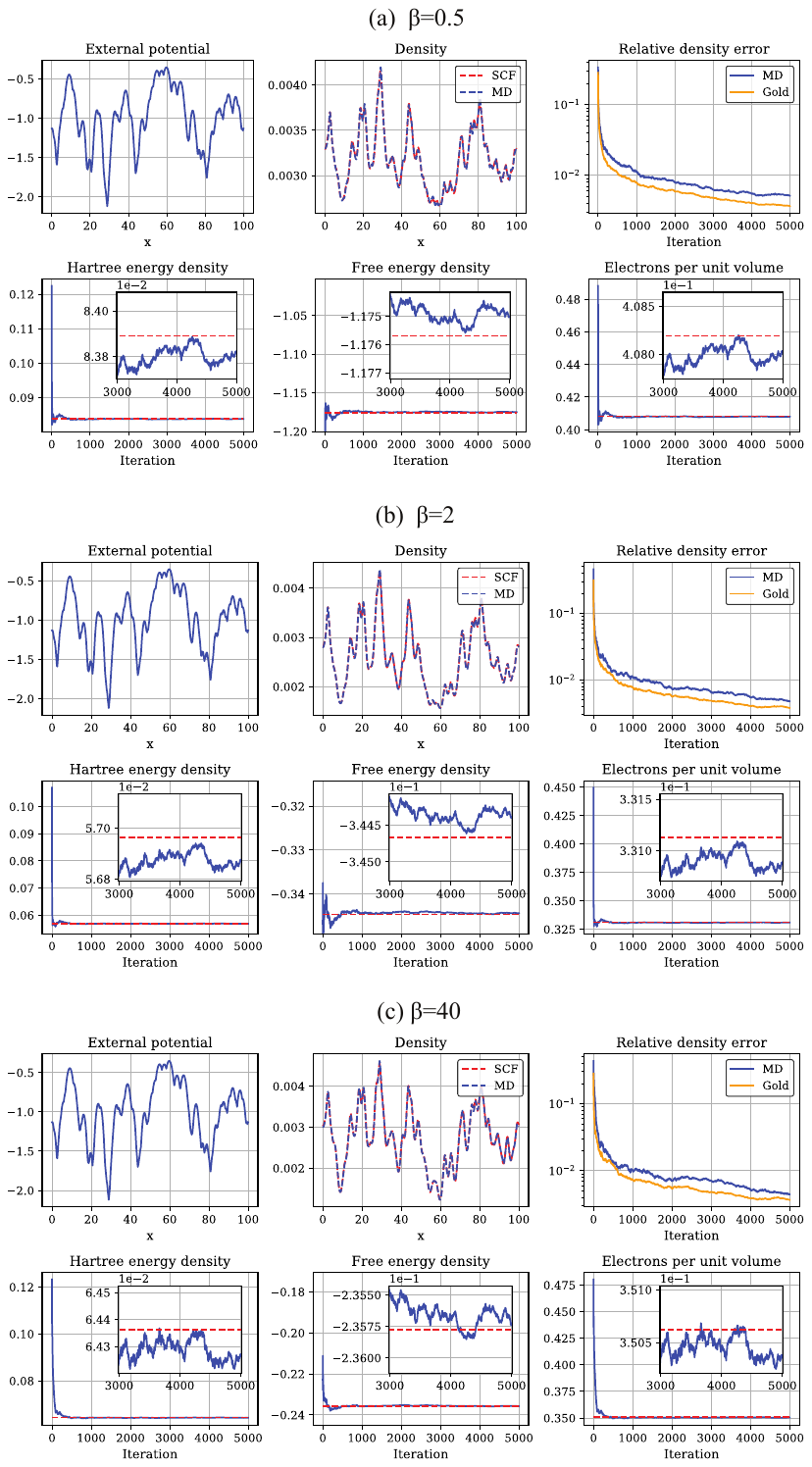}\caption{1D simulation results of the mirror descent algorithm. We fix the
grid size to be $\mathbf{n}=12801$ and the box size to be $\mathbf{L}=100$.
We consider $\beta=0.5,2,40$. The other settings are the same as
in Figure \ref{fig:1D-limits}. \label{fig:1D-beta}}
\par\end{centering}
\end{figure}
\begin{figure}
\begin{centering}
\includegraphics[bb=75bp 240bp 520bp 742bp,clip,width=1\textwidth]{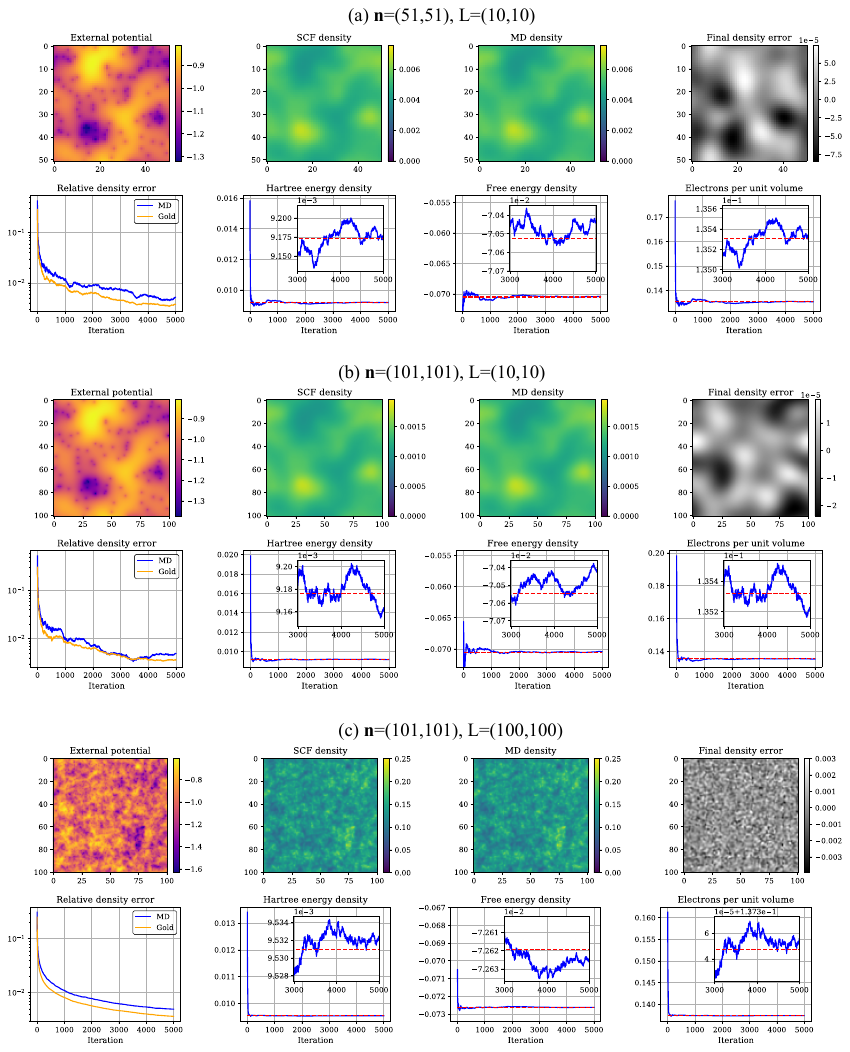}
\par\end{centering}
\centering{}\caption{2D simulation results of the mirror descent algorithm. We fix $\beta=10$
and change the box size $\mathbf{L}$ and grid size $\mathbf{n}$
as indicated in the subplots. We use the red heatmap to show the external
potential, the green heatmaps to show the density functions, and the
gray heatmap to show the absolute error of density functions. The
blue and gold lines denote the results of the mirror descent and the
gold standard, while the red dashed lines denote the final outputs
of the deterministic SCF. \label{fig:2D-limits}}
\end{figure}
\begin{figure}

\centering{}\includegraphics[bb=75bp 230bp 520bp 742bp,clip,width=1\textwidth]{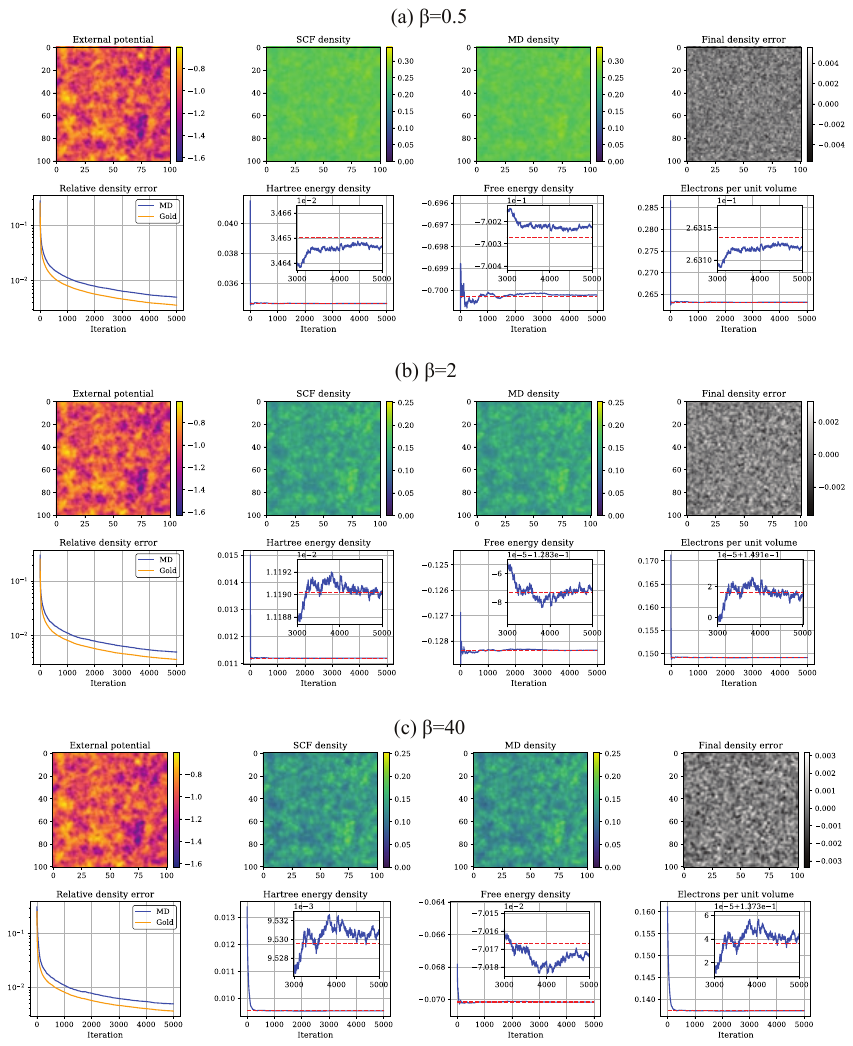}\caption{2D simulation results of the mirror descent algorithm. We fix the
grid size to be $\mathbf{n}=(101,101)$ and the box size to be $\mathbf{L}=(100,100)$.
We consider $\beta=0.5,2,40$. The other settings are the same as
in Figure \ref{fig:2D-limits}. \label{fig:2D-beta}}
\end{figure}
\begin{figure}
\centering{}\includegraphics[bb=75bp 230bp 520bp 752bp,clip,width=1\textwidth]{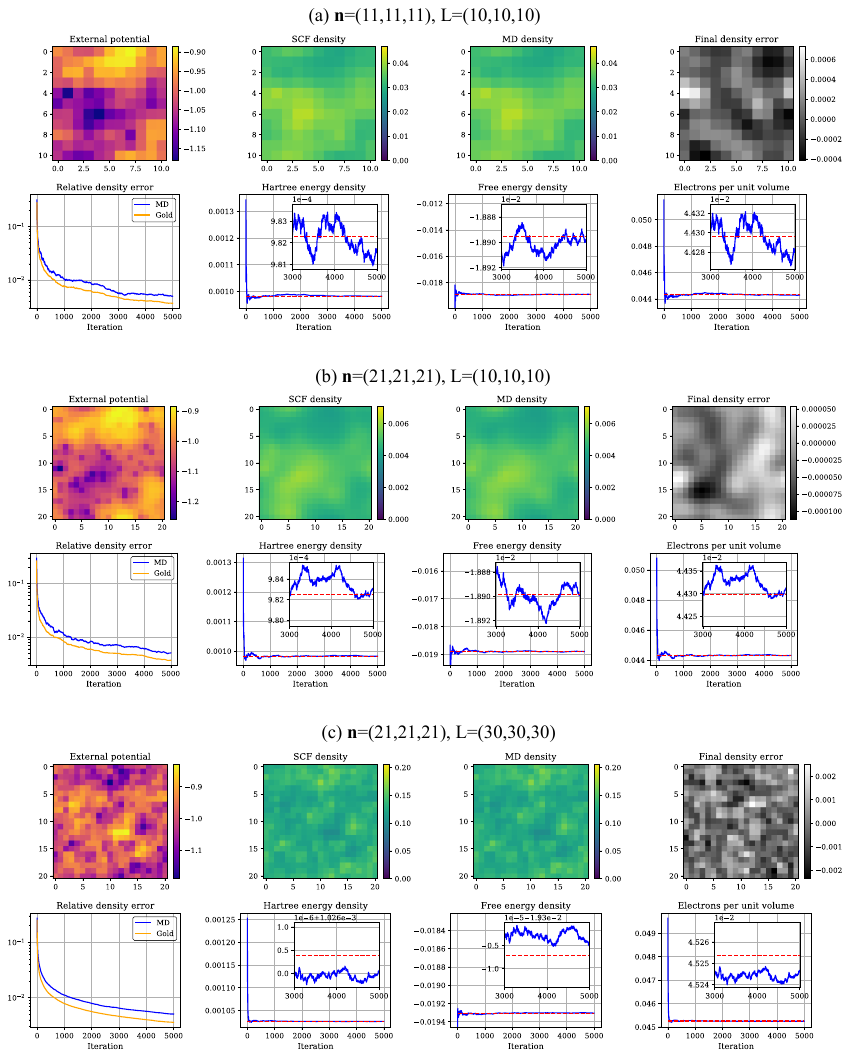}\caption{3D simulation results of the mirror descent algorithm. Here we fix
$\beta=10$ and change the box size $\mathbf{L}$ and grid size $\mathbf{n}$
as indicated in the subplots. For the first row of each subplot, we
show a slice of the external potential, density function, and the
density errors. The other settings are the same as Figure \ref{fig:2D-limits}.
\label{fig:3D-limits} }
\end{figure}
\begin{figure}
\centering{}\includegraphics[bb=75bp 190bp 520bp 712bp,clip,width=1\textwidth]{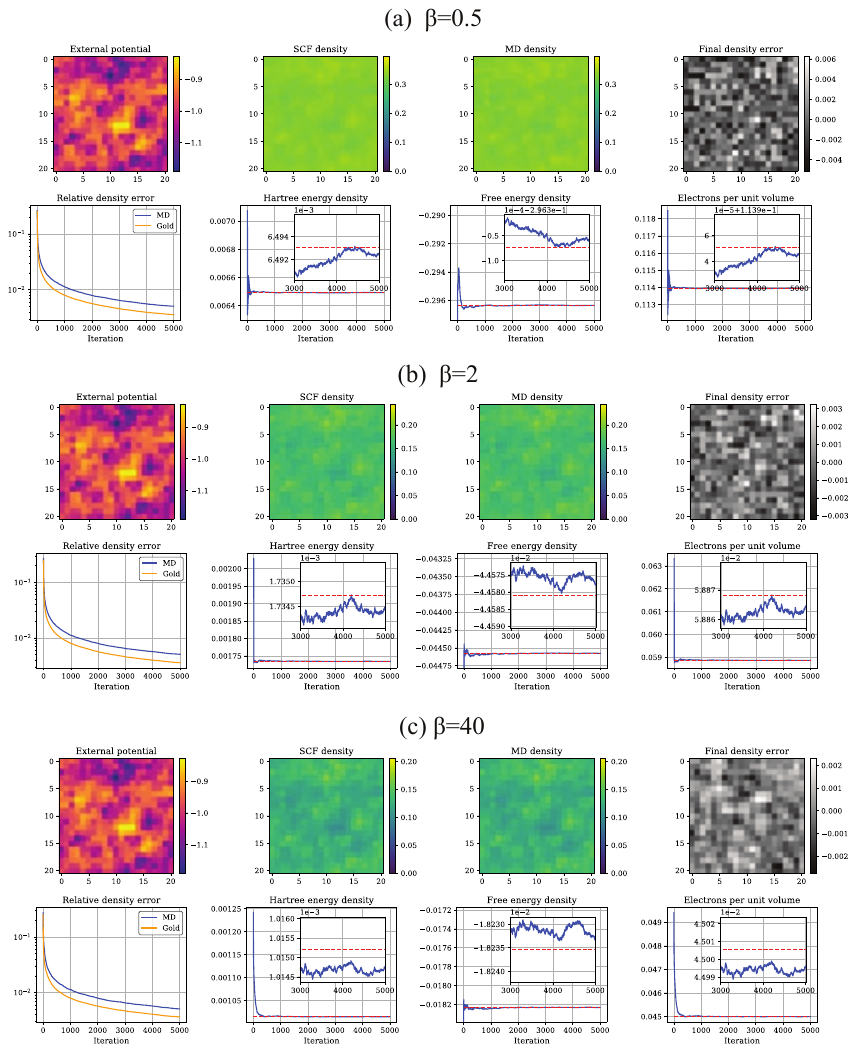}\caption{3D simulation results of the mirror descent algorithm. We fix the
grid size to be $\mathbf{n}=(21,21)$ and the box size to be $\mathbf{L}=(30,30)$.
We consider $\beta=0.5,2,40$. The other settings are the same as
in Figure \ref{fig:3D-limits}. \label{fig:3D-beta}}
\end{figure}
\begin{figure}
\centering{}\includegraphics[bb=75bp 190bp 520bp 702bp,clip,width=1\textwidth]{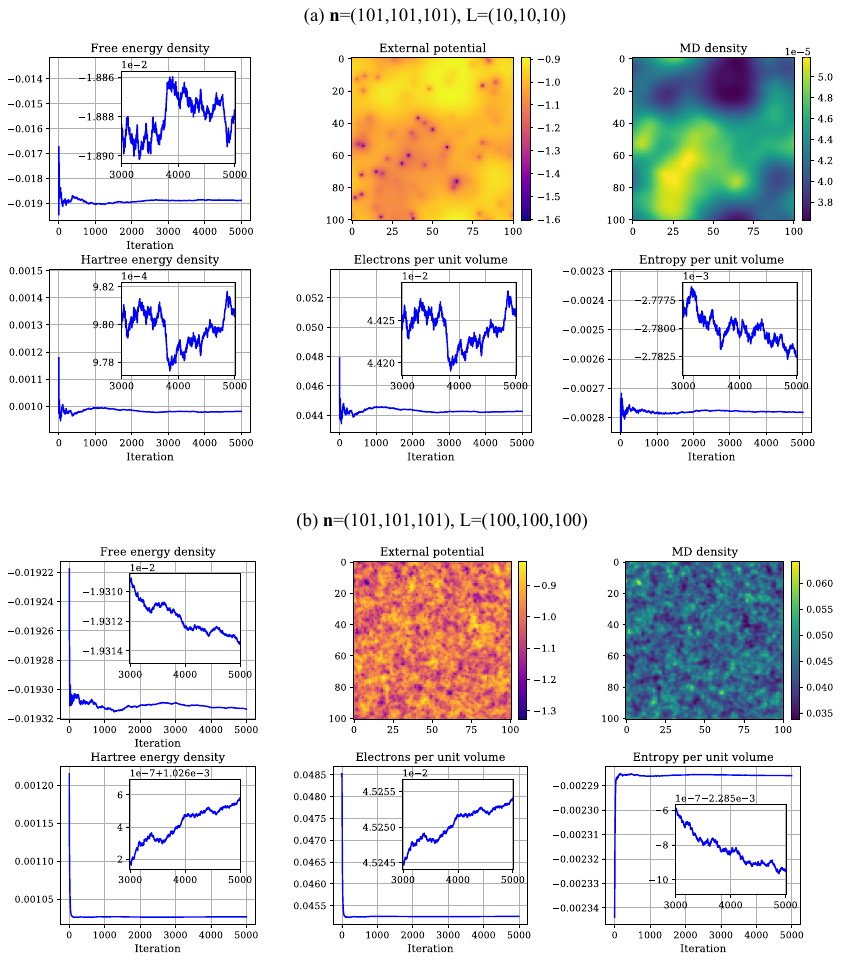}\caption{3D simulation results of the mirror descent algorithm for two large
systems. Here $\beta=10$ and $\mathbf{n}=(101,101,101)$. \label{fig:large} }
\end{figure}
\par\end{center}

\bibliographystyle{plain}
\bibliography{sdft}

\begin{thebibliography}{10}

\bibitem{PhysRevLett.111.106402}
Roi Baer, Daniel Neuhauser, and Eran Rabani.
\newblock Self-averaging stochastic kohn-sham density-functional theory.
\newblock {\em Phys. Rev. Lett.}, 111:106402, Sep 2013.

\bibitem{beck2003mirror}
Amir Beck and Marc Teboulle.
\newblock Mirror descent and nonlinear projected subgradient methods for convex optimization.
\newblock {\em Operations Research Letters}, 31(3):167--175, 2003.

\bibitem{beck2009fast}
Amir Beck and Marc Teboulle.
\newblock A fast iterative shrinkage-thresholding algorithm for linear inverse problems.
\newblock {\em SIAM journal on imaging sciences}, 2(1):183--202, 2009.

\bibitem{boyd_chebyshev_2001}
John~P. Boyd.
\newblock {\em Chebyshev and {Fourier} spectral methods}.
\newblock Dover Publications, Mineola, N.Y, 2nd ed., rev edition, 2001.

\bibitem{10.1063/1.4871575}
Eric~A. Carlen and Elliott~H. Lieb.
\newblock Remainder terms for some quantum entropy inequalities.
\newblock {\em Journal of Mathematical Physics}, 55(4):042201, 04 2014.

\bibitem{chen1993convergence}
Gong Chen and Marc Teboulle.
\newblock Convergence analysis of a proximal-like minimization algorithm using bregman functions.
\newblock {\em SIAM Journal on Optimization}, 3(3):538--543, 1993.

\bibitem{cytterStochasticDensityFunctional2018}
Yael Cytter, Eran Rabani, Daniel Neuhauser, and Roi Baer.
\newblock Stochastic {{Density Functional Theory}} at {{Finite Temperatures}}.
\newblock {\em Physical Review B}, 97(11):115207, March 2018.

\bibitem{Daubechies_1992}
Ingrid Daubechies.
\newblock {\em Ten Lectures on Wavelets}.
\newblock Society for Industrial and Applied Mathematics, January 1992.

\bibitem{PhysRevLett.70.3631}
David~A. Drabold and Otto~F. Sankey.
\newblock Maximum entropy approach for linear scaling in the electronic structure problem.
\newblock {\em Phys. Rev. Lett.}, 70:3631--3634, Jun 1993.

\bibitem{Efron_1982}
Bradley Efron.
\newblock {\em The Jackknife, the Bootstrap and Other Resampling Plans}.
\newblock Society for Industrial and Applied Mathematics, January 1982.

\bibitem{eldowaGeneralTailBounds2024}
Khaled Eldowa and Andrea Paudice.
\newblock General {{Tail Bounds}} for {{Non-Smooth Stochastic Mirror Descent}}.
\newblock In {\em Proceedings of {{The}} 27th {{International Conference}} on {{Artificial Intelligence}} and {{Statistics}}}, pages 3205--3213. PMLR, April 2024.

\bibitem{fabianStochasticDensityFunctional2019}
Marcel~David Fabian, Ben Shpiro, Eran Rabani, Daniel Neuhauser, and Roi Baer.
\newblock Stochastic density functional theory.
\newblock {\em WIREs Computational Molecular Science}, 9(6), November 2019.

\bibitem{Gygi}
Fran{\c c}ois Gygi.
\newblock All-electron plane-wave electronic structure calculations.
\newblock {\em Journal of Chemical Theory and Computation}, 19(4):1300--1309, 02 2023.

\bibitem{hale2008computing}
Nicholas Hale, Nicholas~J Higham, and Lloyd~N Trefethen.
\newblock Computing a\^{}$\alpha$,$\backslash$log(a), and related matrix functions by contour integrals.
\newblock {\em SIAM Journal on Numerical Analysis}, 46(5):2505--2523, 2008.

\bibitem{10.1063/1.4732310}
Edward~G. Hohenstein, Robert~M. Parrish, and Todd~J. Mart\'{i}nez.
\newblock Tensor hypercontraction density fitting. {I. Quartic scaling second- and third-order M{\o}ller-Plesset perturbation theory}.
\newblock {\em The Journal of Chemical Physics}, 137(4):044103, 07 2012.

\bibitem{10.1063/1.4768241}
Edward~G. Hohenstein, Robert~M. Parrish, C.~David Sherrill, and Todd~J. Mart\'{i}nez.
\newblock Communication: Tensor hypercontraction. {III. Least-squares tensor hypercontraction for the determination of correlated wavefunctions}.
\newblock {\em The Journal of Chemical Physics}, 137(22):221101, 12 2012.

\bibitem{hutchinson1989stochastic}
Michael~F Hutchinson.
\newblock A stochastic estimator of the trace of the influence matrix for laplacian smoothing splines.
\newblock {\em Communications in Statistics-Simulation and Computation}, 18(3):1059--1076, 1989.

\bibitem{KohnSham}
W.~Kohn and L.~J. Sham.
\newblock Self-consistent equations including exchange and correlation effects.
\newblock {\em Phys. Rev.}, 140:A1133--A1138, Nov 1965.

\bibitem{lanOptimalMethodStochastic2012}
Guanghui Lan.
\newblock An optimal method for stochastic composite optimization.
\newblock {\em Mathematical Programming}, 133(1-2):365--397, June 2012.

\bibitem{li2022high}
Shaojie Li and Yong Liu.
\newblock High probability guarantees for nonconvex stochastic gradient descent with heavy tails.
\newblock In {\em International Conference on Machine Learning}, pages 12931--12963. PMLR, 2022.

\bibitem{JCPM2013}
L.~Lin, M.~Chen, C.~Yang, and L.~He.
\newblock Accelerating atomic orbital-based electronic structure calculation via pole expansion and selected inversion.
\newblock {\em J. Phys. Condens. Matter}, 25:295501, 2013.

\bibitem{CMS2009}
L.~Lin, J.~Lu, L.~Ying, R.~Car, and W.~E.
\newblock Fast algorithm for extracting the diagonal of the inverse matrix with application to the electronic structure analysis of metallic systems.
\newblock {\em Comm. Math. Sci.}, 7:755, 2009.

\bibitem{lin2009pole}
Lin Lin, Jianfeng Lu, Lexing Ying, and E~Weinan.
\newblock Pole-based approximation of the fermi-dirac function.
\newblock {\em Chinese Annals of Mathematics, Series B}, 30(6):729--742, 2009.

\bibitem{lindsey2023fastrandomizedentropicallyregularized}
Michael Lindsey.
\newblock Fast randomized entropically regularized semidefinite programming, 2023.

\bibitem{liuPlaneWaveBasedStochasticDeterministicDensity2022}
Qianrui Liu and Mohan Chen.
\newblock Plane-{{Wave-Based Stochastic-Deterministic Density Functional Theory}} for {{Extended Systems}}.
\newblock {\em Physical Review B}, 106(12):125132, September 2022.

\bibitem{liuHighProbabilityConvergence2023}
Zijian Liu, Ta~Duy Nguyen, Thien~Hang Nguyen, Alina Ene, and Huy~L{\^e} Nguyen.
\newblock High {{Probability Convergence}} of {{Stochastic Gradient Methods}}, February 2023.

\bibitem{doi:10.1137/16M1099546}
Haihao Lu, Robert~M. Freund, and Yurii Nesterov.
\newblock Relatively smooth convex optimization by first-order methods, and applications.
\newblock {\em SIAM Journal on Optimization}, 28(1):333--354, 2018.

\bibitem{LU2015329}
Jianfeng Lu and Lexing Ying.
\newblock Compression of the electron repulsion integral tensor in tensor hypercontraction format with cubic scaling cost.
\newblock {\em Journal of Computational Physics}, 302:329--335, 2015.

\bibitem{Meyer2021-pn}
Raphael~A Meyer, Cameron Musco, Christopher Musco, and David~P Woodruff.
\newblock Hutch++: Optimal stochastic trace estimation.
\newblock {\em Proc SIAM Symp Simplicity Algorithms}, 2021:142--155, January 2021.

\bibitem{nemirovskiRobustStochasticApproximation2009}
A.~Nemirovski, A.~Juditsky, G.~Lan, and A.~Shapiro.
\newblock Robust {{Stochastic Approximation Approach}} to {{Stochastic Programming}}.
\newblock {\em SIAM Journal on Optimization}, 19(4):1574--1609, January 2009.

\bibitem{Nemirovski1983problem}
Arkadi Nemirovski and David Yudin.
\newblock Problem complexity and method efficiency in optimization, 1983.

\bibitem{10.1063/1.4768233}
Robert~M. Parrish, Edward~G. Hohenstein, Todd~J. Mart\'{i}nez, and C.~David Sherrill.
\newblock Tensor hypercontraction. {II. Least-squares renormalization}.
\newblock {\em The Journal of Chemical Physics}, 137(22):224106, 12 2012.

\bibitem{szabo1996modern}
Attila Szabo and Neil~S. Ostlund.
\newblock {\em Modern Quantum Chemistry: Introduction to Advanced Electronic Structure Theory}.
\newblock Dover Publications, Inc., Mineola, first edition, 1996.

\bibitem{trefethen2019approximation}
Lloyd~N Trefethen.
\newblock {\em Approximation theory and approximation practice, extended edition}.
\newblock SIAM, 2019.

\bibitem{van1992bi}
Henk~A Van~der Vorst.
\newblock Bi-cgstab: A fast and smoothly converging variant of bi-cg for the solution of nonsymmetric linear systems.
\newblock {\em SIAM Journal on scientific and Statistical Computing}, 13(2):631--644, 1992.

\bibitem{vuralMirrorDescentStrikes2022}
Nuri~Mert Vural, Lu~Yu, Krishna Balasubramanian, Stanislav Volgushev, and Murat~A. Erdogdu.
\newblock Mirror {{Descent Strikes Again}}: {{Optimal Stochastic Convex Optimization}} under {{Infinite Noise Variance}}.
\newblock In {\em Proceedings of {{Thirty Fifth Conference}} on {{Learning Theory}}}, pages 65--102. PMLR, June 2022.

\bibitem{wainwright2019high}
Martin~J Wainwright.
\newblock {\em High-dimensional statistics: A non-asymptotic viewpoint}, volume~48.
\newblock Cambridge University Press, 2019.

\bibitem{whiteFastUniversalKohnSham2020}
A.~J. White and L.~A. Collins.
\newblock Fast and {{Universal Kohn-Sham Density Functional Theory Algorithm}} for {{Warm Dense Matter}} to {{Hot Dense Plasma}}.
\newblock {\em Physical Review Letters}, 125(5):055002, July 2020.

\bibitem{10.1063/1.5007066}
Steven~R. White.
\newblock Hybrid grid/basis set discretizations of the schr\"{o}dinger equation.
\newblock {\em The Journal of Chemical Physics}, 147(24):244102, 12 2017.

\end{thebibliography}

\appendix
\pagebreak{}

\part*{Appendices}

\section{Proofs for elementary facts \label{app:elem} (Section \ref{subsec:Elementary-facts})}
\begin{proof}
[Proof of Lemma \ref{lem:stronglyconvex}.]We can verify by elementary
computations that: 
\begin{align*}
S_{\mathrm{FD}}(X) & =n\left(\Tr\left[\frac{X}{n}\log\left(\frac{X}{n}\right)\right]-\Tr\left[\frac{X}{n}\right]\right)\\
 & \quad\quad+\ \ n\left(\Tr\left[\frac{\mathbf{I}_{n}-X}{n}\log\left(\frac{\mathbf{I}_{n}-X}{n}\right)\right]-\Tr\left[\frac{\mathbf{I}_{n}-X}{n}\right]\right)+n+n\log n\\
 & =n\,S_{\mathrm{VN}}(n^{-1}X)+n\,S_{\mathrm{VN}}(n^{-1}[\mathbf{I}_{n}-X])+n+n\log n,
\end{align*}
 where $S_{\mathrm{VN}}(Y):=\Tr(Y\log Y)-\Tr(Y)$ is the unnormalized
von Neumann entropy on 
\[
\{Y\,:\,Y\succeq0,\ \Tr[Y]\leq1\}.
\]
 Note that both $X$ and $\mathbf{I}_{n}-X$ lie int this domain for
$X\in\mathcal{X}$.

Now $S_{\mathrm{VN}}$ is $1$-strongly convex with respect to $\Vert\,\cdot\,\Vert_{*}$
\cite{10.1063/1.4871575}. It is equivalent (cf., e.g., Proposition
1 of \cite{doi:10.1137/16M1099546}) to say that the Hessian satisfies
\[
\left\langle Z,\nabla^{2}S_{\mathrm{VN}}(Y)\,[Z]\right\rangle \geq\Vert Z\Vert_{*}^{2}.
\]
 But 
\[
\nabla^{2}S_{\mathrm{FD}}(X)=\frac{1}{n}\nabla^{2}S_{\mathrm{VN}}(n^{-1}X)+\frac{1}{n}\nabla^{2}S_{\mathrm{VN}}(n^{-1}[\mathbf{I}_{n}-X]),
\]
 and therefore 
\[
\left\langle Z,\nabla^{2}S_{\mathrm{FD}}(X)\,[Z]\right\rangle \geq\frac{2}{n}\Vert Z\Vert_{*}^{2},
\]
 hence $S_{\mathrm{FD}}$ is $(2/n)$-strongly convex with respect
to $\Vert\,\cdot\,\Vert_{*}$.
\end{proof}
\begin{center}----------------------------------------------------------------------\end{center}
\begin{proof}
[Proof of Lemma \ref{lem:divbound}.] First compute 
\[
S_{\mathrm{FD}}(X_{0})=\Tr\left[\log(\mathbf{I}_{n}/2)\right]=-n\log2.
\]
 Meanwhile it is straightforward to see that $S_{\mathrm{FD}}(X)\leq0$,
since 
\[
x\log x+(1-x)\log(1-x)\leq0
\]
 for all $x\in[0,1]$. Finally, observe that $\nabla S_{\mathrm{FD}}(X_{0})=\log\left((\mathbf{I}_{n}/2)(\mathbf{I}_{n}/2)^{-1}\right)=0$.

Then by definition
\[
D(X\Vert X_{0})=S_{\mathrm{FD}}(X)-S_{\mathrm{FD}}(X_{1})-\left\langle \nabla S_{\mathrm{FD}}(X_{0}),X-X_{0}\right\rangle \leq n\log2,
\]
 as was to be shown.
\end{proof}

\section{Proofs for sub-exponential concentration (Section \ref{subsec:Sub-exponential-concentration})
\label{app:subexp}}
\begin{proof}
[Proof of Corollary \ref{cor:subexp}.] For $t\geq\frac{\nu}{2}$,
we have that $e^{-\frac{t^{2}}{2\nu^{2}}}\leq e^{-\frac{t}{4\nu}}$.
Therefore by Proposition \ref{prop:subexp}, we have 
\[
\P\left[x\geq\mu+t\,\vert\,\mathcal{F}\right]\leq e^{-\frac{t}{4\nu}}
\]
 for all $t\geq\frac{\nu}{2}$. Since the right-hand side is deterministic,
we can take the expectation of both sides and use the tower property
of the conditional expectation to deduce that 
\[
\P\left[x\geq\mu+t\right]\leq e^{-\frac{t}{4\nu}}
\]
 for all $t\geq\frac{\nu}{2}$.

Therefore 
\[
\P\left[x\geq\mu+\frac{\nu}{2}+t\right]\leq e^{-\frac{(t+\nu/2)}{4\nu}}\leq e^{-\frac{t}{4\nu}}
\]
 for all $t\geq0$.

Then solving $e^{-\frac{t}{4\nu}}=\delta$ for $t$, we see that 
\[
x\leq\mu+\frac{\nu}{2}+4\nu\log(1/\delta)
\]
 with probability at least $1-\delta$.
\end{proof}
\begin{center}----------------------------------------------------------------------\end{center}
\begin{proof}
[Proof of Lemma \ref{lem:hutchsubexp}.] Throughout we will simply
treat $A$ as deterministic and aim to show that 
\[
\E\left[e^{\lambda(z^{\top}Az-\Tr[A])}\right]\leq e^{\frac{(2\Vert A\Vert_{\mathrm{F}})^{2}\lambda^{2}}{2}},\quad\text{for all }\vert\lambda\vert\leq\frac{1}{4\Vert A\Vert}.
\]
 We can reduce to this case by disintegration.

Without loss of generality we can also let $A$ be symmetric. (Otherwise,
replace $A$ with its symmetrization $\frac{1}{2}(A+A^{\top})$, which
leaves $z^{\top}Az$ unchanged and moreover cannot expand either the
Frobenius or spectral norms.)

Use the spectral theorem to write $A=UEU^{\top}$ where $U$ is orthogonal
and $E=\mathrm{diag}(\ve_{1},\ldots,\ve_{n})$. Then 
\[
z^{\top}Az=(U^{\top}z)^{\top}E(U^{\top}z)=\sum_{i=1}^{n}\ve_{i}\tilde{z}_{i}^{2},
\]
 where $\tilde{z}:=U^{\top}z$. Note that the components $\tilde{z}_{i}$
are independent standard normal random variables. 

It is well-known \cite{wainwright2019high} that a squared standard
normal random variable is sub-exponential with parameters $(2,4)$,
so 
\begin{align*}
\E\left[e^{\lambda(z^{\top}Az-\Tr[A])}\right] & =\E\left[e^{\lambda\sum_{i}\ve_{i}(\tilde{z}_{i}^{2}-1)}\right]\\
 & =\prod_{i=1}^{n}\E\left[e^{\lambda\ve_{i}(\tilde{z}_{i}^{2}-1)}\right]\\
 & \leq\prod_{i=1}^{n}e^{\frac{2^{2}(\lambda\ve_{i})^{2}}{2}}
\end{align*}
 as long as $\vert\lambda\ve_{i}\vert\leq\frac{1}{4}$ for all $i$,
which in particular holds as long as $\vert\lambda\vert\leq\frac{1}{4\Vert A\Vert}$,
since $\Vert A\Vert=\max_{i}\vert\ve_{i}\vert$. Then since in addition
we have that $\Vert A_{\mathrm{F}}\Vert^{2}=\sum_{i=1}^{n}\ve_{i}^{2}$,
we continue our computation to deduce that 
\[
\E\left[e^{\lambda(z^{\top}Az-\Tr[A])}\right]\leq e^{\frac{(2\Vert A\Vert_{\mathrm{F}})^{2}\lambda^{2}}{2}},
\]
 meaning that $z^{\top}Az$ is sub-exponential with parameters $(2\Vert A\Vert_{\mathrm{F}},4\Vert A\Vert)$.
\end{proof}

\section{Contour integral details \label{app:contour}}

\begin{figure}
\centering{}\includegraphics[width=1\textwidth]{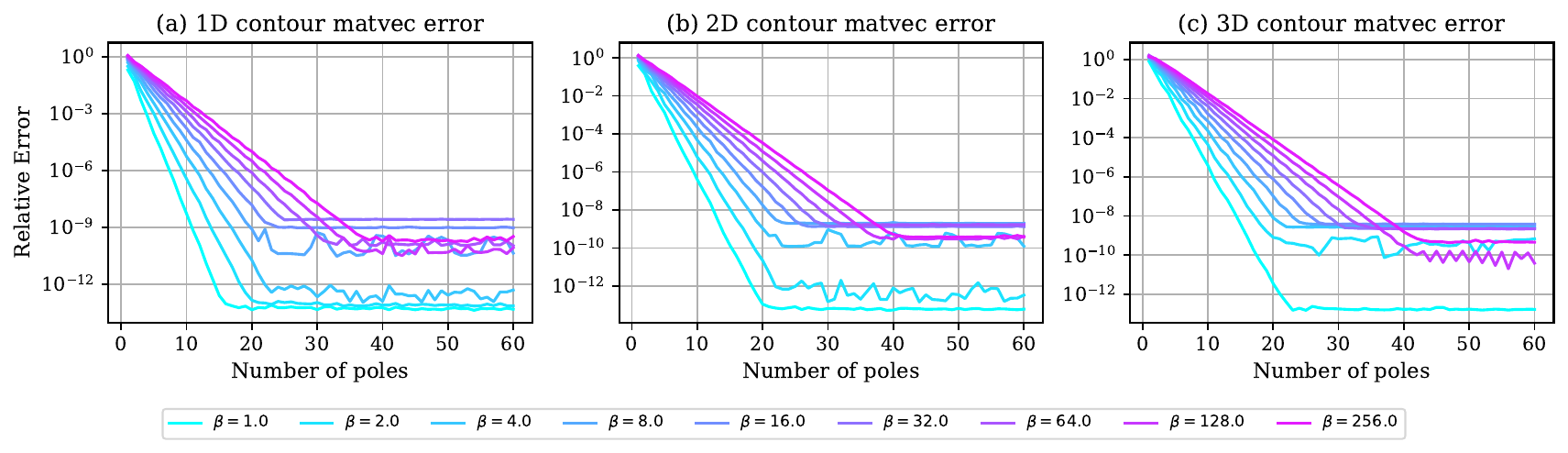}\caption{Relative error of the contour matvec as a function of $N_{p}$ for
various spatial dimensions and inverse temperatures $\beta$. For
the experiments, following the notation of Section \ref{sec:Experiments},
we set the grid sizes to $\mathbf{n}=101$ for 1D, (31,31) for 2D,
and (11,11,11) for 3D, while the corresponding box sizes are $\mathbf{L}=100$
for 1D, $\mathbf{L}=(30,30)$ for 2D, and $\mathbf{L}=(10,10,10)$
for 3D. In generating the external potential, we use $\alpha=0.5$
to define the Yukawa interaction. The matrix here is $C=K+U$, the
vectors are random Gaussian samples, and the error is averaged over
a sample size of $N_{g}=10$. \label{fig:contour-error}}
\end{figure}

We now discuss the use of contour integration to approximate matvecs
by the square-root Fermi-Dirac function of the Hamiltonian, i.e.,
products of the form $X_{t}^{1/2}z=f_{\beta}^{1/2}(H_{t})z$ for arbitrary
vectors $z\in\R^{n}$. The idea, borrowed from \cite{hale2008computing},
is to apply Cauchy's integral theorem to the holomorphic extension
of $f_{\beta}^{1/2}$, i.e., 
\begin{equation}
f_{\beta}^{1/2}(H_{t})z=\left[\int_{\partial\Gamma}g(s)\,(s\mathbf{I}_{n}-H_{t})^{-1}\,ds\right]z\approx g_{N_{p}}(H_{t})z\coloneqq\sum_{i=1}^{N_{p}}w_{i}\,g(s_{i})\,\Big[(s_{i}\mathbf{I}_{n}-H_{t})^{-1}z\Big],\label{eq:contourmatvec}
\end{equation}
where $\Gamma$ is a good (smooth-boundaried) open region covering
all the eigenvalues of $H_{t}$, $g$ is a holomorphic extension of
$f_{\beta}^{1/2}(x)$ from $\R$ to $\Gamma$, $N_{p}$ is the number
of points used to discretize the contour integral, $s_{i}\in\partial\Gamma$
are the discretization points, and $w_{i}$ are quadrature weights
associated to $s_{i}$. Specifically, we apply the same contour and
discretization as in PEXSI \cite{lin2009pole}. (See the rightmost
panel of Figure \ref{fig:extensiondumbbell} below.)

Although it is not immediately obvious that such a holomorphic extension
$g$ exists, we will demonstrate that it does below in Appendix \ref{app:Holomorphic-extensions}.

Then the following theorem, whose proof is identical to its analogue
from \cite{lin2009pole}, bounds the error of this approach.
\begin{thm}
[Section 2.2 of \cite{lin2009pole}] There exists a constant $C$
such that aforementioned contour integral approximation satisfies
\[
\|f_{\beta}^{1/2}(H_{t})-g_{N_{p}}(H_{t})\|=O\big(e^{-CN_{p}/\log(\beta\,\sigma(H))}\big),
\]
 where $\sigma(H)$ is the maximal singular value of $H$.
\end{thm}

The plots in Figure \ref{fig:contour-error} validates this result,
demonstrating consistent behavior across spatial dimensions and inverse
temperatures.

\subsection{Holomorphic extension \label{app:Holomorphic-extensions}}

Now we focus on the construction of the holomorphic extension $g$
adapted to the dumbbell contour depicted in the rightmost panel of
Figure \ref{fig:extensiondumbbell}. Note that it suffices to construct
a holomorphic extension of $\log f_{\beta}.$ Indeed, if $h(z)$ is
a suitable holmorphic extension of $\log f_{\beta}$, then $\exp(h(z)/2)$
defines an extension for $f_{\beta}^{1/2}$.

Without loss of generality, we assume $\beta=1$ via a horizontal
scaling, since $f_{\beta}(x)=f_{1}(\beta x)$. Then for $\log f_{1}$,
we construct $h$ as:
\[
h(z)\coloneqq\begin{cases}
\log f_{1}(z), & \text{if }\Re(z)\le0,\\
\log|1+\exp(\beta z)|+i\cdot\Big\{\arg\big[1+\exp(\beta z)\big]+2\pi\Big\lfloor\frac{\Im(z)-\pi}{2\pi}\Big\rfloor\Big\}, & \text{if }\Re(z)>0.
\end{cases}
\]
Therefore, for $f_{1}^{1/2},$we construct the holomorphic extension
$g$ as:
\[
g(z)\coloneqq\begin{cases}
f_{1}^{1/2}(z), & \text{if }\Re(z)\le0,\\
\big|f_{1}^{1/2}(z)\big|\cdot\exp\Bigg\{ i\Bigg[\pi\bigg(1-\Big\lfloor\frac{\Im(z)-\pi/2}{\pi}\Big\rfloor\bigg)+\frac{\arg\big[f_{\beta}(z)\big]}{2}\bigg]\Bigg\}, & \text{if }\Re(z)>0.
\end{cases}
\]
All these constructions are holomorphic on $\mathbb{C}\setminus\{iy\,|\,y\in(-\infty,-\pi]\cup[\pi,\infty)\}$,
as illustrated in Figure \ref{fig:extensiondumbbell}. 
\begin{figure}
\begin{centering}
\includegraphics[width=1\textwidth]{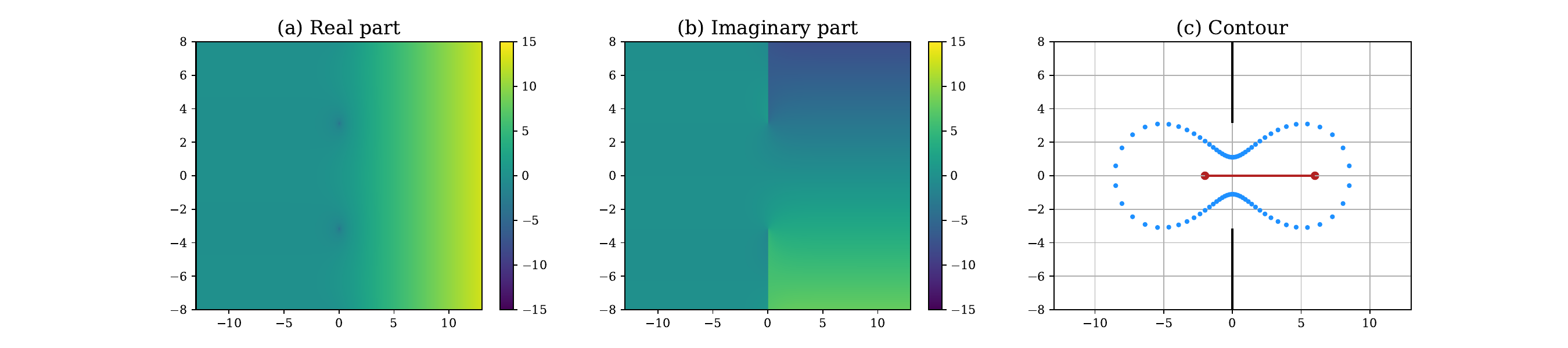}
\par\end{centering}
\caption{Panels (a) and (b) together show the holomorphic extension of $-\log f_{\beta}(x)$.
Panel (c) shows the 20 poles generated for the eigenvalue range $[-2,6]$
and the inverse temperature $\beta=1$. \label{fig:extensiondumbbell}}

\end{figure}

\subsection{Solver and preconditioner \label{app:precondition}}

For our experiments outlined in Section \ref{sec:Experiments}, we
use BiCGSTAB \cite{van1992bi} to solve each linear system 
\[
(s_{i}\mathbf{I}_{n}-H_{t})x=z
\]
 appearing in the right-hand side of (\ref{eq:contourmatvec}).

In our periodic sinc basis, a nice property of the mirror descent
update is that $H_{t}$ can always be written as 
\[
H_{t}=c_{t}K+\text{diag}^{*}(v_{t}),
\]
for some scalar $c_{t}\in\R$ and a suitable effective potential $v_{t}\in\R^{n}$.
Let $\overline{v}_{t}\in\R$ denote the mean of the components of
$v_{t}$. Then our choice of preconditioner in BiCGSTAB is defined
by the linear map 
\[
x\mapsto\big(s_{i}\mathbf{I}_{n}-c_{t}K-\overline{v}_{t}\mathbf{I}_{n}\big)^{-1}x,
\]
 which permits log-linear implementation via FFT.

In Figure \ref{fig:Scaling-of-contour}, we report results indicating
that the time compexity of the contour matvec scales as $O(\sqrt{\beta})$,
independently of the box size.
\begin{figure}
\begin{centering}
\includegraphics[width=0.7\textwidth]{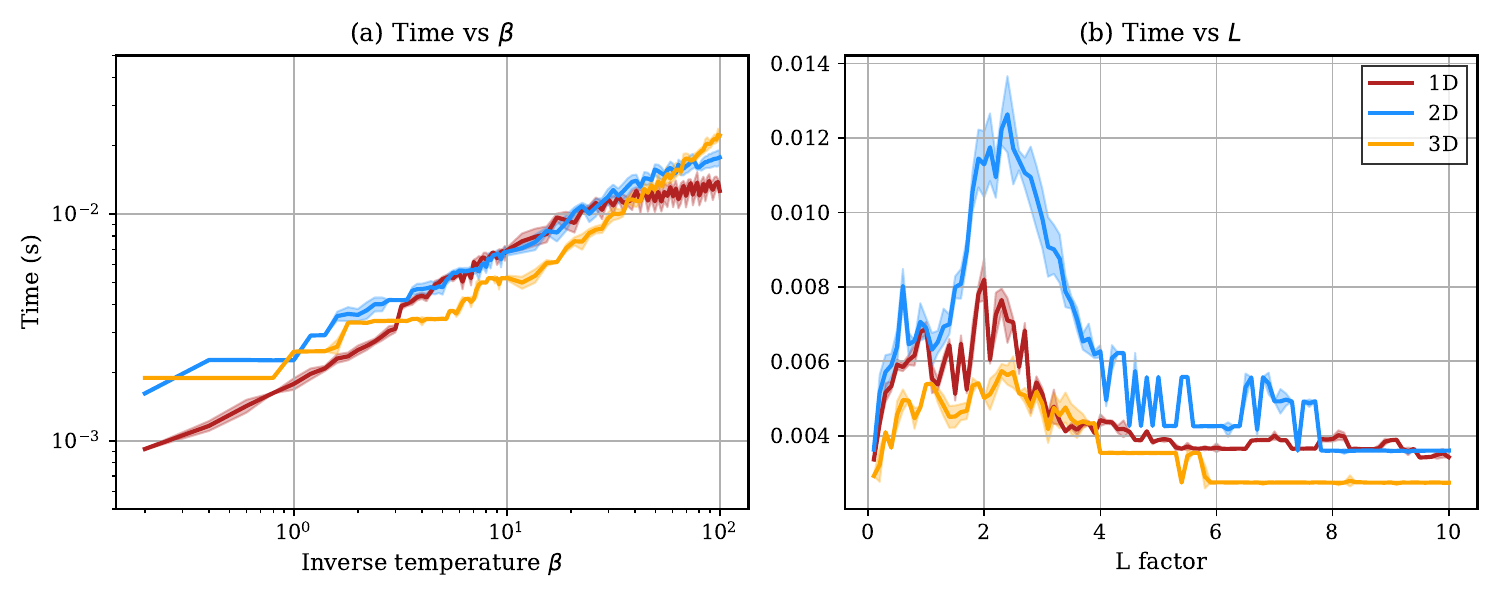}
\par\end{centering}
\caption{Scaling of the wall clock time of a batch of $N_{g}=20$ contour matvecs
with respect to the inverse temperature and the box size. Here we
choose the number of poles $N_{p}$ to ensure that the relative error
is about $10^{-5}$. Following the notation of Section \ref{sec:Experiments},
the base choices of $(\mathbf{n},L)$ for 1D, 2D and 3D are $(101,100),((31,31),30),((11,11,11),10)$.
In the right panel (b), the `$L$ factor' indicates a scaling factor
applied to the base box size $L$. (a) The log-log plot has a slope
close to $1/2$, which indicates the scaling with respect to $\beta$
is approximately $O(\sqrt{\beta})$. (b) Fixing $\beta=10$, the results
indicate that the time complexity is upper-bounded by a constant independent
of the box size.\label{fig:Scaling-of-contour}}
\end{figure}

\subsection{Entropy term in the objective}

We also use the contour integral technique to estimate the objective
in order to plot the objective convergence of our algorithm.

To do so, we must explain how to estimate two additional terms: the
single-electron term $\Tr(CX_{t})$ and the entropy term $\frac{1}{\beta}S_{\text{FD}}(X_{t})$.
The estimation of $\Tr(CX_{t})$ reuses the contour matvec results
for the gradient estimation by rewriting: 
\[
\Tr(C\ensuremath{X_{t})}=\E\left\{ \big[f_{\beta}^{1/2}(H_{t})z_{t}\big]^{\top}C\big[f_{\beta}^{1/2}(H_{t})z_{t}\big]\right\} .
\]

On the other hand, to estimate the entropy function, we must approximate
an alternative matrix function. Indeed, 
\[
S_{\text{FD}}(X_{t})=\E\left\{ z_{t}^{\top}\big[X_{t}\log X_{t}+(\mathbf{I}_{n}-X_{t})\log(\mathbf{I}_{n}-X_{t})\big]z_{t}\right\} .
\]
Therefore, we aim to approximate the matvec 
\[
\big[X_{t}\log X_{t}+(I-X_{t})\log(X_{t})\big]z_{t}
\]
 For simplicity, we can focus on computing 
\[
[X_{t}\log X_{t}]z_{t},
\]
 since the second term is analogous.

Since $X_{t}=f_{\beta}(H_{t})$, we are motivated to apply the contour
integration technique to a suitable holomorphic extension $\tilde{g}$
of $f_{\beta}\log f_{\beta}$. But since we already constructed an
extension for $\log f_{\beta}$, such an extension is recovered easily.
Specifically, for $f_{1}\log f_{1}$, we may construct $\tilde{g}$
as: 
\[
\tilde{g}(z)\coloneqq\begin{cases}
f_{1}(z)\log f_{1}(z), & \text{if }\Re(z)\le0,\\
-\frac{\log|1+\exp(z)|+i\cdot\Big\{\arg\big[1+\exp(z)\big]+2\pi\Big\lfloor\frac{\Im(z)-\pi}{2\pi}\Big\rfloor\Big\}}{1+\exp(z)}, & \text{if }\Re(z)>0.
\end{cases}
\]
Indeed, this extension is also holomorphic on $\mathbb{C}\setminus\{iy\,|\,y\in(-\infty,-\pi]\cup[\pi,\infty)\}$.
Hence, we can use the same contour and discretization as before. Then
the linear solves can be recycled from the earlier computations and
simply weighted differently to construct our objective estimator.

\section{Discussion of more general DFT \label{app:dft}}

In general density functional theory, the energy consists of several
components: 
\[
E(X)=\Tr[CX]+E_{\mathrm{hxc}}(X),
\]
 where 
\[
C_{ij}=\int\psi_{i}(x)\left[-\frac{1}{2}\Delta+v_{\mathrm{ext}}\right]\psi_{j}(x)\,dx
\]
 denotes the matrix of the single-particle part of the quantum chemistry
Hamiltonian in the $\{\psi_{i}\}$ basis, with $v_{\mathrm{ext}}$
denoting the diagonal external potential, and $E_{\mathrm{hxc}}$
denotes the Hartree and exchange-correlation contributions to the
energy.

In turn, 
\[
E_{\mathrm{hxc}}(X)=\tilde{E}(X)+E_{\mathrm{xc}}(X)
\]
 consists of the Hartree energy which we have considered in this work
as well as the \textbf{\emph{exchange-correlation}} energy $E_{\mathrm{xc}}$.

The exchange-correlation energy traditionally it only depends on $X$
via the induced electron density $\rho_{X}$. (Notably, for example,
hybrid functionals may more generally depend on $X$ via the density
matrix $P_{X}$.)

We comment that in particular, local density approximation (LDA) functionals
have been used in stochastic DFT \cite{fabianStochasticDensityFunctional2019}
and assume the form 
\[
E_{\mathrm{xc}}(X)=\int\eps_{\mathrm{LDA}}(\rho_{X}(x))\,dx,
\]
 where $\eps_{\mathrm{LDA}}$ could be a fairly arbitrary function
$\R\ra\R$. Under these circumstances, the exchange-correlation potential
takes the form 
\[
\nabla E_{\mathrm{xc}}(X)=\int\eps_{\mathrm{LDA}}'(\rho_{X}(x))\,\psi_{i}(x)\psi_{j}(x)\,dx.
\]

In order to implement stochastic DFT with an LDA functional, one must
construct $\rho_{X}$ on a spatial grid and evaluate $\eps_{\mathrm{LDA}}'$
pointwise to construct the LDA potential $\eps_{\mathrm{LDA}}'(\rho_{X}(x))$
on the grid, which is then projected back to the basis. Note that
this perspective is compatible with our definition of the vector $\rho(X)$,
which can be interpreted (following the discussion of Section \ref{subsec:Electron-repulsion-integrals})
as the vector of evaluations of the electron density on an interpolating
grid.

There are two difficulties in extending our analysis to a more general
setting. First of all, typically the exchange-correlation functional
fails to be convex, and in fact many local optima may be present.
(In particular, for LDA approximations, $\eps_{\mathrm{LDA}}$ is
not typically convex.) This seriously complicates the convergence
theory for mirror descent, though it is reasonable to expect the theory
to hold qualitatively locally near a local optimizer.

Second, as $\nabla E_{\mathrm{xc}}(X)$ is not generally linear in
$X$, the gradient estimator for the exchange-correlation potential
will be biased, unlike our estimator for the Hartree potential $\nabla\tilde{E}(X)$.
Therefore we do not enjoy the straightforward self-cancellation of
estimation error over the optimization trajectory---captured in our
analysis via the concentration inequality for martingale difference
sequences (cf. Theorem \ref{thm:mds}).

However, we believe that the analysis of the Hartree case alone is
crucial for understanding more general DFT. First of all, it is widely
understood that the Hartree energy is the dominant contribution among
the Hartree and exchange-correlation contributions. Although the exchange-correlation
contribution is extremely relevant chemically, it is quantitatively
much smaller, and therefore the estimation bias may be regarded as
relatively small. If $S$ shots are used in estimator for the density
$\rho(X)$, then formally one expects the bias to scale as $S^{-1}$,
with a preconstant proportional to the overall magnitude of exchange-correlation
contribution. Meanwhile, the unbiased fluctuations (which scale as
$S^{-1/2}$) will enjoy self-cancellation. Additionally, one might
consider bias reduction techniques such as the jackknife \cite{Efron_1982}
for reducing the bias to $S^{-2}$.

Moreover, it is possible to conceive of schemes for general stochastic
DFT that solve the self-consisent Hartree theory as a subroutine,
freezing the exchange-correlation potential within an inner loop in
which the Hartree potential is optimized. Our convergence theory would
be applicable to this inner loop. (We comment that similar two-loop
strategies are used to converge the Hartree-Fock theory as well as
DFT with hybrid functionals.) Since our convergence theory strikingly
suggests that converging the self-consistent Hartree theory to relative
accuracy $\ve$ is about as easy as estimating the Hartree potential
itself to relative accuracy $\ve$, this separation of difficulties
seems worthwhile. Meanwhile, the outer loop itself may be easier to
converge quickly due to the relatively small magnitude of the exchange-correlation
contribution. We leave the implementation and analysis of such directions
to future work.

\section{Proofs for chemical potential optimization (Section \ref{sec:Chemical-potential-optimization})
\label{app:chemical}}
\begin{proof}
[Proof of Lemma \ref{lem:glip}.] The supergradient is 
\[
\partial g_{N,\beta}(\mu)=N-\partial g_{\beta}(\mu),
\]
 where 
\[
\partial g_{\beta}(\mu)=\left\{ \Tr[X]\,:\,X\text{ is a minimizer of }F_{\beta}(X)-\mu\Tr[X]\right\} .
\]
 Hence for $\beta<+\infty$, the function $g_{N,\beta}$ is differentiable,
but in any case $g_{N,\beta}$ is Lipschitz with Lipschitz constant
$n$, since $N-\Tr[X]\in[-n,n]$ for any $X$ satisfying $0\preceq X\preceq\mathbf{I}_{n}$.
\end{proof}
\begin{center}----------------------------------------------------------------------\end{center}
\begin{proof}
[Proof of Lemma \ref{lem:mubound}.] Let us define 
\[
a:=\lambda_{\min}(C)-c_{\mathrm{h}}-\beta^{-1}\log(\gamma^{-1})
\]
 and 
\[
b:=\lambda_{\max}(C)+c_{\mathrm{h}}+\beta^{-1}\log([1-\gamma]^{-1})
\]
 as the left and right endpoints of our desired bounding interval.

Suppose first that $\beta<+\infty$. We know that any optimizer $\mu$
must satisfy $N=\Tr\left[X_{\beta,\mu}\right]$ where $X_{\beta,\mu}$
is the optimizer of (\ref{eq:opt})\@. But recall (\ref{eq:fixedpoint}),
i.e., that 
\[
X_{\beta,\mu}=f_{\beta}\left(\nabla E(X_{\beta,\mu})-\mu\mathbf{I}_{n}\right).
\]

We use the fact that 
\[
f_{\beta}(x)=\frac{1}{1+e^{\beta x}}\leq e^{-\beta x}
\]
 for all $x$ to deduce that 
\[
\Tr\left[X_{\beta,\mu}\right]\leq\Tr\left[\exp\left(-\beta\left[\nabla E(X_{\beta,\mu})-\mu\mathbf{I}_{n}\right]\right)\right].
\]
 Now $\nabla E(X)=C+\nabla\tilde{E}(X)\succeq C-c_{\mathrm{h}}\mathbf{I}_{n}$
by Lemma \ref{lem:Gbound}, so it follows that 
\[
\Tr\left[X_{\beta,\mu}\right]\leq e^{\beta(\mu-\lambda_{\min}(C)+c_{\mathrm{h}})}\,n,
\]
 meaning that $\mu\geq a.$ This gives the desired lower bound for
$\mu$.

Similarly, using the inequality 
\[
f_{\beta}(x)=1-f_{\beta}(-x)\geq1-e^{\beta x},
\]
 we deduce that 
\[
\Tr\left[X_{\beta,\mu}\right]\geq n-\Tr\left[\exp\left(\beta\left[\nabla E(X_{\beta,\mu})-\mu\mathbf{I}_{n}\right]\right)\right],
\]
 which in turn implies that 
\[
1-\gamma\leq e^{\beta(\lambda_{\max}(C)+c_{\mathrm{h}}-\mu)},
\]
 i.e., that $\mu\leq b$ This gives the desired upper bound for $\mu$.

These arguments show that any maximizer must be attained in the desired
interval. But let us verify concretely that a maximizer is in fact
attained. Indeed, note that the same arguments show that if $\mu>a$
holds strictly, then for any $X_{\beta,\mu}$ solving (\ref{eq:opt}),
we have $\Tr\left[X_{\beta,\mu}\right]>N$, i.e., $g_{N,\beta}'(\mu)<0$.
Likewise if $\mu<b$ holds strictly, then for any $X_{\beta,\mu}$
solving (\ref{eq:opt}), we have $\Tr\left[X_{\beta,\mu}\right]<N$,
i.e., $g_{N,\beta}'(\mu)>0$. Together these facts guarantee that
in fact the maximum must be attained on the interval appearing in
the statement of the theorem.

Now consider the limiting case $\beta=+\infty$. The same arguments
show that if $\mu\leq a$, then $\nabla E(X)-\mu\mathbf{I}_{n}\succeq0$
for all $X\in\mathcal{X}$ (and likewise that if $\mu<a$, then $\nabla E(X)-\mu\mathbf{I}_{n}\succ0$).
Therefore if $\mu\leq a$, then $X=0$ is an optimizer of (\ref{eq:opt})
(unique if $\mu<a$). Similarly, if $\mu\geq b$, then $X=\mathbf{I}_{n}$
is an optimizer of (\ref{eq:opt}) (unique if $\mu>b$).

It follows from (\ref{eq:gbeta}) and (\ref{eq:gNmu}) that if $\mu\leq a$,
then $g_{N,\beta}(\mu)=N\mu$, and if $\mu\geq b$, then $g_{N,\beta}(\mu)=(N-n)\mu$.
Thus $g_{N,\beta}$ is non-decreasing for $\mu\leq a$ and non-increasing
for $\mu\geq b$. It follows that a maximum is attained in $[a,b]$.
\end{proof}
\begin{center}----------------------------------------------------------------------\end{center}
\begin{proof}
[Proof of Proposition \ref{prop:chemopt}.] By Theorem \ref{thm:convergence},
we know that 
\begin{equation}
\max_{k=0,\ldots,M}\left|\frac{\hat{g}_{N,\beta,k}}{n}-\frac{g_{N,\beta}(\mu_{k})}{n}\right|\leq O\left(\frac{c_{\mathrm{h}}\,\log(KTm/\delta)}{\sqrt{T}}+\frac{\Vert C\Vert+c_{\mathrm{h}}+\beta^{-1}\log\left([\gamma(1-\gamma)]^{-1}\right)}{T}\right)\label{eq:gNbetak}
\end{equation}
 holds with probability at least $1-\delta$, where we have used the
fact that 
\[
\mu_{k}\in\left[\lambda_{\min}(C)-c_{\mathrm{h}}-\beta^{-1}\log[\gamma^{-1}],\ \lambda_{\max}(C)+c_{\mathrm{h}}+\beta^{-1}\log([1-\gamma]^{-1})\right],
\]
 so 
\[
\Vert C-\mu_{k}\mathbf{I}_{n}\Vert=O\left(\frac{\Vert C\Vert+c_{\mathrm{h}}+\beta^{-1}\log\left([\gamma(1-\gamma)]^{-1}\right)}{T}\right)
\]
 for all $k$.

Recall from Lemma \ref{lem:glip} that $g_{N,\beta}(\mu_{k})$ is
Lipschitz with constant $n$. Moreover, the spacing $h=\mu_{k+1}-\mu_{k}$
satisfies 
\[
h\leq\frac{2\Vert C\Vert+2c_{\mathrm{h}}+\beta^{-1}\log\left([\gamma(1-\gamma)]^{-1}\right)}{M}.
\]
 Let $\mu_{\star}$ denote an optimizer of $g_{N,\beta}.$ We know
by Lemma \ref{lem:mubound} that there must exist some $k$ such that
$\vert\mu_{k}-\mu_{\star}\vert\leq h/2$. Thus the result follows
from Lipschitzness and (\ref{eq:gNbetak}).
\end{proof}

\section{Proofs for complete basis set limit (Section \ref{sec:cbl}) \label{app:cbl}}
\begin{proof}
[Proof of Lemma \ref{lem:stronglyconvextau}.] Similarly to the
proof of Lemma \ref{lem:stronglyconvex}, we can verify by elementary
computations that: 
\begin{align*}
S_{\mathrm{FD}}(X) & =\tau\left(\Tr\left[\frac{X}{\tau}\log\left(\frac{X}{\tau}\right)\right]-\Tr\left[\frac{X}{\tau}\right]\right)\\
 & \quad\quad+\ \ \tau\left(\Tr\left[\frac{\mathbf{I}_{n}-X}{\tau}\log\left(\frac{\mathbf{I}_{n}-X}{\tau}\right)\right]-\Tr\left[\frac{\mathbf{I}_{n}-X}{\tau}\right]\right)+n+n\log\tau\\
 & =\tau\,S_{\mathrm{VN}}(\tau^{-1}X)+\tau\,S_{\mathrm{VN}}(\tau^{-1}[\mathbf{I}_{n}-X])+n+n\log\tau,
\end{align*}
where $S_{\mathrm{VN}}(Y):=\Tr(Y\log Y)-\Tr(Y)$ is the unnormalized
von Neumann entropy on 
\[
\{Y\,:\,Y\succeq0,\ \Tr[Y]\leq1\}.
\]
Assuming that $X\in\mathcal{X}_{\tau}$, we have that $\Tr[X/\tau]\leq1$,
i.e., $\tau^{-1}X$ must lie in this domain for the von Neumann entropy. 

Now $S_{\mathrm{VN}}$ is $1$-strongly convex with respect to $\Vert\,\cdot\,\Vert_{*}$
\cite{10.1063/1.4871575}. It is equivalent (cf., e.g., Proposition
1 of \cite{doi:10.1137/16M1099546}) to say that the Hessian satisfies
\[
\left\langle Z,\nabla^{2}S_{\mathrm{VN}}(Y)\,[Z]\right\rangle \geq\Vert Z\Vert_{*}^{2}.
\]
 But 
\[
\nabla^{2}S_{\mathrm{FD}}(X)=\frac{1}{\tau}\nabla^{2}S_{\mathrm{VN}}(\tau^{-1}X)+\frac{1}{\tau}\underbrace{\nabla^{2}S_{\mathrm{VN}}(\tau^{-1}[\mathbf{I}_{n}-X])}_{\succeq0}.
\]
 We do not have a nonzero lower bound for the second term, but we
do not need it. (The zero lower bound simply follows from convexity.)

It follows that 
\[
\left\langle Z,\nabla^{2}S_{\mathrm{FD}}(X)\,[Z]\right\rangle \geq\frac{1}{\tau}\Vert Z\Vert_{*}^{2},
\]
 hence $S_{\mathrm{FD}}$ is $(1/\tau)$-strongly convex with respect
to $\Vert\,\cdot\,\Vert_{*}$.
\end{proof}
\begin{center}----------------------------------------------------------------------\end{center}
\begin{proof}
[Proof of Lemma \ref{lem:Gestbound-1}.] Following the proof of
Lemma \ref{lem:Gestbound}, we have that $[\hat{\rho}_{t}]_{q}$ is
sub-exponential with parameters $(2[\rho(X_{t})]_{q},4[\rho(X_{t})]_{q})$.
Hence by applying Corollary \ref{cor:subexp} with $\nu=2c_{\Psi}$,
we deduce that 
\[
[\hat{\rho}_{t}]_{q}\leq[\rho(X_{t})]_{q}\left(2+8\log(Tm/\delta)\right)
\]
 holds with probability at least $1-\frac{\delta}{Tm}$.

Then (applying absolute value bars entrywise) it follows that 
\[
\vert V\hat{\rho}_{t}\vert\leq\vert V\vert\,\hat{\rho}_{t}\leq\vert V\vert\,\rho(X_{t})\,\left(2+8\log(Tm/\delta)\right),
\]
 hence by the definiton of $\tilde{c}_{h}$ we have (\ref{eq:chtilde})
\[
\Vert V\hat{\rho}_{t}\Vert_{\infty}\leq\tilde{c}_{\mathrm{h}}\left(2+8\log(Tm/\delta)\right).
\]

Then it follows that 
\[
\Vert\hat{G}_{t}\Vert=\Vert\Psi^{\top}\mathrm{diag}^{*}\left[V\hat{\rho}_{t}\right]\Psi\Vert\leq\Vert V\hat{\rho}_{t}\Vert_{\infty}\leq2\left(1+4\log(Tm/\delta)\right)\tilde{c}_{\mathrm{h}}
\]
 holds for all $t=0,\ldots,T-1$ with probability at least $1-\delta$,
as was to be shown.
\end{proof}
\begin{center}----------------------------------------------------------------------\end{center}
\begin{proof}
[Proof of Lemma \ref{lem:mdsprep-1}.] . Let $Y=X_{\star}-X_{t}$,
and let $y=\left\langle \Delta_{t},Y\right\rangle $. As in the proof
of Lemma \ref{lem:mdsprep}, we have that $y$ is sub-exponential
with parameters $(2\Vert A\Vert_{\mathrm{F}},4\Vert A\Vert)$, conditioned
on $\mathcal{F}_{t-1}$, where $A:=X_{t}^{1/2}\nabla\tilde{E}(Y)X_{t}^{1/2}$.

We can split $A=A_{1}-A_{2}$, where $A_{1}=X_{t}^{1/2}\nabla\tilde{E}(X_{\star})X_{t}^{1/2}$
and $A_{2}=X_{t}^{1/2}\nabla\tilde{E}(X_{t})X_{t}^{1/2}$. Now 
\[
\Vert A_{1}\Vert=\Vert X_{t}^{1/2}\nabla\tilde{E}(X_{\star})X_{t}^{1/2}\Vert\leq\Vert X_{t}^{1/2}\Vert^{2}\,\Vert\nabla\tilde{E}(X_{\star})\Vert\leq\Vert\nabla\tilde{E}(X_{\star})\Vert,
\]
 where we have used the fact that $\Vert X_{t}\Vert\leq1$ in the
last inequality.

Now 
\begin{align*}
\Vert\nabla\tilde{E}(X_{\star})\Vert & =\Vert\Psi^{\top}\mathrm{diag}^{*}[V\rho(X_{\star})]\Psi\Vert\\
 & \leq\Vert V\rho(X_{\star})\Vert_{\infty}\\
 & \le\tilde{c}_{\mathrm{h}},
\end{align*}
 so $\Vert A_{1}\Vert\leq\tilde{c}_{h}$, and similar reasoning shows
$\Vert A_{2}\Vert\leq\tilde{c}_{h}$, hence 
\[
\Vert A\Vert\leq2\tilde{c}_{h}.
\]

Moreover, 
\begin{align*}
\Vert A_{1}\Vert_{\mathrm{F}}^{2} & =\Tr\left[\nabla\tilde{E}(X_{\star})X_{t}\nabla\tilde{E}(X_{\star})X_{t}\right]\\
 & =\left\langle \left(\nabla\tilde{E}(X_{\star})X_{t}\right)^{\top},\nabla\tilde{E}(X_{\star})X_{t}\right\rangle \\
 & \leq\Vert\nabla\tilde{E}(X_{\star})X_{t}\Vert_{\mathrm{F}}^{2}\\
 & \leq\Vert\nabla\tilde{E}(X_{\star})\Vert^{2}\Vert X_{t}\Vert_{\mathrm{F}}^{2}\\
 & \leq\tilde{c}_{\mathrm{h}}^{2}\Vert X_{t}\Vert_{\mathrm{F}}^{2},
\end{align*}
 so $\Vert A_{1}\Vert_{\mathrm{F}}\leq\tilde{c}_{h}\Vert X_{t}\Vert_{\mathrm{F}}$
and similar reasoning shows $\Vert A_{2}\Vert_{\mathrm{F}}\leq\tilde{c}_{h}\Vert X_{t}\Vert_{\mathrm{F}}$.
Therefore 
\[
\Vert A\Vert_{\mathrm{F}}\leq2\tilde{c}_{h}\Vert X_{t}\Vert_{\mathrm{F}},
\]
 which completes the proof.
\end{proof}
\begin{center}----------------------------------------------------------------------\end{center}
\begin{proof}
[Proof of Lemma \ref{lem:TrX0}.] Recall that 
\[
X_{0}=f_{\beta}(H_{0}),
\]
 where $H_{0}=C-\mu\mathbf{I}_{n}$, so 
\begin{align*}
\Tr[X_{0}] & =\sum_{k=1}^{n}f_{\beta}(\lambda_{k}(H_{0}))\\
 & =\sum_{k<c_{\lambda}\mu}f_{\beta}(\lambda_{k}(H_{0}))+\sum_{k\geq c_{\lambda}\mu}f_{\beta}(\lambda_{k}(H_{0}))\\
 & \leq c_{\lambda}\mu+\sum_{k\geq c_{\lambda}\mu}e^{-\beta\lambda_{k}(H_{0})}\\
 & \leq c_{\lambda}\mu+e^{\beta\mu}\sum_{k\geq c_{\lambda}\mu}e^{-\beta\lambda_{k}(C)},
\end{align*}
 where we have used the fact that $f_{\beta}(x)\leq\max(1,e^{-\beta x})$
for all $x$.

Then Assumption \ref{assump:growth} implies in turn that 
\begin{align*}
\Tr[X_{0}] & \leq c_{\lambda}\mu+e^{\beta\mu}\sum_{k\ge c_{\lambda}\mu}e^{-(\beta/c_{\lambda})k}\\
 & =c_{\lambda}\mu+\frac{1}{1-e^{-(\beta/c_{\lambda})}},
\end{align*}
 where we have used the geometric sum formula and the fact that the
first term in the geometric series is bounded above by $e^{-\beta\mu}$.

Then we can use the general inequality $\frac{1}{1-e^{-1/x}}\leq(1+x)$,
which holds for $x\geq0$, to deduce that 
\[
\Tr[X_{0}]\leq c_{\lambda}\mu+(1+\beta^{-1}c_{\lambda})=c_{\lambda}(\mu+\beta^{-1})+1.
\]
\end{proof}
\begin{center}----------------------------------------------------------------------\end{center}
\begin{proof}
[Proof of Lemma \ref{lem:DXstarX0}.] First expand
\[
D(X_{\star}\Vert X_{0})=S_{\mathrm{FD}}(X_{\star})-S_{\mathrm{FD}}(X_{0})+\beta\left\langle C-\mu\mathbf{I}_{n},X_{\star}-X_{0}\right\rangle .
\]
 Now $S_{\mathrm{FD}}\leq0$ always, so immediately we have 
\begin{equation}
D(X_{\star}\Vert X_{0})\leq-S_{\mathrm{FD}}(X_{0})+\beta\left\langle C,X_{\star}-X_{0}\right\rangle .\label{eq:Din1}
\end{equation}

Let us first bound the first term. Since $X_{0}=f_{\beta}(C-\mu\mathbf{I}_{n})$,
we can expand 
\[
-S_{\mathrm{FD}}(X_{0})=\sum_{k=1}^{n}g(\lambda_{k}(C)-\mu),
\]
 where 
\[
g(x):=-f_{\beta}(x)\log f_{\beta}(x)-[1-f_{\beta}(x)]\log[1-f_{\beta}(x)].
\]
 Now $g$ admits the elementary pointwise bound 
\[
g(x)\leq e^{-\beta\vert x\vert/3}
\]
 for all $x\in\R$, so 
\begin{align*}
-S_{\mathrm{FD}}(X_{0}) & \leq\sum_{k=1}^{n}e^{-\frac{\beta}{3}\vert\lambda_{k}(C)-\mu\vert}\\
 & =\sum_{k<c_{\lambda}\mu}e^{-\frac{\beta}{3}\vert\lambda_{k}(C)-\mu\vert}+\sum_{k\geq c_{\lambda}\mu}e^{-\frac{\beta}{3}\vert\lambda_{k}(C)-\mu\vert}\\
 & \leq c_{\lambda}\mu+e^{\beta\mu/3}\sum_{k\geq c_{\lambda}\mu}e^{-\frac{\beta}{3c_{\lambda}}k}
\end{align*}
 where we have used Assumption \ref{assump:growth} in the last line.
Then by summing the geometric series and copying the argument at the
end of the proof of Lemma \ref{lem:TrX0}, we obtain the bound 
\begin{align}
-S_{\mathrm{FD}}(X_{0}) & \leq c_{\lambda}\mu+\frac{1}{1-e^{-\frac{\beta}{3c_{\lambda}}}}\nonumber \\
 & \leq c_{\lambda}\mu+\left(1+3\beta^{-1}c_{\lambda}\right)\nonumber \\
 & =c_{\lambda}(\mu+3\beta^{-1})+1.\label{eq:Din2}
\end{align}

Now we turn to bounding the term $\left\langle C-\mu\mathbf{I}_{n},X_{\star}-X_{0}\right\rangle $
appearing on the right-hand side of (\ref{eq:Din1}). It is useful
to recall (\ref{eq:fixedpoint}), i.e., that the optimizer $X_{\star}$
satisfies 
\[
X_{\star}=f_{\beta}(C-\mu\mathbf{I}_{n}+V_{\star}),\quad\text{where }\ V_{\star}:=\nabla\tilde{E}(X_{\star}).
\]
 Then write 
\begin{align*}
\left\langle C,X_{\star}-X_{0}\right\rangle  & =\left\langle C,X_{\star}\right\rangle -\left\langle C,X_{0}\right\rangle \\
 & =\left\langle C-\mu\mathbf{I}_{n},X_{\star}\right\rangle -\left\langle C-\mu\mathbf{I}_{n},X_{0}\right\rangle +\mu\Tr\left[X_{\star}-X_{0}\right].\\
 & \leq\left\langle C-\mu\mathbf{I}_{n},X_{\star}\right\rangle -\left\langle C-\mu\mathbf{I}_{n},X_{0}\right\rangle ,
\end{align*}
 where in the last line we use the fact that $\Tr[X_{\star}]\leq\Tr[X_{0}]$,
which follows from the same argument demonstrating that $\Tr[X_{t}]\leq\Tr[X_{0}]$
in the proof of Theorem \ref{thm:convergence-1}.

Then we can further manipulate: 
\begin{align*}
\left\langle C,X_{\star}-X_{0}\right\rangle  & \leq\left\langle C-\mu\mathbf{I}_{n}+V_{\star},X_{\star}\right\rangle -\left\langle C-\mu\mathbf{I}_{n},X_{0}\right\rangle -\left\langle V_{\star},X_{\star}\right\rangle \\
 & \leq\left\langle C-\mu\mathbf{I}_{n}+V_{\star},X_{\star}\right\rangle -\left\langle C-\mu\mathbf{I}_{n},X_{0}\right\rangle ,
\end{align*}
 where in the last line we have used the fact that $V_{\star}\succeq0$
(following from Assumption \ref{assump:veenn}).

Then we conclude that 
\[
\left\langle C,X_{\star}-X_{0}\right\rangle \leq\Tr\left[h(H_{\star})-h(H_{0})\right],
\]
 where $H_{0}=C-\mu\mathbf{I}_{n}$, $H_{\star}=C-\mu\mathbf{I}_{n}+V_{\star}$,
and we define 
\[
h(x):=xf_{\beta}(x).
\]
 Then expand in terms of eigenvalues to obtain 
\begin{align}
\left\langle C,X_{\star}-X_{0}\right\rangle  & \leq\sum_{k=1}^{n}\left[h\left(\lambda_{k}(C+V_{\star})-\mu\right)-h\left(\lambda_{k}(C)-\mu\right)\right]\nonumber \\
 & =\sum_{k<c_{\lambda}\mu}\left[h\left(\lambda_{k}(C+V_{\star})-\mu\right)-h\left(\lambda_{k}(C)-\mu\right)\right]\nonumber \\
 & \quad\quad+\ \sum_{k\geq c_{\lambda}\mu}\left[h\left(\lambda_{k}(C+V_{\star})-\mu\right)-h\left(\lambda_{k}(C)-\mu\right)\right].\label{eq:2sums}
\end{align}
 Now it can be verified that $h$ is $2$-Lipschitz (independent of
$\beta$), so the first sum on the right-hand side of (\ref{eq:2sums})
is bounded as 
\begin{align}
 & \sum_{k<c_{\lambda}\mu}\left[h\left(\lambda_{k}(C+V_{\star})-\mu\right)-h\left(\lambda_{k}(C)-\mu\right)\right]\nonumber \\
 & \quad\quad\leq\ 2\sum_{k<c_{\lambda}\mu}\left|\lambda_{k}(C+V_{\star})-\lambda_{k}(C)\right|\nonumber \\
 & \quad\quad\leq\,2c_{\mathrm{h}}c_{\lambda}\mu,\label{eq:1stsum}
\end{align}
 where we have used Weyl's theorem and the fact that $\Vert V_{\star}\Vert\leq c_{\mathrm{h}}$.

Finally, we concentrate on the second sum on the right-hand side of
(\ref{eq:2sums}). Note that for $k\geq c_{\lambda}\mu$, Assumption
(\ref{assump:growth}) implies that $\lambda_{k}(C)-\mu\geq0$, hence
$h\left(\lambda_{k}(C)-\mu\right)\geq0$, hence we can simplify by
dropping the second summand: 
\begin{align*}
 & \sum_{k\geq c_{\lambda}\mu}\left[h\left(\lambda_{k}(C+V_{\star})-\mu\right)-h\left(\lambda_{k}(C)-\mu\right)\right]\\
 & \quad\quad\leq\ \sum_{k\geq c_{\lambda}\mu}h\left(\lambda_{k}(C+V_{\star})-\mu\right).
\end{align*}
Next observe the elementary inequality $\frac{x}{1+e^{x}}\leq e^{-x/2}$,
which holds for all $x\in\R$. If follows that 
\[
h(x)=\frac{x}{1+e^{\beta x}}\leq\beta^{-1}e^{-\beta x/2}
\]
 for all $x$. Thus 
\begin{align*}
\sum_{k\geq c_{\lambda}\mu}h\left(\lambda_{k}(C+V_{\star})-\mu\right) & \leq\beta^{-1}\sum_{k\geq c_{\lambda}\mu}e^{-\beta[\lambda_{k}(C+V_{\star})-\mu]}\\
 & \leq\beta^{-1}\sum_{k\geq c_{\lambda}\mu}e^{-\beta[\lambda_{k}(C)-\mu]}\\
 & \leq\beta^{-1}e^{\beta\mu}\sum_{k\geq c_{\lambda}\mu}e^{-(\beta/c_{\lambda})k}
\end{align*}
 where in the penultimate line we have used the fact that $\lambda_{k}(C+V_{\star})\geq\lambda_{k}(C)$,
which follows from Weyl's monotonicity theorem. Then by the same geometric
series argument as the one at the end of the proof of Lemma \ref{lem:TrX0},
it follows that the second sum on the right-hand side of (\ref{eq:2sums})
is bounded by 
\[
\beta^{-1}\left(1+\beta^{-1}c_{\lambda}\right).
\]

Combining this fact with the bound (\ref{eq:1stsum}) for the first
sum on the right-hand side of (\ref{eq:2sums}) and substituting into
(\ref{eq:2sums}), we conclude that 
\[
\left\langle C,X_{\star}-X_{0}\right\rangle \leq2c_{\mathrm{h}}c_{\lambda}\mu+\beta^{-1}\left(1+\beta^{-1}c_{\lambda}\right).
\]

Therefore, combining with (\ref{eq:Din1}) and (\ref{eq:Din2}), we
obtain 
\[
D(X_{\star}\Vert X_{0})\leq c_{\lambda}(\mu+2\beta c_{\mathrm{h}}\mu+4\beta^{-1})+2.
\]
\end{proof}

\end{document}